\documentclass{amsart}

\usepackage{graphicx} 
\usepackage{amssymb}
\usepackage[hyphens]{url} 
\usepackage{amsmath}
\usepackage{cite}
\usepackage{mathrsfs}
\usepackage{hyperref}
\usepackage{amsthm}
\usepackage{tikz-cd}
\usepackage{caption}
\usepackage{float}
\usepackage[utf8]{inputenc}
\usepackage{graphicx}
\usepackage{stmaryrd}
\usepackage{tikz}
\usepackage{tikz-cd} 
\usepackage[all]{xy}
\usepackage{comment}
\usepackage{bm}
\usepackage{comment}
\usepackage{amsmath,amsfonts,amssymb}
\usepackage{tcolorbox}
\usepackage{cleveref}
\usepackage[a4paper]{geometry}
\usepackage{nameref}
\usepackage{mathabx}
\usepackage{enumitem}
\usepackage{xcolor}
\usepackage{nicefrac}
\hypersetup{
	colorlinks=true,
	linkcolor={blue},
	citecolor={blue},
	urlcolor={blue}}

\theoremstyle{plain}
\newtheorem{mainthm}{Theorem}

\newtheorem{maincor}[mainthm]{Corollary}

\newtheorem{thm}{Theorem}[section]
\newtheorem{prop}[thm]{Proposition}

\newtheorem{cor}[thm]{Corollary}

\newtheorem*{conj*}{Conjecture}

\newtheorem{lemma}[thm]{Lemma}
\newtheorem{lem}[thm]{Lemma}
\newtheorem*{cor*}{Corollary}

\theoremstyle{definition}
\newtheorem{notation}[thm]{Notation}
\newtheorem{defn}[thm]{Definition}

\newtheorem{conv}[thm]{Convention}

\theoremstyle{remark}
\newtheorem{rmk}[thm]{Remark}
\newtheorem{example}[thm]{Example}

\newtheorem{stassumption}[thm]{Standing Assumption}
\newtheorem*{rmk*}{Remark}

\newlength{\plarg}
\usetikzlibrary{cd}

\newcommand{\gen}[1]{\langle #1 \rangle}
\newcommand{\pres}[2]{\langle #1 | #2 \rangle}

\newcommand{\bbZ}{\mathbb{Z}}
\newcommand{\ab}{\mathrm{ab}}

\newcommand{\pcl}[2]{\mathrm{pcl}_{#1}(#2)}

\newcommand{\lepair}{\le}
\newcommand{\embedpair}{\hookrightarrow}

\numberwithin{equation}{section}

\title{Virtual homological torsion in graphs of free groups with cyclic edge groups}
\author{Dario Ascari and Jonathan Fruchter}

\begin{document}

\begin{abstract}
    Let $G$ be a hyperbolic group that splits as a graph of free groups with cyclic edge groups. We prove that, unless $G$ is isomorphic to a free product of free and surface groups, every finite abelian group $M$ appears as a direct summand in the abelianization of some finite-index subgroup $G'\le G$. As an application, we deduce that free products of free and surface groups are profinitely rigid among hyperbolic graphs of free groups with cyclic edge groups. We also conclude that \emph{partial surface words} in a free group are determined by the word measures they induce on finite groups.
\end{abstract}

\maketitle

\tableofcontents

\section{Introduction} \label{intro}

This paper concerns virtual torsion in the abelianization of a large class of hyperbolic groups: those that split as graphs of free groups amalgamated along cyclic subgroups. This class plays a central role in geometric group theory, frequently serving as a testbed for new ideas and techniques. Notable examples include subgroup separability (Wise, \cite{wise:graph-sep}), cubulability and specialness (Hsu and Wise, \cite{wise:cube}), coherence and local quasiconvexity (Bigdely and Wise, \cite{wise:local}), and surface subgroups (Wilton \cite{Wil18}). \par \smallskip

These results can be seen as instances of group-theoretic analogues of major results in low-dimensional topology, especially within the study of (hyperbolic) $3$-manifolds. Coherence and separability echo Scott’s core theorem \cite{Scott1973}; cubulability and specialness parallel Agol and Wise’s work on the virtual Haken conjecture \cite{Wise2021, agol}; and surface subgroups reflect the results of Kahn and Markovic \cite{Kahn2012}. It is worth noting, however, that the proofs in the graphs-of-free-groups setting are often of a very different nature. They rely on combinatorial-geometric arguments, emphasize the local-global interplay between the entire graph of groups and its vertex groups, and are grounded in \emph{very} low-dimensional topology—namely, the topology of graphs. \par \smallskip

One of the greatest open challenges in the study of $3$-manifolds is the \emph{exponential torsion growth conjecture} \cite{Bergeron2012, Lck2002, Lck2013}, which predicts that, akin to L\"uck's approximation theorem for $L^2$-Betti numbers \cite[Theorem 0.1]{Lck1994}, the $L^2$-torsion of every irreducible $3$-manifold $M$ with empty or toroidal boundary can be approximated by the homological torsion in a cofinal tower of regular finite-sheeted covers. More explicitly, given such a $3$-manifold $M$, does there exist a cofinal tower of regular finite-sheeted covers 
\[\cdots \twoheadrightarrow M_n \twoheadrightarrow \cdots \twoheadrightarrow M_1 \twoheadrightarrow M_0 = M\]
for which
\[
\lim _n \frac{\ln \vert \mathrm{Tor}(H_1(M_n;\mathbb{Z}))\vert}{[\pi_1(M):\pi_1(M_n)]}=-\rho^{(2)}(\tilde{M}) = \frac{\mathrm{vol}(M)}{6\pi}\;?
\]
In fact, to the best of our knowledge, it is still an open question whether there exists a finitely presented group $G$, and an exhausting, nested sequence of finite-index normal subgroups $G_n$ of $G$, such that
\[
\lim_n \frac{\ln \vert \mathrm{Tor}(G_n^{\ab})\vert}{[G:G_n]}>0.
\]
Many efforts have been made towards answering these questions. One positive result in this direction is due to Sun \cite[Theorem 1.5]{Sun2015} (later generalized by Chu and Groves \cite[Theorem 1.1]{Chu2022} to a broader class of $3$-manifolds), who showed that for every closed, hyperbolic $3$-manifold $M$, every finite abelian group appears as a direct summand in the first homology of a finite-sheeted cover of $M$.\par \smallskip

We prove that a similar phenomenon occurs for hyperbolic groups that split as graphs of free groups with cyclic edge groups. Note that not all such groups virtually contain torsion in homology: the abelianizations of free groups, surface groups (fundamental groups of closed surfaces of non-positive Euler characteristic), and their free products—and all their finite-index subgroups—are always torsion-free. However, this is the only obstruction; beyond this case, not only does torsion appear, it does so in abundance:

\begin{mainthm} [\Cref{thm:main}] \label{mainthm}
    Let $G$ be a hyperbolic group that splits as a graph of free groups amalgamated along cyclic subgroups, and that is not isomorphic to a free product of free and surface groups. Then for every finite abelian group $M$, there exists a finite-index subgroup $H\le G$ such that $M$ is a direct summand of the abelianization $H^{{\ab}}$ of $H$.
\end{mainthm}

In fact, \Cref{mainthm} holds in the slightly more general setting of hyperbolic graphs of virtually free groups amalgamated along $2$-ended subgroups:

\begin{maincor} [\Cref{cor:virtual}] \label{virtualcor}
    If $G$ is a hyperbolic group that splits as a graph of virtually free groups with virtually cyclic edges, and which is not virtually a free product of free and surface groups, then every finite abelian group appears as a direct summand in the abelianization of a finite-index subgroup of $G$.
\end{maincor}

\subsection*{Applications to profinite rigidity}
Before turning to discuss the ideas behind \Cref{mainthm}, we highlight the applications of our main result to the field of \emph{profinite rigidity}. Recall that the \emph{profinite completion} $\widehat{G}$ of a group $G$ is the closure of the image of $G$ under the diagonal map $g\mapsto (gN)_{N \triangleleft_{\mathrm{f.i}}G}$ from $G$ into the direct product $\prod_{N \triangleleft_{\mathrm{f.i}}G} G/N$ of all finite quotients of $G$ (endowed with the product topology). A finitely generated, residually finite group $G$ is called \emph{(absolutely) profinitely rigid} if, up to isomorphism, it is the only finitely generated, residually finite group whose profinite completion is isomorphic to $\widehat{G}$. The question of absolute profinite rigidity is notoriously difficult, and progress is often made by restricting the question to a class of groups. For instance, the question of whether finitely generated non-abelian free groups are absolutely profinitely rigid (attributed to Remeslennikov, \cite[Question 15]{Noskov1982}) remains unsolved, but Wilton showed that free groups (and surface groups) are profinitely rigid within the class of hyperbolic graphs of free groups with cyclic edge groups (and within the closely related class of finitely generated, residually free groups; see \cite[Corollary D]{Wil18} and \cite[Theorem 2]{wilton:words}). \par \smallskip

The abelianization—and more generally, the collection of abelianizations of all finite-index subgroups—of a finitely generated, residually finite group $G$ is a well-known profinite invariant. It follows almost immediately from \Cref{mainthm} that free products of free and surface groups are profinitely rigid among hyperbolic graphs of free groups with cyclic edges:

\begin{mainthm}[\Cref{thm:pf-rigid}] \label{maincorprof}
    Let $G$ be a free product of finitely many free and surface groups, and let $H$ be a hyperbolic group that splits as a graph of free groups with cyclic edge groups. If $\widehat{G} \cong \widehat{H}$, then $G\cong H$.
\end{mainthm}

\Cref{mainthm} also has implications on the \emph{profinite rigidity of words} in free groups. Every word $w$ in a free group $F_k=\gen{x_1,\ldots,x_k}$ gives rise to a \emph{word map} $w:G^k\rightarrow G$ for every group $G$. If $G$ is finite, every such word map induces a probability measure on $G$, by pushing forward the uniform measure on $G^k$ via the word map $w$. A long-standing conjecture (see \cite[Question 2.2]{amit:measure_conj}, \cite[Conjecture 4.2]{shalev:measure_conj} and \cite[Section 8]{puder:primitive}) posits that whenever two words $w,w'\in F_k$ induce the same measure on every finite group, there exists an automorphism $f\in \mathrm{Aut}(F_k)$ such that $f(w)=w'$. Hanany, Meiri and Puder gave an equivalent formulation of this conjecture \cite[Theorem 2.2]{puder:commutator} in terms of automorphic orbits in the profinite completion $\widehat{F}_k$ of $F_k$:
    
\begin{conj*} \label{conj:words}
    Let $w,w'\in F_k$. If there exists $\widehat{f} \in \mathrm{Aut}(\widehat{F}_k)$ such that $\widehat{f}(w)=w'$, then there exists $f\in \mathrm{Aut}(F_k)$ with $f(w)=w'$.
\end{conj*}

A word $w\in F_k$ is \emph{profinitely rigid} (in $F_k$) if 
    \[\mathrm{Aut}(\widehat{F}_k).w \cap F = \mathrm{Aut}(F_k).w,\]
that is, the profinite orbit of $w$ in $\widehat{F}_k$ (under $\mathrm{Aut}(\widehat{F}_k)$) intersects $F_k$ precisely in the ordinary automorphic orbit of $w$ in $F_k$. As with the question of (absolute) profinite rigidity, this conjecture seems extremely difficult. In fact, it is often viewed as a relative version of Remeslennikov's question mentioned above. At present, the conjecture has been verified for three (or four, when $k$ is even) distinct automorphic orbits in $F_k$ (of words which are not proper powers, to which the question is reduced by \cite[Theorem 1.7]{puder:commutator}). These results have been proven using a variety of deep methods, ranging from probabilistic and geometric to algebraic techniques. The known cases include:
\begin{enumerate}
    \item \emph{primitive words}, that is, words in the orbit of $x_1$ (Parzanchevsky and Puder \cite[Theorem 1.1]{puder:primitive}, and later Wilton \cite[Corollary E]{Wil18} and Garrido and Jaikin-Zapirain \cite[Theorem 1.1]{garrido:ff}),
    \item \emph{commutators of basis elements}, that is, words in the orbit of $[x_1,x_2]$ (Hanany, Meiri and Puder \cite[Theorem 1.4]{puder:commutator}), and
    \item \emph{surface words}, that is, words in the orbits of $x_1^2\ldots x_k^2$ and (when $k$ is even) $[x_1,x_2]\ldots [x_{k-1},x_k]$ (Wilton \cite[Corollary 4]{wilton:words}). 
\end{enumerate} \par

\emph{Partial surface words} in $F_k$ are words that bridge the gaps between (1), (2) and (3) above, i.e., words of the form
\begin{enumerate}
    \item $[x_1,x_2]\cdots [x_{2n-1},x_{2n}]$ for $2n<k$ (\emph{orientable} partial surface words), and
    \item $x_1^2\cdots x_n^2$ for $n<k$ (\emph{non-orientable} partial surface words).
\end{enumerate}
Magee and Puder showed that partial surface words possess the weaker property of being determined by the word measures that they induce on compact groups \cite[Theorem 1.4]{Magee2021}. They are the only examples of words known to have this property, but are not known to be profinitely rigid (a distinction that is conjecturally vacuous). An immediate corollary of \Cref{maincorprof} is that partial surface words in $F_k$ are profinitely rigid:

\begin{maincor}[\Cref{cor:partial_words}]\label{maincorwords}
    Partial surface words in a free group $F_k$ are profinitely rigid, that is, they are determined by the measures that they induce on finite groups.
\end{maincor}

\subsection*{On the proof of \Cref{mainthm}}

The boundary loops of a compact, orientable surface $\Sigma$ of positive genus sum to $0$ in $H_1(\Sigma)$; this relationship between a surface and its boundary, as one would expect, plays a key role in our proof. Sun’s construction (and that of Chu and Groves) of finite-sheeted covers of a closed, hyperbolic $3$-manifold $M$ exhibiting homological torsion, relies on immersing certain $2$-complexes $X_p$ into $M$: surfaces whose boundary components trace $p$-powers of loops. These loops contribute $p$-torsion to the first homology group of $X_p$, which is then promoted to homological torsion in a finite cover of $M$ via a virtual retraction. Our proof also makes use of virtual retractions, but the gadget we employ to produce homological torsion is of a different kind: a \emph{branched surface}. In our context, a branched surface refers to a collection of compact surfaces with boundary, glued along their boundary components; we refer the reader to \Cref{sec:branched} for the precise definition. \par \smallskip

Each compact, orientable surface $\Sigma$ of positive genus and with boundary components $\sigma_1,\ldots,\sigma_k$ (each equipped with an orientation) gives rise to an equation of the form
\[
\pm[\sigma_1]\pm[\sigma_2]\pm \cdots \pm [\sigma_k]=0
\]
in $\pi_1(\Sigma)^{\mathrm{\ab}}$. We refer to equations of the form $\pm x_1 \pm x_2 \pm \cdots \pm x_k=0$ as \emph{$\partial$-equations}. A branched surface built from such orientable, positive genus surfaces with non-empty boundary can thus be viewed as a geometric realization of a \emph{system} of $\partial$-equations. Indeed, if $B$ is such a branched surface, $H_1(B)$ is a quotient of a free abelian group by a system of $\partial$-equations, and conversely, for every system $\Psi$ of $\partial$-equations there are (many) branched surfaces whose first homology group is obtained by imposing the linear constraints in $\Psi$ on a free abelian group. The first step in our proof shows that every system of linear equations in a free abelian group is \emph{equivalent} to a system of $\partial$-equations; here, ``equivalent'' means that both systems yield isomorphic quotient groups. We also show that for a given branched surface $B$ (which is not a surface), any system of $\partial$-equations can be realized by a branched surface that immerses into $B$. This proves \Cref{mainthm} for branched surfaces (see \Cref{prop:branched-torsion} and the accompanying \Cref{fig:example_torsion}). \par \smallskip

In Wilton's proof that surfaces immerse into a non-free hyperbolic graph of free groups with cyclic edges $G$ \cite[Theorem A]{Wil18}, the non-freeness assumption is used as a lower bound on the complexity of vertex links in a graph of spaces decomposition of $G$ (i.e. Whitehead graphs). Replacing non-freeness with the stronger assumption that $G$ is not a free product of free and surface groups imposes a sharper lower bound on the complexity of $G$. One might therefore expect to be able to immerse more intricate 2-complexes, such as branched surfaces, into $G$; however, it remains unclear whether this is always the case. Nevertheless, we show that fundamental groups of branched surfaces admit maps to $G$ that are “injective enough” for our purposes: they induce injections on the torsion part of the abelianization. \par \smallskip

The assumption that $G$ is not a free product of free and surface groups leads (after possibly passing to a finite cover) to one of the following scenarios: \begin{enumerate} 
\item The underlying graph of the graph of groups decomposition of $G$ exhibits genuine branching: \emph{at least} three cyclic subgroups are identified in the graph of groups to form a single copy of $\mathbb{Z}$. More specifically, a \emph{tree of cylinders} of a JSJ decomposition of (a factor in the Grushko decomposition of) $G$ admits a cyclic vertex of valence $\ge 3$. In this case, one can apply Wilton’s \cite[Theorem F]{Wil18}, replacing each non-cyclic vertex group with a surface whose boundary is compatible with the incident edge groups, thereby producing a branched surface that immerses into $G$. 
\item At least one vertex group of $G$ is sufficiently complicated—that is, a JSJ decomposition of (a factor in the Grushko decomposition of) $G$ contains a \emph{rigid} vertex group $G_v$. \end{enumerate}

The latter case, which occurs, for instance, when $G$ is the double of a free group along a complicated word, is considerably more involved. We give a brief overview of our strategy below, and refer the curious reader to the start of \Cref{sec:main}, where our strategy is outlined in greater detail. In this case, we use additional results of Wilton on the structure of rigid vertex groups of $G$ \cite[Theorem 8]{wil:one-ended}, which imply the existence of multiple surfaces immersing into $G_v$ with overlapping boundaries (that cover edge groups adjacent to $G_v$). This \emph{artificial branching} behaviour lies at the heart of our construction, and allows us to construct maps from branched surfaces to $G$. Controlling how these surfaces intersect, and particularly, how they interact in the abelianization, is subtle, but crucial. Extremal surfaces arising from Calegari’s Rationality Theorem \cite[Theorem 4.24]{Cal09} equip us with just the right amount of control, ultimately producing maps from branched surfaces to $G$ that induce injections on the torsion part of the abelianization.

\subsection*{Organization of the paper} \Cref{prelims} is a preliminary section, and contains basic results about graphs of groups (and spaces). In \Cref{pairs} we recall the theory of \emph{group pairs}, which describe the relation between a vertex group of a graph of groups and its incident edge groups, and prove auxiliary lemmas that will be used throughout our constructions. \Cref{sec:branched} is devoted to branched surfaces and their correspondence with systems of $\partial$-equations; in this section we prove \Cref{mainthm} for branched surfaces. In \Cref{sec:main} we construct artificial branching, and prove \Cref{mainthm}. Finally, in \Cref{sec:profinite} we apply \Cref{mainthm} to the study of profinite rigidity. 

\subsection*{Acknowledgments} Ascari was funded by the Basque Government grant IT1483-22. Fruchter received funding from the European Union (ERC, SATURN, 101076148) and the Deutsche Forschungsgemeinschaft (DFG, German Research Foundation) under Germany's Excellence Strategy - EXC-2047/1 - 390685813. We are grateful to Damien Gaboriau for his insightful comments, and to Giles Gardam for assistance with the computer experiments.

\section{Preliminaries} \label{prelims}

In this section, we lay the groundwork by introducing the necessary definitions and collecting a few useful lemmas that will be used throughout the proofs. Our focus is on graphs of groups (and spaces) and their covering spaces, with particular attention to covering space techniques that we will use to construct subgroups with prescribed torsion in the abelianization. 

\subsection{Graphs of groups (and spaces)}

Following Serre's conventions, a \emph{graph} is a quintuple $\Gamma=(\mathrm{V}(\Gamma),\mathrm{E}(\Gamma),\overline{\cdot},\iota,\tau)$; here $\mathrm{V}(\Gamma)$ and $\mathrm{E}(\Gamma)$ are the sets of vertices and edges of the graph respectively. $\overline{\cdot}:\mathrm{E}(\Gamma)\rightarrow \mathrm{E}(\Gamma)$ is the edge-reversing map, satisfying $\overline{\overline{e}}=e$ and $\overline{e}\not=e$ for all edges $e\in \mathrm{E}$. $\iota,\tau:\mathrm{E}(\Gamma)\rightarrow \mathrm{V}(\Gamma)$ are the initial and terminal vertex maps (respectively), mapping each $e\in \mathrm{E}(\Gamma)$ to its endpoints and satisfying $\tau(\overline{e})=\iota(e)$ for all edges $e$.

\begin{defn}
    A \emph{graph of groups} is a quadruple $\mathcal{G}=(\Gamma,\{G_v\}_{v\in \mathrm{V}(\Gamma)},\{G_e\}_{e_\in \mathrm{E}(\Gamma)},\{\varphi_e\}_{e\in \mathrm{E}(\Gamma)})$, where 
    \begin{enumerate}
        \item $\Gamma$ is a connected graph,
        \item $G_v$ are groups associated to the different vertices $v\in \mathrm{V}(\Gamma)$,
        \item $G_e$ are groups associated to the edges $e\in \mathrm{E}(\Gamma)$, and such that $G_{\overline{e}}=G_e$, and
        \item $\varphi_e:G_e\rightarrow G_{\tau(e)}$ are \textit{injective} homomorphisms for every $e\in \mathrm{E}(\Gamma)$
    \end{enumerate}
\end{defn}

Before defining \emph{graphs of spaces}, we should point out that whenever we consider a continuous map $f:Y\rightarrow X$ between two CW complexes, we always implicitly assume that $f(Y^{(k)})\subseteq X^{(k)}$, i.e., that the map sends the $k$-skeleton of $Y$ into the $k$-skeleton of $X$. This can always be obtained up to homotopy.

\begin{defn}
    A \textbf{graph of spaces} is a quadruple $\mathcal{X}=(\Gamma,\{X_v\}_{v\in V(\Gamma)},\{X_e\}_{e_\in E(\Gamma)},\{i_e\}_{e\in E(\Gamma)})$, where $\Gamma$ is a connected graph, $X_v$ are connected CW complexes associated to the vertices of $\Gamma$, $X_e$ are connected CW complexes associated to the edges of $\Gamma$ satisfying $X_{\overline{e}}=X_e$, and for every $e\in \mathrm{E}(\Gamma)$, $i_e:X_e\rightarrow X_{\tau(e)}$ is a $\pi_1$-injective continuous maps.
\end{defn}

Let $\mathcal{X}$ be a graph of spaces; the \textbf{geometric realization} $X$ of $\mathcal{X}$ is the topological space obtained by taking the disjoint union
\[\bigsqcup_{v\in V(\Gamma)}X_v\sqcup\bigsqcup_{e\in E(\Gamma)}X_e\times[0,1],\] 
and, for every edge $e\in E(\Gamma)$, identifying
\begin{enumerate}
    \item the subspace $X_e\times\{0\}$ with $X_{\overline{e}}\times\{0\}$ (using the fact that $X_{\overline{e}}=X_e$), and
    \item the subspace $X_e\times\{1\}$ to $X_{\tau(e)}$ using the continuous map $i_e$.
\end{enumerate}
Since the maps $i_e$ are assumed to be $\pi_1$-injective, one can prove that the ``inclusion'' maps $X_v\rightarrow X$ and $X_e=X_e\times\{0\}\rightarrow X$ are also $\pi_1$-injective. \par \smallskip

Given a graph of groups $\mathcal{G}=(\Gamma,\{G_v\}_{v\in \mathrm{V}(\Gamma)},\{G_e\}_{e_\in \mathrm{E}(\Gamma)},\{\varphi_e\}_{e\in \mathrm{E}(\Gamma)})$, we can always choose a collection of connected CW complexes $X_v,X_e$ with fundamental groups isomorphic to $G_v,G_e$, respectively, along with continuous maps $i_e:X_e\rightarrow X_{\tau(e)}$ inducing the homomorphisms $\varphi_e:G_e\rightarrow G_{\tau(e)}$. This construction yields a graph of spaces 
\[\mathcal{X}=(\Gamma,\{X_v\}_{v\in \mathrm{V}(\Gamma)},\{X_e\}_{e_\in \mathrm{E}(\Gamma)},\{i_e\}_{e\in \mathrm{E}(\Gamma)}).\]
Different choices of $X_v,X_e$ and $i_e$ may produce different graphs of spaces $\mathcal{X}$, but the fundamental groups of their geometric realizations $X$ will always be isomorphic. Therefore, we define 
\[\pi_1(\mathcal{G}):=\pi_1(X),\] 
where $X$ is the geometric realization of any graph of spaces constructed from $\mathcal{G}$ as described above.

\begin{rmk}
    In particular, if we take a graph of groups $\mathcal{G}$, and alter the homomorphism $\varphi_e:G_e\rightarrow G_{\tau(e)}$ by post-composing with conjugation by an element of $G_{\tau(e)}$, the fundamental group $\pi_1(\mathcal{G})$ remains unaffected. Therefore, the resulting graph of groups can be regarded as essentially the same.
\end{rmk}

Throughout this paper, we will focus our attention on graphs of groups whose edge groups are isomorphic to $\mathbb{Z}$. Thus, in our associated graphs of spaces, the edge spaces $X_e$ will always be homeomorphic to $S^1$.

\subsection{Covers, precovers and (geometric) elevations}
\label{subsec:precovers}

Let $\mathcal{X}$ be a graph of spaces and let $X$ be its geometric realization. Every covering space $p:Y\twoheadrightarrow X$ inherits a graph of spaces decomposition $\mathcal{Y}$ from $\mathcal{X}$: the vertex and edge spaces of $\mathcal{Y}$ are the connected components of the preimages under $p$ of the vertex and edge spaces of $X$. Moreover, $p$ induces a \emph{combinatorial} map between the underlying graphs $\Gamma$ and $\Delta$ of $\mathcal{X}$ and $\mathcal{Y}$ respectively (mapping vertices to vertices, and edges to edges). It remains to describe the edge maps of $\mathcal{Y}$, which is the crucial step in understanding the structure of $\mathcal{Y}$. This is captured by the notion of \emph{elevations}, introduced by Wise \cite[Definition 2.18]{wise:graph-sep}. While elevations can be defined for arbitrary edge spaces, the theory simplifies significantly when the edge spaces are circles. For this reason, we will restrict our definition to the case where the edge spaces are $S^1$. \par \smallskip

Let $\alpha:S^1\rightarrow X$ be a continuous map between connected CW complexes, and let $p:Y\rightarrow X$ be a covering map. Consider the pullback $Q=S_1\times _X Y$ of the two maps $\alpha$ and $p$. $Q$ comes equipped with continuous maps $\beta:Q\rightarrow Y$ and $q:Q\rightarrow S^1$. The map $q:Q\rightarrow S^1$ is a cover, and its degree is the same as that of $p:Y\rightarrow X$; however, in general, $Q$ may be disconnected.

\begin{defn}
    \label{def:geom-elev}
    A \emph{(geometric) elevation} of $\alpha:S^1\rightarrow X$ to $Y$ is the restriction $\beta\vert_{Q_i}:Q_i\rightarrow Y$ of $\beta$ to a connected component $Q_i$ of $Q$.

\begin{center}
\begin{tikzpicture}
    \node (A) at (0,2) {$Q=\bigsqcup_{i} Q_i$};
    \node (B) at (4,2) {$Y$};
    \node (C) at (0,0) {$S^1$};
    \node (D) at (4,0) {$X$};

    \draw[->] (A) -- node[above] {$\beta = \bigsqcup_{i} \beta \vert_{Q_i}$} (B);
    \draw[->] (B) -- node[right] {$p$} (D);
    \draw[->] (A) -- node[left] {$q$} (C);
    \draw[->] (C) -- node[below] {$\alpha$} (D);
\end{tikzpicture}
\end{center}
\end{defn}

\begin{rmk}
    Relating to our previous discussion, the edge maps of $\mathcal{Y}$ are precisely the (geometric) elevations of the edge maps of $\mathcal{X}$ to $Y$.
\end{rmk}

The degree of the covering map $q:Q'\rightarrow S^1$ is called \emph{degree} of the elevation. The degree of a geometric elevation can be finite (in which case $Q'\cong S^1$) or infinite (in which case $Q'\cong\mathbb{R}$, the universal cover of $S^1$).

\begin{example}\label{ex:geometric-elevation}
    The purpose of this example is to describe elevations in a hands-on way. Let $X$ be a connected graph with a single vertex and finitely many edges, so that $\pi_1(X)$ is a finitely generated free group. A map $\alpha:S^1\rightarrow X$ corresponds to the homotopy class of a loop in $X$, representing the conjugacy class of an element of $\pi_1(X)$. Let $p:Y\rightarrow X$ be a cover of $X$; note that $Y$ is a graph. We remark that $Y$ may be infinite, in which case we encourage the reader to only consider the \emph{core} $C_Y$ of $Y$, which is a finite subgraph of $Y$ satisfying $\pi_1(C_Y) = \pi_1(Y)$ whenever $\pi_1(Y)$ is finitely generated. Note that if $\alpha$ is reduced, that is, $\alpha$ contains no back-tracks, then the image of every finite-degree elevation of $\alpha$ will be contained in $C_Y$ (and the elevations of $\alpha$ to $C_Y$ coincide with the finite-degree elevations of $\alpha$ to $Y$). The elevations of $\alpha$ to $Y$ can be constructed as follows: \par \smallskip

    Start at a vertex $v_0$ of $Y$ and lift $\alpha$ to a path in $Y$, starting at $v_0$. The terminal point $v_1$ of this path might be different from $v_1$, in which case, lift $\alpha$ to a path in $Y$ starting at $v_1$. Repeat this process, and obtain a sequence of vertices $v_0,v_1,v_2,\dots$ of $Y$. \par \smallskip
    
    Let $d\ge1$ be the smallest such that $v_d=v_0$; we have a degree-$d$ elevation $\beta:S^1\rightarrow Y$ based at $v_0$, which traverses $\alpha$ exactly $d$ times. We emphasize that $d$ should be chosen the smallest possible. Indeed, a map $\beta$ can be defined using any multiple $kd$ of $d$, but such a map will not form a part of the pullback of $\alpha$ along $p$. Moreover, taking $v_i$ instead of $v_0$ as a starting point would produce the exact same elevation.\par \smallskip
    
    If no such $d\ge1$ exists, the resulting line obtained by concatenating lifts of $\alpha$ to $Y$ forms an infinite degree elevation $\beta:\mathbb{R}\rightarrow Y$ (corresponding to the universal cover of $S^1$).
\end{example}

Covering spaces of graphs of spaces can be built ``piece-by-piece'', as we will do later on. It is most convenient to work with compact spaces, which correspond to finite covers and therefore to finite-index subgroups. However, we would like to have compact spaces that represent infinite-index subgroups, for which we recall the notion of \emph{precovers}; these can be thought of as particularly nice compact cores of certain infinite covers of $X$. To define a precover of $\mathcal{X}$, we first gather the following data:

\begin{enumerate}
    \item A collection of covering spaces $\{p_u:Y_u \twoheadrightarrow X_v\}$ of (some of) the vertex spaces $X_v$ of $\mathcal{X}$.
    \item For every edge $e\in \mathrm{E}(\Gamma)$, a (possibly empty) collection $\{j_f:Y_f \rightarrow Y_u\}$ of elevations of $i_e$ to $\bigsqcup_u Y_u$; we denote the covering map from $Y_f$ to $X_e$ by $\theta_f$.
    \item For every edge $e\in \mathrm{E}(\Gamma)$, a degree-preserving bijection $b_e$ between the chosen collections of elevations of $i_e$ and $i_{\overline{e}}$ to $\bigsqcup_u Y_u$. These maps dictate how to assemble $\bigsqcup_u Y_u$ into a graph of spaces.
\end{enumerate}

This data naturally gives rise to a graph $\Delta$. The vertex set of $\Delta$ consists of the chosen covering spaces of the vertex spaces $X_v$ of $\mathcal{X}$. The edges of $\Delta$ are given by the maps $b_e$: for every pair of elevations $j_f:Y_f \rightarrow Y_u$ and $j_{f'}:Y_{f'}\rightarrow Y_{u'}$ such that $b_e(j_f)=j_{f'}$, there is an edge $f=(Y_{u'},Y_u)\in \mathrm{E}(\Delta)$ (and $\overline{f} = (Y_u, Y_{u'})$. In addition, the construction gives rise to a natural combinatorial map $\theta:\Delta \rightarrow \Gamma$. Finally, the quadruple 
\[\mathcal{Y}=(\Delta,\{Y_u\},\{Y_f\},\{j_f\}),\] 
accompanied by the collection of maps 
\[\Theta=(\theta,\{\theta_u\},\{\theta_f\}):\mathcal{Y}\rightarrow\mathcal{X},\] 
is called a \emph{precover} of $\mathcal{X}$. The elevations of edge maps of $\mathcal{X}$ to the vertex spaces of $\mathcal{Y}$, which are not edge maps of $\mathcal{Y}$, are called \emph{hanging elevations}. The significance of hanging elevations lies in their potential for subsequent gluing, allowing us to extend a precover to a larger precover (or a cover) of $\mathcal{X}$ if desired. For a more general treatment of precovers (where the edge spaces are not assumed to be circles), we refer the reader to Wilton's works \cite[Section 2]{wil:sep} and \cite[Section 1]{wil:one-ended}. \par \smallskip

There is a natural map between the geometric realizations $X$ and $Y$ of $\mathcal{X}$ and $\mathcal{Y}$, which we also denote by $\theta$. Assuming that $Y$ is connected, we thus obtain a subgroup $\theta_*(\pi_1(Y))=H$ of $\pi_1(X)=G$. $Y$ embeds as a subspace of the covering space $\widehat{\theta}:\widehat{Y}\rightarrow X$ associated to the subgroup $H\le G$. In particular, the map $\theta:Y\rightarrow X$ is $\pi_1$-injective:

\begin{lemma}[{\cite[Lemma 16]{wil:one-ended}}]
    If $\mathcal{Y}$ is a connected precover of $\mathcal{X}$ then the map $\theta:Y \rightarrow X$ induces an injective homomorphism $\theta_*:\pi_1(Y) \rightarrow \pi_1(X)$.
\end{lemma}

\begin{rmk}
    If all the covers $\theta_u:Y_u\rightarrow X_{p(u)}$ have finite degree, and if $\mathcal{Y}$ has no hanging elevations, then the map $\Theta:\mathcal{Y}\rightarrow \mathcal{X}$ is a finite-sheeted covering map of graphs of spaces. This map corresponds to a finite cover at the level of geometric realizations, and to a finite-index subgroup at the level of fundamental groups.
\end{rmk}

Throughout the paper, we slightly abuse terminology by referring to precovers of graphs of \emph{groups}, rather than restricting to precovers of graphs of spaces.

\subsection{Separability properties}

We continue by reviewing separability properties of hyperbolic graphs of free groups with cyclic edge groups. The main goal of this section is to explain how the abelianization of an infinite-index (finitely generated) subgroup $H$ of a hyperbolic graph of free groups amalgamated along cyclic subgroups $G$ can be promoted to a direct factor in the abelianization of a finite-index subgroup $G'$ of $G$.

\begin{defn}
    Let $G$ be a group. A subgroup $H\le G$ is 
    \begin{enumerate}
        \item \emph{separable} (in $G$) if for every $g\in G - H$ there exists a finite-index subgroup $G'\le G$ such that $H \le G'$ and $g \notin G'$.
        \item a \emph{virtual retract} (of $G$) if there exists a finite-index subgroup $G'\le G$ containing $H$ and a retraction $r:G'\twoheadrightarrow H$.
    \end{enumerate}
    We say that $G$ is \emph{subgroup separable} (or \emph{LERF}, which stands for \emph{locally extended residually finite}) if every finitely generated subgroup of $G$ is separable; we say that $G$ admits \emph{local retractions} if every finitely generated subgroup of $G$ is a virtual retract.
\end{defn}

Being a virtual retract of a group $G$ is stronger than being separable in $G$:

\begin{lemma}[{\cite[Lemma 2.2]{min:vr}, \cite[Lemma 3.9]{wise:res}}]
    Let $G$ be a residually finite group and let $H$ be a subgroup. If $H$ is a virtual retract of $G$ then $H$ is separable in $G$. In particular, if $G$ admits local retractions then $G$ is subgroup separable.
\end{lemma}    

Moreover, a standard argument implies that admitting a retract onto a subgroup has strong homological consequences:

\begin{lemma}
    \label{lem:direct_factor}
    Let $G$ be a group and let $H\le G$. If there is a retract $r:G\twoheadrightarrow H$ then $H_n(H;\mathbb{Z})$ embeds as a direct factor in $H_n(G;\mathbb{Z})$.
\end{lemma}

\begin{proof}
Denote by $i:H\hookrightarrow G$ the inclusion map, and consider the induced maps on homology
\[
\xymatrix{H_n(H;\mathbb{Z}) \ar[rr]^{i_*} && H_n(G;\mathbb{Z}) \ar[rr]^{r_*} && H_n(H;\mathbb{Z}).}
\]

Since the composition $r_* \circ i_*$ is the identity, $i_*$ is injective. This gives rise to the following short exact sequence of abelian groups
\[
\xymatrix{ 
0 \ar[rr] && H_n(H; \mathbb{Z}) \ar[rr]^{i_*} && H_n(G;\mathbb{Z}) \ar@{->>}[rr] && \nicefrac{H_n(G;\mathbb{Z})}{\mathrm{im}(i_*)} \ar[rr] && 0 
}.\]

The splitting lemma, combined with the existence of the map $r_*$, implies that this sequence is split, giving the desired result.

\end{proof}

Hsu and Wise showed that hyperbolic graphs of free groups with cyclic edges are cubulated \cite[Main Theorem]{wise:cube}; this implies that they are virtually special \cite[Theorem 1.1]{agol}. Special groups virtually retract onto their quasiconvex subgroups \cite[Section 6]{wise:special}, and hyperbolic graphs of free groups with cyclic edges are locally quasiconvex, meaning that each of their finitely generated subgroups is quasiconvex \cite[Theorem D]{wise:local}. Combining these results with \Cref{lem:direct_factor} implies:

\begin{cor} \label{cor:vr-direct-factor}
    If $G$ is a hyperbolic graph of free groups amalgamated along cyclic subgroups, then $G$ has a finite-index subgroup $G'$ with the following property: for every finitely generated $H\le G$, there exists a finite-index subgroup $G_H\le G'$ such that $H^{\ab}$ is a direct factor of $G_H^{\ab}$.
\end{cor}

\section{Free groups and pairs} \label{pairs}

To study the structure of graphs of free groups with cyclic edge groups, it is helpful to adopt a local perspective, focusing on the interaction between a vertex group and its adjacent edges. This motivates the notion of a \emph{pair}, which encodes precisely this local data. Our definition is in the spirit of Wilton's framework (see \cite[Section 1]{Wil18}), with some variations tailored to suit the constructions and arguments developed later on. As one might expect given our setting, we restrict attention to pairs where the vertex group is free and the adjacent edge groups are cyclic. These serve as natural building blocks for the constructive arguments in Sections~\ref{sec:branched} and~\ref{sec:main}. We give special attention to pairs arising from surfaces and their boundaries, whose behaviour in the abelianization will play a central role in our arguments.

\subsection{Pairs and peripheral structures on free groups}
\label{subsec:pairs}
Let $F$ be a finitely generated free group and let $w\in F$ be a non-trivial element. We denote the conjugacy class of $w$ by $[w]_F$; when there is no ambiguity regarding the ambient group $F$, we abuse notation and simply write $[w]$. 

\begin{defn} \label{def:peri}
    A \emph{peripheral structure} on $F$ consists of a set of pairwise distinct conjugacy classes $[w_1],\dots,[w_n]$ of non-trivial elements of $F$. We refer to $(F,\{[w_1],\dots,[w_n]\})$ as a \emph{pair}.    
\end{defn}

\begin{rmk}
    Some authors define peripheral structures on $F$ as conjugacy classes of maximal cyclic subgroups of $F$, rather than as conjugacy classes of elements (or cyclic subgroups). We chose to use the latter, as it allows us to keep better track of attaching maps when constructing covers (and precovers) of graphs of free groups with cyclic edges in Sections \ref{sec:branched} and \ref{sec:main}.
\end{rmk}

\begin{notation}
    Given a tuple of elements $(w_1,\ldots,w_n)\in F^n$, we will often abbreviate and write $\underline{w}=(w_1,\dots,w_n)$. Consequently, we will denote by $[\underline{w}]$ the peripheral structure $\{[w_1],\dots,[w_n]\}$ and by $(F,[\underline{w}])$ the corresponding pair.
\end{notation}

In \Cref{subsec:precovers}, we recalled the notion of (geometric) elevations (see \Cref{def:geom-elev}), which were introduced to describe the graph of spaces structure inherited by a covering of a graph of spaces. We now give an algebraic definition of elevations, phrased in terms of conjugacy classes of elements (that is, in terms of peripheral structures), which coincides with the geometric definition in the finite-degree case.

\begin{defn}[Elevation]
    Let $F$ be a finitely generated free group, let $1\ne w \in F$ and let $H\le F$. An \emph{elevation} of $[w]_F$ to $H$ is a conjugacy class $[u]_H$ of an element $u\in H$ satisfying the following conditions:
    \begin{enumerate}
        \item $u=w_0^d$ for some $d\ge1$ and some representative $w_0\in [w]_F$.
        \item $d$ is the smallest integer for which $w_0^d\in H$.
    \end{enumerate}
    In that case, $d$ is called the \emph{degree} of the elevation $[u]_H$, denoted by $d=\deg_{[w]_F}[u]_H$ (or simply $d=\deg_{[w]}[u]$ when it is clear who $F$ and $H$ are).
\end{defn}

\begin{rmk}
\label{rmk:elevations}
    Represent $F$ as the fundamental group of a graph $X$, and let $\theta:Y\rightarrow X$ be the covering map corresponding to the subgroup $H\le F$. If $H$ is of infinite index, replacing $Y$ with its core graph and replacing $\theta$ by an immersion (as in \Cref{fig:elevation}) is often more convenient, and bears no consequences on what follows. The conjugacy class $[w]_F$ can be represented by means of a continuous map $i:S^1\rightarrow X$, and the conjugacy class $[u]_H$ can be similarly represented by some $j:S^1\rightarrow Y$. We have that $[u]_H$ is an elevation of $[w]_F$ if and only if the map $j$ can be chosen to be a finite-degree (geometric) elevation of the map $i$, as described in \Cref{ex:geometric-elevation}. Note that in this case, the degree of the elevation $\deg_{[w]}[u]$ coincides with the degree of the (geometric) elevation. Recall that (geometric) elevations can have infinite degree, and such elevations are not represented by any conjugacy class $[u]_H$. In what follows, we will only be using finite-degree elevations.
\end{rmk}

\begin{notation}{\label{note:geom-rep}}
    Following \Cref{rmk:elevations}, we will often denote maps such as $i:S^1 \rightarrow X$ described above by $w:S^1\rightarrow X$; more generally given a peripheral structure $[\underline{w}]=\{[w_1],\ldots,[w_n]\}$ on $F$, we abuse notation and use $\underline{w}$ to denote the disjoint union of the maps $w_i:S^1 \rightarrow F$ , that is
    \[
    \underline{w}=\Big(\bigsqcup\limits_{i=1}^n w_i\Big):\bigsqcup \limits_{i=1}^n S^1 \longrightarrow X.
    \]
\end{notation}

\begin{rmk}
    It is straightforward to see that an elevation of an elevation is itself an elevation: suppose that $K\le H\le F$ and let $v\in K$, $u\in H$ and $w\in F$ be non-trivial elements. If $[v]_K$ is an elevation of $[u]_H$ and $[u]_H$ is an elevation of $[w]_F$, then $[v]_K$ is an elevation of $[w]_F$.
\end{rmk}

\begin{figure}[h]
    \centering
    \includegraphics[width=\linewidth]{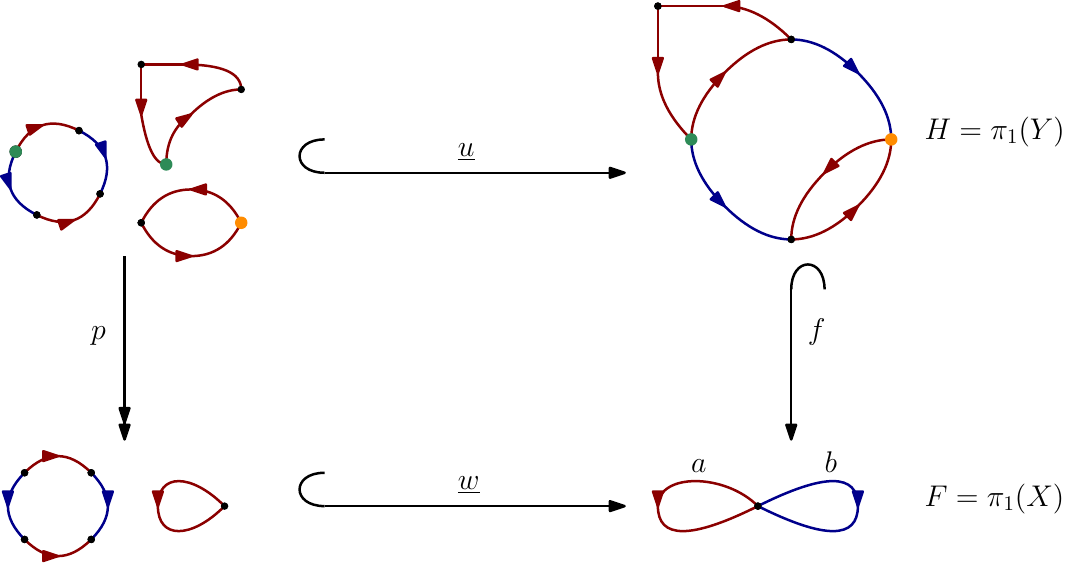}
    \caption{Let $X$ be the rose graph with two petal and let $F=\langle a,b \rangle$ be its fundamental group. Let $w_1=a$ and $w_2=[a,b]$, and let $\underline{w}:S^1 \sqcup S^1 \rightarrow X$ be the map realizing the conjugacy classes $[w_1]$ and $[w_2]$. The (pullback) diagram above describes the elevations of $[\underline{w}]$ to the subgroup $H=\gen{a^3, [a,b],ba^2b^{-1}}$ of $F$. The graph $Y$ at the top-right corner, whose basepoint appears in green, is the core graph of the cover corresponding $H$. The three elevations corresponding to the conjugacy classes $[a^3]$ and $[ba^2b^{-1}]_H$ (of degrees $3$ and $2$ over $[a]_F$ respectively) and $[[a,b]]_H$ (degree $1$ over $[[a,b]]_F$) form the map $\underline{u}:S^1 \sqcup S^1 \sqcup S^1 \rightarrow Y$.}
    \label{fig:elevation}
\end{figure}

\begin{lemma}\label{lem:finitely-many-elevations}
    Let $F$ be a free group, let $[w]_F$ be a non-trivial conjugacy class and let $H\le F$ be a finitely generated subgroup. Then $[w]_F$ has finitely many elevations $[u]_H$ to $H$.
\end{lemma}
\begin{proof}
    Represent $F$ as the fundamental group of a graph $X$, and let $p:Y'\rightarrow X$ be the covering map corresponding to $H$. Since $H$ is finitely generated, $Y'$ admits a compact core $Y\subseteq Y'$ with isomorphic fundamental group $\pi_1(Y)=\pi_1(Y')$. If $[w]_F$ is represented by a reduced loop in $X$, then every (finite-degree) elevation of $[w]_F$ to $Y$ is represented by a reduced loop in $Y'$, and must therefore be contained in $Y$. Since $Y$ is finite, there are only finitely many choices for the ``starting point'' of an elevation (as in \Cref{ex:geometric-elevation}), and thus there are only finitely many elevations $[u]_H$ of $[w]_F$ to $H$.
\end{proof}

As mentioned in the beginning of this section, pairs are meant to describe the relationship between a vertex group of a graph of groups and its adjacent edge groups. To make this assertion precise, let $\mathcal{G}=(\Gamma,\{G_v\},\{G_e\},\{\varphi_e\})$ be a graph of groups with free vertex groups and infinite cyclic edge groups, i.e. $G_e\cong\mathbb{Z}$ for all $e\in E(\Gamma)$, and fix a generator $z_e\in G_e$ for every $e\in E(\Gamma)$. For a vertex $v\in V(\Gamma)$, we consider the conjugacy classes $[\varphi_e(z_e)]_{G_v}$ where $e$ ranges over all edges $e\in E(\Gamma)$ with terminal endpoint $\tau(e)=v$. We say that $\mathcal{G}$ \emph{induces} the peripheral structure $[\underline{w}]=[\varphi_e(z_e)]_{G_v}$ on $G_v$, and refer to the pair $(G_v,[\underline{w}])$ as the \emph{induced pair} (at $G_v$). We are interested in understanding the peripheral structures induced by covers (and precovers) of $\mathcal{G}$ on their vertex groups; to this end we define:

\begin{defn}[Subpairs]\label{def:le-pair}
    Let $(H,[\underline{u}]_H),(F,[\underline{w}]_F)$ be two pairs.
    \begin{enumerate}
        \item We say that $(H,[\underline{u}]_H)$ is a \emph{subpair} of $(F,[\underline{w}]_F)$, denoted by $(H,[\underline{u}]_H)\lepair (F,[\underline{w}]_F)$, if $H\le F$ and every conjugacy class in $[\underline{u}]_H$ is an elevation of some conjugacy class in $[\underline{w}]_F$.
        \item If $(H,[\underline{u}]_H)\lepair (F,[\underline{w}]_F)$ we define the \emph{peripheral closure} of $[\underline{u}]_H$ in $[\underline{w}]_F$ to be
        \[
            \pcl{[\underline{w}]_F}{[\underline{u}]_H}:=\{[w_i] \in [\underline{w}] \; \vert \; [\underline{u}] \text{ contains at least one elevation of } [w_i]\}.
        \]
    \end{enumerate}
\end{defn}

\begin{defn}[Pull-back structure]\label{def:pull-back-structure}
    Let $(F,[\underline{w}]_F)$ be a pair and let $F'\le F$ be a finite index subgroup. The \emph{pull-back structure} of $[\underline{w}]_F$ on $F'$ is the set $[\underline{w}']_{F'}$ of all elevations of elements of $[\underline{w}]_F$ to $F'$. The subpair $(F',[\underline{w}'])\lepair (F,[\underline{w}])$ is called a \emph{finite-index pull-back pair}.
\end{defn}

Note that if $(F,[\underline{w}])$ is a pair and $F'\le F$ is a finite-index subgroup, then by \Cref{lem:finitely-many-elevations} the pull-back structure $[\underline{w}']_{F'}$ contains only finitely many conjugacy classes. It follows that $(F',[\underline{w}'])$ is a pair (and $(F',[\underline{w}'])\lepair (F,[\underline{w}])$). \par \smallskip

With these notions in place, we now have the terminology to discuss peripheral structures induced by coverings of $\mathcal{G}$ on their vertex groups. Two particular cases of interest are the following:
\begin{enumerate}
    \item Let $\mathcal{H}$ be a finite cover of $\mathcal{G}$, let $G_v$ be a vertex group of $\mathcal{G}$ and let $H_u$ be a vertex group of $\mathcal{H}$ that lies above $G_v$. Choosing an appropriate basepoint, we may assume that $H_u \le G_v$. Denote by $(G_v, [\underline{w}])$ the induced pair at $G_v$. In this case, the peripheral structure $[\underline{w}']_{H_u}$ induced on $H_u$ by $\mathcal{H}$ coincides with the pull-back structure of $[\underline{w}]$ on $H_u$, and the induced pair $(H_u,[\underline{w}'])$ is the corresponding finite-index pull-back pair $(H_u,[\underline{w}'])\le(G_v,[\underline{w}])$.
    \item Similarly, let $\mathcal{K}$ be a precover of $\mathcal{G}$, let $G_v$ be a vertex group of $\mathcal{G}$ and let $K_u$ be a vertex group of $\mathcal{K}$ such that $K_u \le G_v$. As before, we denote by $(G_v,[\underline{w}])$ and $(K_u,[\underline{w}'])$ the pairs induced respectively by $\mathcal{G}$ and $\mathcal{K}$ on $G_v$ and $K_u$. Then we have that $(K_u,[\underline{w}']) \le (G_v, [\underline{w}])$. \par \smallskip
    Let $X_v$ be a graph with fundamental group $G_v$ and let $Y_u$ be the core graph of the cover corresponding to $K_u$. Bring to mind that in \Cref{def:le-pair} we did not require the peripheral structure on a subpair to include all possible elevations. The elevations of $[\underline{w}]$ to $K_u$ which do not appear in $[\underline{w}']$ are \emph{exactly} the finite degree (geometric) hanging elevations of $\underline{w}:\bigsqcup S^1 \rightarrow X_v$ to $Y_u$.
\end{enumerate}

Constructing precovers by gluing pieces together, as we will do many times in this paper, requires precise control over the elevations of conjugacy classes to (infinite-index) subgroups. To handle this, we introduce the following notion:

\begin{defn}[Embedded subpairs]\label{def:embed-pair}
    Let $(H,[\underline{u}]_H)$ and $(F,[\underline{w}]_F)$ be two pairs. We say that $(H,[\underline{u}]_H)$ \emph{embeds} in $(F,[\underline{w}]_F)$, and denote $(H,[\underline{u}]_H)\embedpair (F,[\underline{w}]_F)$, if  the following conditions hold:
    \begin{enumerate}
        \item $(H,[\underline{u}]_H)\lepair (F,[\underline{w}]_F)$.
        \item Every $[u]_H\in[\underline{u}]$ is degree-1 elevation of $[u]_F\in[\underline{w}]$.
        \item The map $[\underline{u}]\rightarrow [\underline{w}]$ sending $[u]_H\mapsto[u]_F$ is injective.
    \end{enumerate}
\end{defn}

A direct implication of Marshall Hall's theorem is that every subpair of $(F,[\underline{w}])$ embeds in a finite-index pull-back pair $(F',[\underline{w}'])\lepair (F,[\underline{w}])$:

\begin{lemma}\label{lem:fi-embed-1}
    For every $(K,[\underline{u}])\lepair (F,[\underline{w}])$ there is a finite-index subgroup $F'\le F$ and a corresponding finite-index pull-back pair $(F',[\underline{w}'])\le (F,[\underline{w}])$ such that $(K,[\underline{u}])\embedpair (F',[\underline{w}'])$.
\end{lemma}

The property of being an embedded subpair is preserved under taking finite covers.

\begin{lemma}\label{lem:fi-embed-2}
    Suppose that $(K,[\underline{u}])\embedpair (F,[\underline{w}])$ and suppose that none of the elements in $\underline{w}$ is a proper power. Let $F'\le F$ be a finite-index subgroup and define $K'=K\cap F'$. Let $[\underline{w}']$ and $[\underline{u}']$ be the pull-back structures induced by $(F,[\underline{w}])$ and $(K,[\underline{u}])$ on $F'$ and $K'$. Then $(K',[\underline{u}'])\embedpair (F',[\underline{w}'])$
\end{lemma}
\begin{proof}
    Let $[s]_{K'}\in[\underline{u}']$; by definition it must be an elevation of some $[u]_K\in[\underline{u}]$. Write $s=u^d$ for $d\ge1$, and note that $d$ is the minimal power of $u$ that lies in $K'$. Since $[\underline{u}]$ embeds in $[\underline{w}]$, we must have that $[u]_F\in[\underline{w}]$. If a power of $u$ belongs to $F'$, then it also belongs to $K'=K\cap F'$ (since $u\in K$). It follows that $u^d\in F'$, and no smaller power of $u$ belongs to $F'$. Thus $[s]_{F'}=[u^d]_{F'}$ belongs to $[\underline{w}']$, and $[s]_{K'}$ is an elevation of $[s]_{F'}$ of degree $1$. \par \smallskip

    Having established that conditions 1 and 2 of \Cref{def:embed-pair} hold, it remains to verify condition 3. Suppose that $[s_1]_{K'},[s_2]_{K'}$ are sent to the same conjugacy class $[s_1]_{F'}=[s_2]_{F'}$ in $[\underline{w}']$. For $i=1,2$ we have that $[s_i]_{K'}$ is an elevation of some $[u_i]_K$; write $s_i=u_i^{d_i}$. Since $[\underline{u}]$ is embedded in $[\underline{w}]$, we have that $[u_i]_F$ belongs to the peripheral structure of $F$, and thus $[s_i]_{F'}$ is an elevation of $[u_i]_F$ of degree $d_i$. As $[s_1]_{F'}=[s_2]_{F'}$, it follows that $[u_1]_F=[u_2]_F$ and $d_1=d_2$. Moreover, the fact that $[\underline{u}]$ is embedded in $[\underline{w}]$ implies that $[u_1]_K=[u_2]_K$. We conclude that $[s_1]_{K'}$ and $[s_2]_{K'}$ are elevations of the same element in $[u]_K$, of the same degree $d$. \par \smallskip

    Without loss of generality, write $s_1=u^d$ and $s_2=ku^dk^{-1}$ for some $k\in K$. Since $[s_1]_{F'}=[s_2]_{F'}$ we deduce that $s_2=hu^dh^{-1}$ for some $h\in F'$. But then $u^d$ commutes with $k^{-1}h$, and thus $k^{-1}h=u^a$ for some integer $a$ (since $u$ isn't a proper power). This means that $h=ku^a$ and in particular $h\in K$, yielding that $h\in K\cap F'=K'$, and thus $[s_1]_{K'}=[s_2]_{K'}$, as desired.
\end{proof}

\begin{rmk}
    The requirement that none of the elements of $\underline{w}$ is a proper power in \Cref{lem:fi-embed-2} is necessary: consider the group $F=\gen{a,b}$ and the finite-index subgroups 
    \[K=\gen{a^2,b,aba} \;\;\;\text{   and   }\;\;\; F'=\gen{a^4,ba^3,aba^2,a^2ba,a^3b},\]
    whose intersection is given by
    \[
        K'=K\cap F'=\gen{a^4,ba^4b^{-1},aba^3b^{-1},a^2ba^2b^{-1},a^3bab^{-1},b^2a^2,baba,ba^2b,ba^3ba^3}.
    \]
    We have that $(K,\{[a^2]\})\embedpair (F,\{[a^2]\})$; however, the finite-index pull-back structures yield the following subpair $(K',\{[a^4],[ba^4b^{-1}]\})\lepair (F',\{[a^4]\})$, which is not embedded.
\end{rmk}

\subsection{Peripheral structures on surfaces with boundary}
\label{subsec:surfaces}\
Compact surfaces with boundary will serve as the primary building blocks in our constructions. In this subsection, we gather fundamental results on finite-sheeted covers of such surfaces, with a particular focus on how their boundaries behave in these covers. \par \smallskip

Let $\Sigma$ be a compact surface with non-empty boundary, and let $[x_1],\ldots,[x_n]$ be free homotopy classes representing the loops in $\partial\Sigma$; assume that each $[x_i]$ comes equipped with some orientation. We will use the notation $(\Sigma, \partial \Sigma)$ to refer to the pair $(F=\pi_1(\Sigma),[\underline{x}])$. Note that if $\Sigma' \twoheadrightarrow \Sigma$ is a finite-sheeted cover, the pair $(\Sigma',\partial \Sigma')$ is simply the finite-index pull-back pair (and in particular $(\Sigma',\partial \Sigma') \le (\Sigma, \partial \Sigma)$). Suppose now that $\Sigma$ is orientable and connected. We denote by $g(\Sigma)$ the genus of $\Sigma$, and by $b(\Sigma)$ the number of boundary components of $\Sigma$. Recall that $\chi(\Sigma)=2-2g(\Sigma)-b(\Sigma)$, and in particular $\chi(\Sigma)\equiv b(\Sigma)\;\mod\;2$. \par \smallskip

One of the perks of working with a peripheral structure given by the boundary of a surface, is the near-complete freedom it provides in constructing finite-sheeted covers whose boundary consists of elevations of prescribed degrees.

\begin{lemma}[\emph{cf. }{\cite[Lemma 3.2]{neu:surfaces}}]\label{lem:covers-of-surfaces}
    Let $\Sigma$ be a compact, connected and orientable surface of genus $\ge1$ and with a non-empty boundary $\partial \Sigma$. Let $d$ be a positive integer. For each boundary component $x$ in $\partial \Sigma$, pick a collection of $C_x$ of degrees summing to $d$. Then the following are equivalent:
    \begin{enumerate}
        \item There is a $d$-sheeted cover $\Sigma'$ of $\Sigma$ such that for every $[x]$ in $\partial \Sigma$, the collection of degrees of elevations of $[x]$ to $\Sigma'$ is $C_x$.
        \item\label{itm:modulo-2} $b(\Sigma')\equiv d\cdot b(\Sigma)\;\mod\;2$.
    \end{enumerate}
\end{lemma}

As discussed in the \nameref{intro}, a key ingredient in our constructions is the link between $\partial$-equations and orientable surfaces with boundary of positive genus (see \Cref{sec:branched} for more details). The following lemma illustrates how positive genus can be introduced by passing to a finite-sheeted cover:

\begin{lemma} \label{lem:pos-gen}
    Let $\Sigma$ be a compact, connected, orientable surface of genus $0$ and with $r\ge3$ boundary components. Then there is a finite cover $\Sigma'$ of $\Sigma$ of positive genus.
\end{lemma}
\begin{proof}
We divide the proof into two cases, depending on the number of boundary components of $\Sigma$.
\begin{enumerate}[label=\textbf{Case \arabic*:}]
    \item $b(\Sigma)=r\ge4$. In this case, $\pi_1(\Sigma)$ is a rank $r-1$ free group $F=\gen{a_1,\ldots,a_{r-1}}$, and the peripheral structure $[\partial \Sigma]$ coincides with $\{[a_1],\ldots,[a_{r-1}],[a_1a_2\cdots a_{r-1}]\}$. Consider the index-$2$ subgroup $H\le F$ given by 
    \[H=\gen{a_1^2,a_1a_2,a_2a_1,\dots,a_1a_{r-1},a_{r-1}a_1},\] 
    and observe that the pull-back structure of $[\partial \Sigma]$ on $H$ has exactly $r+(r - 1\;\mod\;2)$ conjugacy classes. If $\Sigma'$ is the cover of $\Sigma$ corresponding to $H$, then 
    \[2-2g(\Sigma')-r-(r - 1 \; \mod \; 2)=\chi(\Sigma')=2\chi(\Sigma)=4-2r,\] 
    and thus $2g(\Sigma')=r-2-(r - 1 \; \mod \; 2)\ge1$ and we are done.
    \item $b(\Sigma)=3$. This time, the pair $(\Sigma,\partial \Sigma)$ coincides with $(F_2=\gen{a,b},\{[a],[b],[ab]\})$. Consider the index-$3$ subgroup $H\le F$ given by 
    \[H=\gen{a^3,ba^2,aba,a^2b},\]
    and similar to the previous case, observe that the pull-back structure of $[\partial \Sigma]$ on $H$ consists of $3$ elevations. To finish, let $\Sigma'$ be the cover of $\Sigma$ corresponding to $H$, and now
    \[2-2g(\Sigma')-3=\chi(\Sigma')=3\chi(\Sigma)=-3.\] 
    Hence $g(\Sigma')=1$ as desired.
    
\end{enumerate}
\end{proof}

The orientation induced by an orientable surface on its boundary also plays an important part in our proofs. We now examine how the orientation on the boundary of a surface behaves under finite covers. To begin, recall some basic facts. Suppose that $\Sigma$ is compact, connected and orientable, write $g(\Sigma)=g$ and denote the boundary components in $\partial \Sigma$ by $x_1,\dots,x_b$. Fix an orientation on $\Sigma$, inducing orientations on $x_1,\dots,x_b$. Then the fundamental group of $\Sigma$ is
\[
    \pi_1(\Sigma)=\pres{s_1,t_1,\dots,s_g,t_g,x_1,\dots,x_b}{[s_1,t_1]\dots[s_g,t_g]x_1\dots x_b}
\]
and its first homology group is free abelian:
\[
    H_1(\Sigma)=\gen{s_1,t_1,\dots,s_g,t_g}\oplus\frac{\gen{x_1,\dots, x_b}}{\gen{x_1+\dots+x_b}},
\]
where the boundary of the fundamental class $[\Sigma]\in H_2(\Sigma,\partial\Sigma)$ is $\partial[\Sigma]=x_1+\dots+x_b\in H_1(\partial\Sigma)$. The following lemma is standard: 

\begin{lemma}[\emph{cf. }{\cite[Exercise 9, Section 3.3]{hatcher}}]
    Let $\Sigma$ be a compact, connected and orientable surface with nonempty boundary. Fix an orientation on $\Sigma$, inducing orientations on its boundary components $x_1,\dots,x_b$. Let $\Sigma'$ be a finite cover of $\Sigma$ (equipped with the orientation inherited from $\Sigma$). Let $x'_{ij}$ be the boundary components covering $x_i$, $1\le i \le b$ (oriented in such a way that $[x_{ij}]$ is a positive-degree elevation of $x_i$). Then
    \[
        \partial[\Sigma']=(x'_{11}+x'_{12}+\dots)+(x'_{21}+x'_{22}+\dots)+\dots+(x'_{b1}+x'_{b2}+\dots).
    \]
\end{lemma}

We seal the subsection by describing the boundary of the fundamental class in the orientable double cover of a non-orientable surface.

\begin{lemma}\label{lem:boundary-non-orient}
    Let $\Sigma$ be a compact, connected and non-orientable surface with nonempty boundary. Fix an arbitrary orientation on each of the boundary components $x_1,\dots,x_b$ of $\Sigma$. Let $\Sigma'$ be the orientable $2$-sheeted cover of $\Sigma$; then we have the following:
    \begin{enumerate}
        \item $\Sigma'$ has boundary components $x_i',x_i''$ (each covering $x_i$ with positive degree $1$), for $i=1,\dots,b$.
        \item The boundary of the fundamental class of $\Sigma'$ is
        \[
            \partial[\Sigma']=(x_1'-x_1'')+\dots+(x_b'-x_b'')
        \]
        (up to swapping between $x_i'$ and $x_i''$ for $i=1,\dots,b$).
    \end{enumerate}
\end{lemma}
\begin{proof}
    For $i=1,\dots,b$, consider a small open neighbourhood of $x_i$ in $\Sigma$, and observe that such a neighbourhood is orientable. It follows that $x_i$ lifts with degree $1$ to the orientable double $\Sigma'$, and that $\Sigma'$ has two boundary components $x_i',x_i''$ each covering $x_i$ with degree $1$. \par \smallskip

    Choose a path $\sigma$ in $\Sigma$, starting and ending at a point $*\in x_i$, such that following $\sigma$ reverses the orientation. Take an ordered basis of the tangent space $T_* \Sigma$ at $*$, consisting of a vector $v_1$ pointing ``inside the surface'' and a vector $v_2$ tangent to $x_i$, and pointing in accordance with the orientation of $x_i$. We transport this basis along $\sigma$; since following $\sigma$ reverses the orientation, we obtain another ordered basis of $T_* \Sigma$: $(v_1, -v_2)$. \par \smallskip
    
    Lifting $\sigma$ to $\Sigma'$, we obtain a path $\sigma'$ connecting $*' \in x_i'$ and $*'' \in x_i''$. Fix an orientation on $\Sigma'$ given by the ordered basis of $T_{*'} \Sigma'$, consisting of a vector $v'_1$ pointing ``inside the surface'' and a vector $v_2$ in the direction of $x_i'$, and in accordance with the orientation on $x_i'$. This orientation coincides with the orientation given by the ordered basis of the tangent space at $*''$, consisting of a vector pointing ``inside the surface'' and a vector going along $x_i''$, but in direction opposite to the orientation of $x_i''$. Hence the boundary of the fundamental class of $\Sigma'$ will contain $x_i'$ and $x_i''$ with opposite signs. The conclusion follows.
\end{proof}

\subsection{Pair splittings and JSJ decompositions}\label{sec:JSJ-decomposition}

Recall that a group $G$ splits \emph{relative} to a subgroup $H$ if there is a graph of groups $\mathcal{G}$ with $\pi_1(\mathcal{G})=G$ and such that $H$ is conjugate into a vertex group of $\mathcal{G}$. We say that the pair $(F,[\underline{w}])$ \emph{splits} if $F$ splits as a graph of groups $\mathcal{F}$, and for every $[w_i]\in [\underline{w}]$ there exists $u_i\in [w_i]$ that is contained in a vertex group of $\mathcal{F}$. When the edge groups of $\mathcal{F}$ take a specific form, e.g. when they belong to a class of groups $\mathcal{C}$, we will say that $(F,[\underline{w}])$ splits \emph{over $\mathcal{C}$} and call this splitting a $\mathcal{C}$-splitting. A splitting of $(F,[\underline{w}])$ is \emph{trivial} if the entire group $F$ is conjugate into a vertex group. 

\begin{defn}[One-ended pairs]
    A pair $(F,[\underline{w}])$ is called \emph{one-ended} if it does not admit a non-trivial free splitting (that is, a splitting where all edge groups are trivial).
\end{defn}

\begin{defn}[Rigid pairs]
    A pair $(F,[\underline{w}])$ is called \emph{rigid} if it has no non-trivial $\mathbb{Z}$-splitting.
\end{defn}

Unlike rigid pairs, which have no non-trivial $\mathbb{Z}$-splittings, \emph{flexible pairs} arise naturally from surfaces with boundary and admit many splittings:

\begin{defn}[Flexible pairs] \label{def:flex}
    A pair $(F,[\underline{w}])$ is called \emph{flexible} if the following conditions hold:
    \begin{enumerate}
        \item There is a compact, connected (possibly non-orientable) surface $\Sigma$ with a non-empty boundary and an isomorphism $F \cong \pi_1(\Sigma)$. We require that $\Sigma$ is not a disk, an annulus, a pair of pants, or a M\"obius band. 
        \item For every $[w_i]\in [\underline{w}]$ there exists a boundary component $[\sigma] \in \partial \Sigma$ such that $[w_i]=[\sigma^d]$ for some $d\in \mathbb{Z}$.
        \item For every $[\sigma]\in \partial \Sigma$ there exists $[w_i]\in [\underline{w}]$ such that $[w_i]=[\sigma^d]$ for some $d\in \mathbb{Z}$. 
    \end{enumerate}
\end{defn}

\begin{rmk}
    If a pair $(F,[\underline{w}])$ is one-ended (respectively, flexible or rigid), we will say that the peripheral structure $[\underline{w}]$ on $F$ is one-ended (respectively, flexible or rigid).
\end{rmk}

The pair of pants stands out in the list of excluded surfaces in the first condition of \Cref{def:flex} above: it is the only one whose fundamental group is a non-abelian free group. It is easy to see that if $\Sigma$ is a pair of pants, then the pair $(\Sigma, \partial \Sigma)$ is rigid. Moreover, let $(F,[\underline{w}])\le (\Sigma,\partial \Sigma)$ be such that $F=\pi_1(\Sigma)$ and $\pcl{\partial \Sigma}{[\underline{w}]}=\partial \Sigma$; then just like $(\Sigma, \partial \Sigma)$, the pair $(F,[\underline{w}])$ does not split over $\mathbb{Z}$. However, by \Cref{lem:pos-gen}, $\Sigma$ admits a finite-sheeted cover $\Sigma'$ of positive genus. This implies that both $(\Sigma', \partial \Sigma')$ and the finite-index pull-back pair $(F'=\pi_1(\Sigma'),[\underline{w}'])\le(F,[\underline{w}])$ are flexible. In fact, a quick check reveals that every finite cover $\Sigma'$ of $\Sigma$ yields a flexible pair $(\Sigma', \partial \Sigma')$ (and the same holds for the corresponding $(F'=\pi_1(\Sigma'),[\underline{w}'])$). \par \smallskip

Aside from the examples described in the previous paragraph, the properties of being one-ended, rigid or flexible are stable under taking finite-index pull-backs; for further detail we refer the reader to Cashen's work \cite[Section 4]{Cas16}. For an in-depth treatment of rigid and flexible peripheral structures, we refer the reader to Guirardel's and Levitt's book \cite{GL17}. \par \smallskip

Our next goal is to describe \emph{JSJ decompositions} of a hyperbolic graph of free groups with cyclic edge groups $\mathcal{G}$. Let $G=\pi_1(\mathcal{G})$, and note that, by Shenitzer's lemma (see, for example, \cite[Theorem 18]{wil:one-ended}, and also \cite{Shenitzer1955}, \cite{Swarup1986} and \cite{DF05}), $G$ admits a non-trivial free splitting if and only if there exists a vertex group $G_v$ of $\mathcal{G}$ such that the induced pair $(G_v,[\underline{w}])$ at $G_v$ is not one-ended. This, in turn, implies that each factor in the Grushko decomposition of $G$ is a hyperbolic graphs of free groups with cyclic edge groups. \par \smallskip

Suppose now that $G$ is one-ended. Roughly speaking, a \emph{cyclic JSJ decomposition} of $G$ serves as a ``dictionary'' that records all cyclic splittings of $G$ (up to conjugation). Before stating this formally, we briefly explain the meaning of the previous sentence. A (cyclic) JSJ decomposition of $G$ is a graph of groups $\mathcal{G}'$ with cyclic edge groups such that
 $\pi_1(\mathcal{G}')=G$. In addition, $\mathcal{G}'$ induces either a rigid or a flexible peripheral structure on each of its vertex groups $G'_v$. The cyclic splittings of $G$ are encoded in $\mathcal{G}'$ in the following way: if $\mathcal{S}$ is a $\mathbb{Z}$-splitting of $G$ with a single edge, then (up to conjugation) either 
\begin{enumerate}
    \item the edge group $S_e$ of $\mathcal{S}$ is an edge group of $\mathcal{G}'$ (and $\mathcal{S}$ can be obtained from $\mathcal{G}'$ by regarding subgraphs of $\mathcal{G}'$ as single vertices), or
    \item the edge group $S_e$ is contained in a flexible vertex $G'_v$ of $\mathcal{G}'$; then, one can refine the splitting $\mathcal{G}'$ to obtain a graph of groups $\mathcal{G}''$ in which $S_e$ appears as an edge group (and as in the previous case, collapse subgraphs into single vertices to obtain $\mathcal{S}$).
\end{enumerate}

We finish the section by stating the precise version of the JSJ decomposition theorem that will be used in \Cref{subsec:normal}, where we put $G$ into a normal form that is convenient to work with:

\begin{thm}[follows from {\cite[Theorem 1]{GL17}}] \label{prop:JSJ}
    Let $G$ be the fundamental group of a graph of free groups with cyclic edge groups. Then $G = \bigast_{i=1}^n G_i\ast F$, where $F$ is a finitely generated free group and each $G_i$ is a one-ended. Furthermore, each $G_i$ splits as a graph of free groups with cyclic edges, which induces either a rigid or a flexible peripheral structure on each of its vertex groups.
\end{thm}

\begin{rmk}
    To see how \Cref{prop:JSJ} follows from Guirardel's and Levitt's \cite[Theorem 1]{GL17}, let $G_i$ be a factor of the Grushko decomposition of $G$ and let $\mathcal{G}_i$ be its splitting as a graph of free groups with cyclic edge groups. Let $\mathcal{G}'_i$ be a JSJ decomposition of $G_i$ obtained by \cite[Theorem 1]{GL17}; note that the edge groups of $\mathcal{G}'_i$ are all cyclic and that $\pi_1(\mathcal{G}'_i)=G_i$. By \cite[Proposition 2.2]{GL17}, there is a refinement of $\mathcal{G}'_i$ dominating $\mathcal{G}_i$. Thus, every universally elliptic vertex group of $\mathcal{G}'_i$ must be contained in some vertex group of $\mathcal{G}_i$. In particular, it must be a free group, on which $\mathcal{G}'_i$ induces a rigid peripheral structure. Moreover, by \cite[Theorem 6.2]{GL17}, non-universally elliptic vertex groups of $\mathcal{G}'_i$ must also be free groups, and the pairs induced by $\mathcal{G}'_i$ at the non-universally elliptic vertices are all flexible.
\end{rmk}

\subsection{Clean covers and free splittings}

We begin by recalling the notion of \emph{cleanliness}, originally introduced by Wise \cite[Definition 4.5]{wise:graph-sep}: a graph of spaces $\mathcal{X}$ is called \emph{clean} if for every edge space $X_e$, the attaching map $i_e:X_e \rightarrow X_{\tau(e)}$ is an embedding. Since cleanliness is a local phenomenon, describing a relation between a vertex space and an adjacent edge space, it makes sense to refer to cleanliness at the level of pairs induced on the vertex groups of $\mathcal{X}$ (see \cite[Definition 27]{wil:one-ended}). Since we do not use cleanliness directly, but rather its consequences, we forego recalling the relevant definitions and refer the reader to \cite{wise:graph-sep} and \cite{wil:one-ended}. \par \smallskip

The consequences of cleanliness, at the level of pairs, are particularly valuable when applied to a rigid pair: Wilton showed \cite[Theorem 8]{wil:one-ended} that rigid pairs admit finite-index pull-back subpairs that satisfy the following strong version of one-endedness (which will be used when we create artificial branching in \Cref{sec:main}):

\begin{defn}[Strongly one-ended pair]\label{def:strongly-one-ended-pair}
    A pair $(F,[\underline{w}])$ is called \emph{strongly one-ended} if for every $[w_i]\in[\underline{w}]$ the pair $(F,[\underline{w}]\setminus\{[w_i]\})$ is one-ended.
\end{defn}

We now recall Wilton's theorem:

\begin{thm}[{\cite[Theorem 8]{wil:one-ended}}]\label{thm:rigid-strongly-one-ended}
    Let $(F,[\underline{w}])$ be a rigid pair which is not a pair of pants. Then there exists a finite-index subgroup $F'\le F$ such that the finite-index pull-back pair $(F',[\underline{w}'])\le(F,[\underline{w}])$ is strongly one-ended.
\end{thm}

\begin{rmk}
    Since being one-ended is preserved under taking finite-index pull-back pairs, the same holds for being strongly one-ended. Moreover, if $(F,[\underline{w}])$ is strongly one-ended, then for every $[w_i]\in [\underline{w}]$ and for every finite-index pull-back pair $(F',[\underline{w}'])\le (F,[\underline{w}])$, the pair
    \[
    (F',[\underline{w}'] \; \setminus \; \{[w'_j]\in [\underline{w}']\;\vert \; [w'_j] \text{ is an elevation of } [w_i]\})
    \]
    is one-ended.
\end{rmk}

We conclude this discussion with the following simple observation, which follows directly from \cite[Definition 27 and Theorem 8]{wil:one-ended}. Let $(F,[\underline{w}])$ be a rigid pair, and identify $F$ with the fundamental group of a rose graph $\Gamma$. Let  $(F',[\underline{w}'])\le(F,[\underline{w}])$ be a strongly one ended finite-index pull-back pair obtained by means of \Cref{thm:rigid-strongly-one-ended}, and let $\Gamma'$ be the finite-sheeted cover of $\Gamma$ that corresponds to $F'$. Recall \Cref{rmk:elevations} and \Cref{note:geom-rep}; Wilton constructs the pair $(F',[\underline{w}'])$ so that for every $[w'_i]\in [\underline{w}']$, the map $w'_i:S^1\longrightarrow \Gamma'$ is an embedding. Therefore, $w'_i$ is a primitive element of $F$. \par \smallskip 
We stress that this only implies that \emph{each} of the elements in $[\underline{w}']_{F'}$ is primitive; there is no guarantee that two or more of them can be realized as part of a common basis. However, once a single primitive element appears in a finite-index pull-back peripheral structure, it is straightforward to promote it to a larger subset of a basis in a further cover:

\begin{lem}\label{lem:part-of-a-basis}
    Let $(F,[\underline{w}])$ be a pair and suppose that $[w_1]\in [\underline{w}]$ is primitive. Let $(F',[\underline{w}'])\lepair (F,[\underline{w}])$ be a finite-index pull-back pair, and let $[w_{11}']_{F'},[w_{12}']_{F'},\dots,[w_{1k}']_{F'}\in[\underline{w}']$ be the conjugacy classes which are elevations of $[w_1]_F$. Then the representatives $w_{11}',w_{12}',\dots,w_{1k}'$ can be chosen to form part of a basis of $F'$.
\end{lem}

\subsection{Normal form} \label{subsec:normal}

Recall the discussion following \Cref{def:peri}, where we said that some authors define a peripheral structure to be a collection of conjugacy classes of maximal cyclic subgroups. We opted to define peripheral structures as conjugacy classes of elements (rather than of maximal cyclic subgroups), as this formulation streamlines technically involved gluing arguments (and hopefully makes them easier to convey). This will be apparent in this section, where we explain how to manipulate a hyperbolic graph of free groups with cyclic edge groups to obtain a finite-index subgroup with a graph of groups splitting that admits a particularly convenient \emph{normal form}. The resulting normal form is a tree of cylinders of a JSJ decomposition (see \cite{GL11}), in which rigid vertex groups are strongly one-ended (and each conjugacy class in their peripheral structure corresponds to a primitive element), and flexible vertices correspond to orientable surfaces of strictly positive genus. \par \smallskip

Recall that a subgroup $H$ of a group $G$ is called \emph{malnormal} if whenever $gHg^{-1} \cap H \ne \{1\}$ for some $g\in G$, we have that $g\in H$. The following definition essentially describes pairs in which the conjugacy classes in the peripheral structure are ``pairwise malnormal''.

\begin{defn}\label{def:malnormal-pair}
    A pair $(F,[\underline{w}])$ is called \emph{malnormal} if for every $[w_i],[w_j]\in[\underline{w}]$ and for every two representatives $w_i\in [w_i]$ and $w_j\in[w_j]$, if $\gen{w_i}\cap g\gen{w_j}g^{-1}\ne\{1\}$, then $w_i=w_j$ and $g\in\gen{w_i}$.
\end{defn}

It is clear from the definition that every subpair of a malnormal pair is itself malnormal.

\begin{rmk}
    \Cref{def:malnormal-pair} is equivalent to the following: every $[w_i]$ is the conjugacy class of a maximal cyclic subgroup, and for every $i\ne j$ the conjugacy classes $[w_i]$ and $[w_j]$ are not inverses of one another. Another equivalent condition, is that for some (or, equivalently, any) choice of representatives $g_i \in [w_i]$, the tuple $(g_1,\ldots,g_n)$ is \emph{independent} (see \cite[Definition 3.1]{wise:graph-sep}) and each $g_i$ generates a maximal cyclic subgroup of $F$.
\end{rmk}

Another reason for our choice to define peripheral structures as conjugacy classes of elements, and not of maximal cyclic subgroups, is that keeping track of a given orientation on the boundary of a surface is crucial for our arguments in the following sections. Following the discussion on $\partial$-equations in the \nameref{intro}, we define:

\begin{defn}[Surface-type pairs]\label{def:surface-type-pair}
    A pair $(F,[\underline{w}])$ is of \emph{surface-type} if there exists a compact, connected and orientable surface with boundary $\Sigma$ of genus $\ge 1$, such that for some orientation on the boundary components of $\Sigma$, $(F,[\underline{w}]) \cong (\Sigma, \partial \Sigma)$.
\end{defn}

Note that every pair of surface-type is both flexible and malnormal; in addition, every finite-index pull-back pair of a surface-type pair is of surface-type.\par \smallskip

A key component of \Cref{subsec:surfaces} was \Cref{lem:covers-of-surfaces}, which grants one optimal flexibility in taking covers of surface-type pairs. Wise introduced a weaker form of this flexibility, \emph{omnipotence}, which holds for malnormal pairs.

\begin{defn}[{\cite[Definition 3.2]{wise:graph-sep}}]
    A group $G$ is called \emph{omnipotent} if, for every tuple of non-trivial independent elements $g_1,\dots,g_r\in G$ (that is, no two of them have conjugate non-trivial powers), there is a constant 
    \[d=d(g_1,\dots,g_r),\] 
    such that for every $r$-tuple of natural numbers $(n_1,\dots,n_r)$, there is a finite quotient $q:G\rightarrow Q$ in which the order of $q(g_i)$ is exactly $d\cdot n_i$ for $i=1,\dots,r$.
\end{defn}

\begin{thm}[{\cite[Theorem 3.5]{wise:graph-sep}}]
    \label{thm:omnipotence}
    Free groups are omnipotent.
\end{thm}

\begin{rmk}
    Note that the independence requirement in the definition of omnipotence is essential: suppose that $(F,[\underline{w}])$ is a pair such that there are two distinct conjugacy classes $[w_1],[w_2]\in[\underline{w}]$ with common powers $[w_1^a]=[w_2^b]$. Let $H\le F$ be a subgroup and let $[u]_H$ be a conjugacy class. If $a,b>0$, then it could be that $[u]_H$ is both a degree-$d_1$ elevation of $[w_1]_F$ and a degree-$d_2$ elevation of $[w_2]_F$. In this case omnipotence clearly fails. A similar failure occurs when, for example, $a>0>b$. This is one of the reasons for introducing the notion of malnormal pairs (\Cref{def:malnormal-pair}).
\end{rmk}

We next describe how to put a one-ended hyperbolic graph of free groups with cyclic edges $G$ into a normal form; both the statement and the proof of \Cref{prop:normalization} are long and tedious. However, converting $G$ into a normal form will put us in an optimal position when we depart towards Sections \ref{sec:branched} and \ref{sec:main}.

\begin{prop}[Normalization]\label{prop:normalization}
Let $H$ be a one-ended hyperbolic group that splits as a graph of free groups with cyclic edges. Then there is a finite index subgroup $G\le H$, which is the fundamental group of a graph of free groups with cyclic edge groups $\mathcal{G}$, satisfying the following properties:
    \begin{enumerate}
        \item \label{itm:1} The vertices of the underlying graph $\Gamma$ of $\mathcal{G}$ are either \emph{cyclic} (with vertex group isomorphic to $\mathbb{Z}$) or \emph{non-cyclic}. 
        \item\label{itm:2} The pairs induced by $\mathcal{G}$ at the non-cyclic vertices are either rigid or flexible; none of them is isomorphic to a pair of pants.
        \item\label{itm:3} The underlying graph $\Gamma$ of $G$ is bipartite; each edge connects a cyclic vertex to a non-cyclic (and hence rigid or flexible) vertex. Moreover, for every $e\in \mathrm{E}(\Gamma)$, if $\tau(e)$ is cyclic then $\varphi_e:G_e \rightarrow G_{\tau(e)}$ is an isomorphism.
        \item For every non-cyclic vertex $v\in \mathrm{V}(\Gamma)$, and for every $[w_i]$ in the induced peripheral structure $[\underline{w}]$ at $G_v$, there is exactly one edge $e\in \mathrm{E}(\Gamma)$ such that $\varphi_e(G_e)\cap[w_i] \ne \emptyset$.
        \item\label{itm:rigid} For every $v\in \mathrm{V}(\Gamma)$ such that the induced pair $(G_v,[\underline{w}])$ at $G_v$ is rigid, $(G_v, [\underline{w}])$ is malnormal and strongly one-ended. In addition, every $[w_i]\in [\underline{w}]$ is a conjugacy class consisting of primitive elements.
        \item\label{itm:flexible} For every $v\in \mathrm{V}(\Gamma)$, if the induced pair $(G_v,[\underline{w}])$ at $G_v$ is flexible, then it is of surface-type.
        \item \label{itm:last} There are no valence-$1$ cyclic vertices; there are no valence-$2$ cyclic vertex adjacent to two flexible vertices.
    \end{enumerate}
\end{prop}

\begin{rmk}
Note that $\mathcal{G}$ in \Cref{prop:normalization} above is a cyclic JSJ decomposition of $G$. In addition, as explained in the paragraph preceding \Cref{lem:part-of-a-basis}, each conjugacy class belonging to the peripheral structure induced at a rigid vertex is primitive when considered alone; however, together, they will not form a part of a common basis.
\end{rmk}

\begin{proof}
Let $\mathcal{H}$ be a JSJ decomposition of $H$ as described in \Cref{sec:JSJ-decomposition}, that is, we can assume that $\mathcal{H}=(\Delta,\{H_v\},\{H_e\},\{\psi_e\})$ has free vertex groups and infinite cyclic edge groups. Up to subdividing each edge (replacing it with two edges joined to a cyclic vertex), we can assume that non-cyclic vertices are never adjacent in $\Delta$.
\begin{enumerate}[label=\textbf{Step \arabic*:}]
\item \textbf{expansions and slide moves.} \label{step:moves} Let $v\in V(\Delta)$ be a non-cyclic vertex, and consider the induced pair $(H_v,[\underline{w}])$ at $H_v$. If some $[w_i]\in [\underline{w}]$ is a proper power, let $u_i\in H_v$ be such that $\gen{u_i}$ is a maximal cyclic subgroup of $H_v$ and $w_i \in \gen{u_i}$. We next perform an \emph{expansion move}, adding a cyclic vertex with vertex group $\gen{u_i}$, attaching it to $H_v$ by identifying it with $\gen{u_i}\le H_v$, and attaching everything that was previously attached to $w_i$ to this new cyclic vertex. We can now assume that every conjugacy class in $[\underline{w}]$ corresponds to a maximal cyclic subgroup of $H_v$. \par \smallskip

Similarly, suppose that multiple edges are attached to the same conjugacy class $[w_i]\in[\underline{w}]$ (or, respectively, to two conjugacy classes $[w_i],[w_i^{-1}]\in[\underline{w}]$). In this case, we perform \emph{slide moves}: if $e_1,e_2\in \mathrm{E}(\Delta)$ are attached to $[w_i]$ (or $[w_i]$ and $[w_i^{-1}]$, we slide $e_2$ along $e_1$ so that $\tau(e_1)$ is now the cyclic vertex $\iota(e_2)$. We repeat this until exactly one edge is attached to the conjugacy class $[w_i]$ (or to one of the two conjugacy classes $[w],[w^{-1}]$). \par \smallskip

Doing so for every non-cyclic vertex group of $\mathcal{G}$, we may assume that $\mathcal{G}$ induces a malnormal pair at every non-cyclic vertex.

\item \textbf{Collapse moves and bipartiteness.} Suppose that two distinct cyclic vertices $v_1,v_2\in \mathrm{V}(\Delta)$ are joined by an edge $e$. If both the inclusions $\psi_e,\psi_{\overline{e}}$ are proper, then the centralizer of the generator of $H_e$ is not cyclic, contradicting the fact that $H$ is a torsion-free hyperbolic group. Thus, one of these two edge maps must be an isomorphism, and we can collapse the edge $e$. Therefore we can assume two different cyclic vertices are never joined by an edge. \par \smallskip

To establish bipartiteness, suppose for a contradiction that there is an edge $e$ such that $\tau(e)=\iota(e)$ is a cyclic vertex. This implies that $H$ contains a Baumslag-Solitar subgroup, which again contradicts hyperbolicity. This shows that $\Delta$ is bipartite.

\item \label{step:flex} \textbf{From flexible to surface-type.} Suppose now that the induced pair $(H_v,[\underline{w}])$ at $H_v$ is either flexible or isomorphic to a pair of pants (and hence rigid). By Lemmas \ref{lem:covers-of-surfaces} and \ref{lem:pos-gen}, we can replace $(H_v,[\underline{w}])$ by a finite-index pull-back pair $(H'_v,[\underline{w}'])$ which is isomorphic to $(\Sigma, \partial \Sigma)$ and such that $\Sigma$ is a surface of positive genus with an even number of boundary components; in particular, $(H'_v,[\underline{w}'])$ is of surface-type. \par \smallskip

In preparation for the next steps, we replace $(H'_v, [\underline{w}'])$ by a further finite-index pull-back pair $(G_v,[\underline{w}''])$, which depends on a collection of constants $\{n_i \;\vert \; [w_i] \in [\underline{w}']\}$ that will be determined later. Each conjugacy class $[w_i']$ in $[\underline{w}']$ is an elevation of a unique conjugacy class in $[\underline{w}]$, having a certain degree which we denote by $d_i'$; let $d_v$ be such that the least common multiple of these degrees divides $d_v$. For any chosen set of additional constants $\{n_i \;\vert \; [w'_i] \in [\underline{w}']\}$, using \Cref{lem:covers-of-surfaces}, we obtain a finite-index pull-back pair $(G_v,[\underline{w}''])$ such that the degree of every elevation of a conjugacy class $[w_i]$ (taken from the original peripheral structure $[\underline{w}]$ on $H_v$) to $G_v$ is $n_i\cdot d_v$.

\item \textbf{Rigid vertices and omnipotence.} \label{step:rigid} For this step, let $v\in \mathrm{V}(\Delta)$ be such that the induced pair $(H_v,[\underline{w}])$ at $H_v$ is rigid (and not isomorphic to a pair of pants). By \Cref{thm:rigid-strongly-one-ended} and the subsequent observation, there is a finite-index pull-back pair $(H_v',[\underline{w}'])$ which is strongly one-ended, and such that every conjugacy class in $[\underline{w}']$ corresponds to a primitive element. We now repeat the construction in Step $3$, replacing \Cref{lem:covers-of-surfaces} with omnipotence. Choose $d_v$ as in the previous step. Up to replacing $d_v$ with some $n\cdot d_v$, \Cref{thm:omnipotence}  implies the following: for any chosen constants $\{n_i \;\vert \; [w_i] \in [\underline{w}']\}$, there is a finite-index pull-back pair $(G_v,[\underline{w}''])\le (H_v',[\underline{w}'])$ such that every elevation of $[w_i]\in [\underline{w}]$ to $[\underline{w}'']$ is of degree $n_i \cdot d_v$.

\item \textbf{Arranging cyclic vertices.} \label{step:cyc} Let $d$ be such that $d_v \vert d$ for every non-cyclic vertex $v\in \mathrm{V}(\Delta)$. Let $v\in \mathrm{V}(\Delta)$ be a cyclic vertex. For every $e\in \mathrm{E}(\Delta)$ with $\tau(e)=v$, there exists $k_e\in \mathbb{Z}$ such that $\psi(H_e)=k_eH_v$. Let $k_v$ be the least common multiple of $d$ and all such $k_e$. We replace $H_v$ with its index-$k_v$ subgroup $G_v=k_vH_v$.

\item \textbf{Choosing the constants.} Let $v\in \mathrm{V}(\Delta)$ be a non-cyclic vertex. Recall that in steps $3$ and $4$, we constructed a finite-index pull-back pair $(H_v',[\underline{w}'])\le (H_v,[\underline{w}])$ satisfying additional properties. Having carried out Step $1$, there is a unique edge $e\in \mathrm{E}(\Delta)$ connecting $[w_i]$ to a cyclic vertex $H_{\tau(e)}$. We choose $n_i = k_{\tau(e)}/d_v$.

\item \textbf{Gluing the pieces.} To finish, we take $m_v$ copies of each $G_v$ such that $m_v:[H_v:G_v]$ is uniform across all $v\in \mathrm{V}(\Delta)$. Our choice of the constants $n_i$ and $m_v$ implies that we can match \emph{all} of the (hanging) elevations in the different vertex groups and obtain a finite (possibly disconnected) cover of $\mathcal{H}$. Choosing a connected component of this cover, we obtain a graph of groups $\mathcal{G}$ whose fundamental group $G$ is a finite index subgroup of $H$. Note that $\mathcal{G}$ satisfies items \ref{itm:1}-\ref{itm:flexible}. Moreover, a cyclic vertex can not have valence $1$, nor can it have valence $2$ and be adjacent to two flexible vertices (otherwise, the original graph of groups $\mathcal{H}$ would have not been a JSJ decomposition of $H$). The conclusion follows.
\end{enumerate}
\end{proof}

\begin{rmk}
    Compare the above normal form with the \emph{tree of cylinders} described in \cite{GL11}, and with the normal form used in Cashen's work \cite{Cas16} (see also \cite{CM11}).
\end{rmk}

\subsection{Surface-type subpairs}

In our context, pairs of surface-type are a particularly useful building block when assembling a graph of free groups with cyclic edges: in the abelianization, the sum of the boundary components in the peripheral structure of a surface-type pair (equipped with the induced orientation) vanishes. Calegari, originally motivated by the goal of proving that \emph{stable commutator length} (\emph{scl}) in a free group is rational (and computable), showed that a converse to this statement also holds. Specifically, he proved that \emph{scl} is computable and rational by finding subpairs of surface type (\emph{extremal surfaces}, see \cite[Proposition 2.74]{Cal09}). We will use Calegari's result, which we recall below, in the proof of \Cref{mainthm}. \par \smallskip

Let $F$ be a finitely generated free group, and let 
\[B_1(F;\mathbb{Z})=\Bigg\{A= \sum \limits_{i=1}^n \lambda_i g_i \;\Bigg\vert\; \lambda_i \in \mathbb{Z},\; g_i\in F,\;  A=0\;\text{ in } F^{\ab}\Bigg\} \le \mathbb{Z}[F]\] 
be the abelian group of integral (group) $1$-boundaries. In more detail, fixing a basis $a_1,\dots,a_n$ of $F$, $A\in B_1(F;\mathbb{Z})$ if and only if the total number of occurrences of each $a_i$ in the sum is zero: writing $\epsilon_j(g)$ to denote the sum of exponents of $a_j$ in $g$ (written as a reduced word in $(a_1,\ldots,a_n)$), for every $j$
\[
\lambda_1 \cdot \epsilon_j(g_1) + \cdots + \lambda_n \cdot \epsilon_j(g_n) = 0.
\]

Let $H\le B_1(F;\mathbb{Z})$ be the subgroup generated by elements of the form $g^d-dg$ and $g-hgh^{-1}$, for $g,h\in F$ and $d\in\mathbb{Z}$. Consider the quotient $B_1^H(F;\mathbb{Z})=B_1(F;\mathbb{Z})/H$. The elements of $B_1^H(F;\mathbb{Z})$ can be thought as formal sums of conjugacy classes of maximal cyclic subgroups of $F$ that vanish in $F^{\ab}$.

Let $\Sigma$ be a compact connected surface equipped with an orientation, and let $[\sigma_1],\dots,[\sigma_b]$ be the conjugacy classes in $\pi_1(\Sigma)$ that correspond to its boundary components (equipped with the orientation inherited from $\Sigma$). Any injective homomorphism $\psi:\pi_1(\Sigma)\rightarrow F$ gives rise to an element of $B_1^H(F;\mathbb{Z})$:
\[\partial_\psi\Sigma:=\psi(\sigma_1)+\dots+\psi(\sigma_b).\]
Note that different choices of representatives of the conjugacy classes $[\sigma_1],\ldots,[\sigma_b]$ will give the same element in $B_1^H(F;\mathbb{Z})$.

\begin{thm}[Calegari {\cite[Theorem 4.24(4)]{Cal09}}]\label{thm:Calegari-surfaces}
    Let $A\in B_1^H(F;\mathbb{Z})$ be a non-zero element. Then there is a compact, connected and oriented surface $\Sigma$, and an injective homomorphism $\psi:\pi_1(\Sigma)\rightarrow F$, such that $\partial_\psi\Sigma=\mu A$ for some $\mu\in\mathbb{Z}\setminus\{0\}$.
\end{thm}
\begin{rmk}
    Technically, Calegari works with real and rational (group) $1$-boundaries, rather than with integral ones. However, up to taking integral multiples of the classes involved (and finite covers of the corresponding surfaces), everything can easily be assumed to be integral.
\end{rmk}

We will use the following special case Calegari's \Cref{thm:Calegari-surfaces}.

\begin{cor}\label{cor:Calegari}
    Let $(F,[\underline{w}])$ be a malnormal pair satisfying a non-trivial linear relation 
    \[ \sum \limits _{[w_i]\in [\underline{w}]} \lambda_iw_i=0\]
    in $F^\ab$, with $\lambda_i\in\mathbb{Z}$. Then
    \begin{enumerate}
        \item there exists a subpair $(H,[\underline{u}])\lepair (F,[\underline{w}])$ of surface-type, and
        \item $\pcl{[\underline{w}]}{[\underline{u}]}$ contains every $[w_i]$ that appears in the equation with coefficient $\lambda_i\not=0$.
    \end{enumerate}
\end{cor}

\begin{proof}
    Let $(F,[\underline{w}])$ be a malnormal pair satisfying some non-trivial equation $\sum \lambda_iw_i=0$ in $F^\ab$, with $\lambda_i\in\mathbb{Z}$; define $A=\sum \lambda_iw_i\in B_1^H(F,\mathbb{Z})$. Let $\psi:\pi_1(\Sigma)\rightarrow F$ be the injective homomorphism given by \Cref{thm:Calegari-surfaces}, let $\sigma_1,\dots,\sigma_k\in\pi_1(\Sigma)$ be elements representing the conjugacy classes corresponding to boundary components of $\Sigma$, and let $H\le F$ be the subgroup given by the image of $\psi$.\par \smallskip
    Note that $\sigma_1,\dots,\sigma_k$ are not proper powers in $\pi_1(\Sigma)$, and that they lie in distinct conjugacy classes of maximal cyclic subgroups. Since $\partial_\psi\Sigma$ is a multiple of $A$ in $B_1^H(F;\mathbb{Z})$, $\psi$ must send each of the elements $\sigma_i$ to an elevation of a conjugacy class in $[\underline{w}]$ to $H$. This gives rise to a peripheral structure $[\underline{u}]$ on $H$, making $(H,[\underline{u}])\lepair (F,[\underline{w}])$ into a subpair of surface-type, which completes the proof.
\end{proof}

It is worth noting that Calegari has already used \Cref{thm:Calegari-surfaces} in the context of hyperbolic graphs of free groups with cyclic edge groups; he showed that such a group with non-vanishing second homology must contain a closed, hyperbolic surface subgroup \cite[Theorem 3.4]{cal:subs}. Inspired by Calegari's techniques, Wilton was able to eliminate the aforementioned homological requirement, and prove that every non-free hyperbolic graph of free groups amalgamated along cyclic subgroups contains a surface subgroup \cite[Theorem A]{Wil18}. Wilton's proof relies on ``local analysis'' of pairs induced by a graph of groups at the different vertex groups (and like Calegari's rationality theorem, utilizes linear programming techniques). \par \smallskip
This shift—from requiring non-vanishing second homology to assuming non-freeness—also manifests at the level of pairs. When seeking a surface subpair (a building block for a surface subgroup), instead of assuming that the conjugacy classes in the peripheral structure $[\underline{w}]$ of a malnormal pair $(F,[\underline{w}])$ satisfy a non-trivial linear relation in the abelianization as in \Cref{thm:Calegari-surfaces}, Wilton simply assumes that the pair $(F,[\underline{w}])$ is one-ended. We finish by recalling a few definitions and stating Wilton's theorem.

\begin{defn}
    A pair $(H,[\underline{u}])$ is of \emph{weak surface-type} if it is isomorphic to a free product of pairs of surface-type.
\end{defn}

\begin{defn}
    A subpair $(H,[\underline{u}])\lepair (F,[\underline{w}])$ is called \emph{admissible} (of degree $d\in\mathbb{N}$) if for every $[w_i]\in[\underline{w}]$ we have that
    \[
        \sum\limits_{[u]\in[\underline{u}]\text{ elevation of }[w_i]}\deg_{[w_i]}[u]=d.
    \]
\end{defn}

\begin{thm}[Wilton {\cite[Theorem 5.11]{Wil18}}]\label{thm:Wilton-surfaces}
    Let $(F,[\underline{w}])$ be a malnormal, one-ended pair. Then there exists an admissible subpair $(H,[\underline{u}])\lepair (F,[\underline{w}])$ of weak surface-type.
\end{thm}

\begin{rmk}
    Note that Wilton \cite{Wil18} does not require that the surfaces are orientable or of genus $\ge1$; the surfaces are only required to be flexible (see {\cite[Theorem 5.6]{Wil18}}). However, these extra properties can be easily obtained by taking a finite cover of each surface (while, at the same time, preserving the admissibility).
\end{rmk}

Finally, our arguments in \Cref{sec:main} require handling multiple surface subpairs simultaneously within a single pair. The following proposition ensures that they can be jointly embedded into a common finite-index pull-back pair.

\begin{prop}\label{prop:two-surfaces-in-one-subgroup}
    Let $(F,[\underline{w}])$ be a malnormal pair, and let $(S_1,[\underline{v}_1]),(S_2,[\underline{v}_2])\lepair (F,[\underline{w}])$ be two surface-type subpairs. Then there is a finite-index pull-back pair $(F',[\underline{w}'])\le (F,[\underline{w}])$ such that, for $i=1,2$, the finite-index pull-back pair $(S_i\cap F',[\underline{v}_i'])\lepair (S_i,[\underline{v}_i])$ satisfies $(S_i\cap F',[\underline{v}_i'])\embedpair (F',[\underline{w}'])$.
\end{prop}
\begin{rmk}
   The subgroup $F'$ in \Cref{prop:two-surfaces-in-one-subgroup} above can be chosen so that $S_1\cap S_2\le F'$. Furthermore, standard applications of Marshall Hall's theorem imply that $S'_1$ and $S'_2$ can be chosen to be free factors of $F'$ (albeit not simultaneously).
\end{rmk}

\begin{proof}
    By \Cref{lem:fi-embed-1}, there exists a finite-index pull-back pair $(F',[\underline{w}'])\lepair (F,[\underline{w}])$ such that $(S_1,[\underline{v}_1])\embedpair (F',[\underline{w}'])$. Consider the finite-index pull-back pair $(S_2\cap F',[\underline{v}_2'])\lepair (S_2,[\underline{v}_2])$, and apply \Cref{lem:fi-embed-1} again to obtain a finite-index pull-back pair $(F'',[\underline{w}''])\lepair (F',[\underline{w}'])$ such that $(S_2\cap F',[\underline{v}_2'])\embedpair (F'',[\underline{w}''])$. Now consider the finite-index pull-back pair $(S_1\cap F'',[\underline{v}_1''])\lepair (S_1,[\underline{v}_1])$. Since $(F,[\underline{w}])$ is malnormal, \Cref{lem:fi-embed-2} applies. We obtain that $(S_1\cap F'',[\underline{v}_1''])\embedpair (F'',[\underline{w}''])$, as required.
\end{proof}

\section{\texorpdfstring{$\partial$}{∂}-equations and branched surfaces} \label{sec:branched}

Every finitely generated abelian group $G$ is obtained by surjecting a finitely generated free abelian group $\mathbb{Z}^r$ to a system of linear constraints. The structure theorem for finitely generated abelian groups states that these constraints can be chosen to be of the form $n_i \cdot x_i = 0$. In this section, we take a slightly different approach and show that every finitely generated abelian group can be realized as the abelianization of a \emph{branched surface} (see \Cref{subsec:branch}), possibly after taking a direct product with some $\mathbb{Z}^n$. \par \smallskip

Our argument is essentially an algorithm that converts a finite system of linear equations of the form
\[
\Xi = \{n_i\cdot x_i = 0 \; \vert \; 1\le i \le r,\; n_i \in \mathbb{Z}\}
\]
into a system $\Psi$ of \emph{$\partial$-equations}: equations of the form $\epsilon_1\cdot x_1 + \epsilon_2\cdot x_2 + \cdots + \epsilon_k \cdot x_k=0$ where $\epsilon_i \in \{\pm 1\}$. Every $\partial$-equation $\epsilon_1\cdot x_1 + \epsilon_2\cdot x_2 + \cdots + \epsilon_k \cdot x_k=0$ admits natural geometric counterparts: these come in the form of compact, connected and orientable surfaces $\Sigma$ of genus $\ge 1$ with $k$ boundary components; the signs $\epsilon_i$ in the $\partial$-equation correspond to a choice of orientations on the boundary components $\sigma_1,\ldots,\sigma_k$ in $\partial \Sigma$. \par \smallskip

\begin{conv}
    We treat a system $\Psi$ of $\partial$-equations as a multiset: we allow equations to appear multiple times in $\Psi$. This clearly has no effect when we interpret $\Psi$ as a system of constraints on a free abelian group. However, this technical choice gives us more flexibility when constructing geometric objects that carry the information encoded by $\Psi$ (and it will play a role in \Cref{subsec:standard}).
\end{conv}

Every system $\Psi$ of $\partial$-equations in variables $x_1,\ldots,x_n$ in turn gives rise to (many) graphs of spaces $\mathcal{X}_\Psi$ (in fact, graphs of free groups with cyclic edge groups). These graphs of spaces differ only in the genera of the surfaces used to represent the $\partial$-equations in $\Psi$. The vertex spaces of such $\mathcal{X}_\Psi$ come in two forms:
\begin{enumerate}
    \item \emph{variable} vertex spaces (cyclic vertex spaces): for each variable $x_i$ of $\Psi$ there is a vertex space $X_i$ homeomorphic to (an oriented) circle $S^1$;
    \item \emph{equation} vertex spaces (surface-type vertex spaces): for every $\partial$-equation $\psi = \epsilon_1\cdot x_1 + \epsilon_2\cdot x_2 + \cdots + \epsilon_k \cdot x_k=0$ in $\Psi$ there is a vertex space $X_\psi$ which is a compact, connected and orientable surface of positive genus and with $k$ boundary components $\sigma  ^\psi_1,\ldots, \sigma^\psi_k$. We fix an orientation $O_\psi$ on every $X_\psi$ (which induces an orientation on every $\sigma^\psi _i$). 
\end{enumerate}

The underlying graph $\Gamma_\Psi$ of $\mathcal{X}_\Psi$ is bipartite, with an edge connecting a vertex with variable vertex space $X_i$ to a vertex with equation vertex space $X_\psi$ if and only if the variable $x_i$ appears in the equation $\psi$. The edge spaces of $\mathcal{X}_\Psi$ are (oriented) circles, and for every edge $e\in \mathrm{E}(\Gamma_\Psi)$ with $\tau(e)$ a variable vertex $X_j$, the edge map $i_e$ is an orientation-preserving homeomorphism. If $\iota(e)$ is an equation vertex, the edge map $i_{\overline{e}}$ identifies $X_e$ with the boundary component $\sigma^\psi_j$; the sign $\epsilon_i\in \{ \pm 1\}$ indicates whether the gluing map respects the aformentioned orientations on $X_e$ and $\sigma^\psi_i$ (i.e. $\epsilon_i=1$) or not ($\epsilon_i=-1$). We invite the reader to verify that the abelianization of the geometric realization of $\mathcal{X}_\Psi$ is isomorphic to a finitely generated abelian group obtained by quotienting out the system of $\partial$-equations $\Psi$ from a free abelian group $\mathbb{Z}^r$ (for some $r \ge n$). A graph of spaces $\mathcal{X}_\Psi$ obtained in this manner is called a \emph{$\Psi$-complex}. \par \smallskip

We say that a system of equations $\Xi'$ in variables $x_1,\ldots, x_{n'}$ is a \emph{void extension} of a system $\Xi$ in variables $x_1,\ldots,x_n$ if $n'\ge n$, but the equations in $\Xi'$ are exactly the equations in $\Xi$ (that is, the variables $x_{n+1},\ldots, x_{n'}$ do not appear in any equation in $\Xi'$). Two systems of equations $\Xi$ and $\Psi$ in variables $x_1,\ldots,x_n$ are \emph{equivalent} if the two quotients of $\mathbb{Z}^n=\gen{x_1,\ldots,x_n}$ by $\Xi$ and $\Psi$ are isomorphic.

\begin{example}
    Let $\Psi$ be a system of $\partial$-equations and suppose that every variable of $\Psi$ appears in at most two equations in $\Psi$. In this case, the system $\Psi$ is equivalent to a void extension of either the empty system of equations, or the system $\{2x_1=0\}$. It is easiest to verify this geometrically: let $\mathcal{X}_\Psi$ be a $\Psi$-complex, and note that the condition on $\Psi$ implies that the geometric realization $X_\Psi$ of $\mathcal{X}_\Psi$ is a compact surface. If $X_\Psi$ has a non-empty boundary, or is closed and orientable, then $\pi_1(X_\Psi)^{\ab}$ is free abelian; therefore $\Psi$ is equivalent to a void extension of the empty system of equations. If $X_\Psi$ is a closed and non-orientable surface, then the abelianization of $\pi_1(X_\Psi)$ takes the form $\mathbb{Z}^r\oplus \mathbb{Z}/2\mathbb{Z}$, and $\Psi$ is (equivalent to) a void extension of $\{2x_1=0\}$.
\end{example}

Since our goal is to find \emph{arbitrary} torsion in the abelianization of graphs of free groups with $\mathbb{Z}$ edges, the remainder of the section focuses on systems $\Psi$ of $\partial$-equations in which some variables appear in more than two equations. We next look at \emph{branched surfaces}, which can be thought of as $\Psi$-complexes for such $\Psi$.

\subsection{Branched surfaces} \label{subsec:branch}

Before we begin our discussion of branched surfaces, a word of caution is in order: throughout this paper, the term branched surface \emph{does not} refer to a branched cover of a surface, but rather to a compact $2$-complex which is locally modeled on a surface, except along a finite collection of embedded loops where branching occurs. Formally:

\begin{defn}
    \label{def:branched-surface}
    A \emph{branched surface} is a hyperbolic, one-ended graph of free groups amalgamated along cyclic subgroups $\mathcal{B}$ which satisfies:
    \begin{enumerate}
        \item $\pi_1(\mathcal{B})$ is not the fundamental group of a compact surface, and
        \item the normal form of $\mathcal{B}$ (see \Cref{prop:normalization}) does not contain rigid vertices.
    \end{enumerate}
\end{defn}

\begin{figure}[h]
    \centering
    \includegraphics[width=0.6\linewidth]{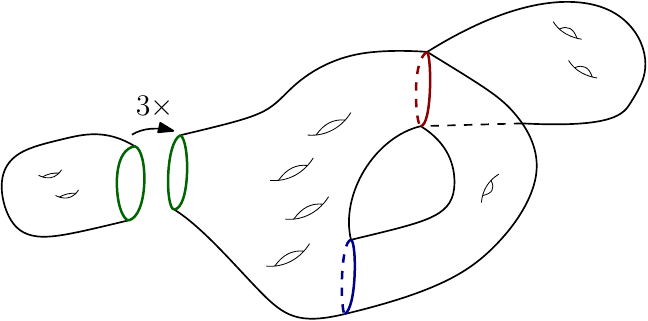}
    \caption{A branched surface $\mathcal{B}$. The branching locus $\mathcal{C}$ is highlighted in green and in red (the green connected component is of branching valence $4$, and the red connected component is of branching valence $3$). Note that the blue curve (where two surface pieces are glued) is not a part of the branching locus.}
    \label{fig:branched surface}
\end{figure}

\begin{rmk}
    In particular, by the last condition of \Cref{prop:normalization}, the degree of every cyclic vertex in the underlying graph $\Gamma$ of $\mathcal{B}$ is at least $3$.
\end{rmk}

By abuse of notation, we also refer to the geometric realization $B$ of $\mathcal{B}$ as a branched surface. The \emph{branching locus} $\mathcal{C}$ of a branched surface $\mathcal{B}$ consists of all points $x\in B$ that do not admit a neighbourhood homeomorphic to $\mathbb{R}^2$. If $\mathcal{B}$ is in normal form, the branching locus of $\mathcal{B}$ coincides with the collection of circles
\[
\bigsqcup \limits_{v \text{ is a cyclic vertex of } \mathcal{B}} X_v \subset B.
\]

Each connected component $C$ of the branching locus $\mathcal{C}$ of $\mathcal{B}$ has a \emph{(branching) valence}: an integer $\ge 3$ such that for every $x\in C$, there is a neighbourhood of $x$ in $B$ that is modeled after $d$ half-planes $\mathbb{R} \times \mathbb{R}_{\ge 0}$ glued together along the $d$ copies of the line $\mathbb{R} \times \{0 \}$. If $\mathcal{B}$ is in normal form, and $C$ is a connected component of the branching locus $\mathcal{C}$ of $\mathcal{B}$, then the (branching) valence of $C$ coincides with its valence in the underlying graph of $\mathcal{B}$.

\begin{rmk} Suppose that $\mathcal{B}$ is a branched surface in normal form. Then there exists a system $\Psi_{\mathcal{B}}$ of $\partial$-equations such that $\mathcal{B}$ is a $\Psi_{\mathcal{B}}$-complex. Moreover, every variable in $\Psi_{\mathcal{B}}$ appears in at least $3$ $\partial$-equations. Conversely, given a system $\Psi$ of $\partial$-equations, every associated $\Psi$-complex is a branched surface if and only if: 
\begin{enumerate} 
\item every variable appears in at least $2$ equations in $\Psi$, and 
\item some variable appears in more than $2$ equations in $\Psi$. 
\end{enumerate}
\end{rmk}

\subsection{Standard branched surfaces}\label{subsec:standard}

We continue by proving that every branched surface admits a \emph{standard} precover; such a precover is a particularly simple branched surface, and it will serve as fertile ground for constructing further, more intricate precovers.

\begin{defn} \label{def:standard}
    A \emph{standard branched surface} is a branched surface $\mathcal{B}$ satisfying (see \Cref{fig:standard-branched-surface}):
    \begin{enumerate}
        \item $\mathcal{B}$ is in normal form.
        \item $\mathcal{B}$ has $3$ surface-type vertices, denoted $\Sigma$, $\Theta$ and $\Pi$, each of which has $2$ boundary components.
        \item $\mathcal{B}$ has $2$ cyclic vertices, which we denote by $C_1$ and $C_2$.
        \item Fixing orientations on $C_1$ and $C_2$, the boundary components of each of $\Sigma$, $\Theta$ and $\Pi$ are identified with $C_1$ and $C_2$ so that
        \[
        \partial[\Sigma]=\partial[\Theta]=\partial[\Pi]=C_1-C_2,
        \]
        where $[\Sigma]$, $[\Theta]$ and $[\Pi]$ are the fundamental classes of $\Sigma$, $\Theta$ and $\Pi$,
    \end{enumerate}
\end{defn}

\begin{rmk}
    Standard branched surfaces are exactly the $\Psi$-complexes of the system of $\partial$-equations which consists of the equation $x_1-x_2$, repeated three times.
\end{rmk}

\begin{figure}[h!]
    \centering
    \includegraphics[width=0.5\linewidth]{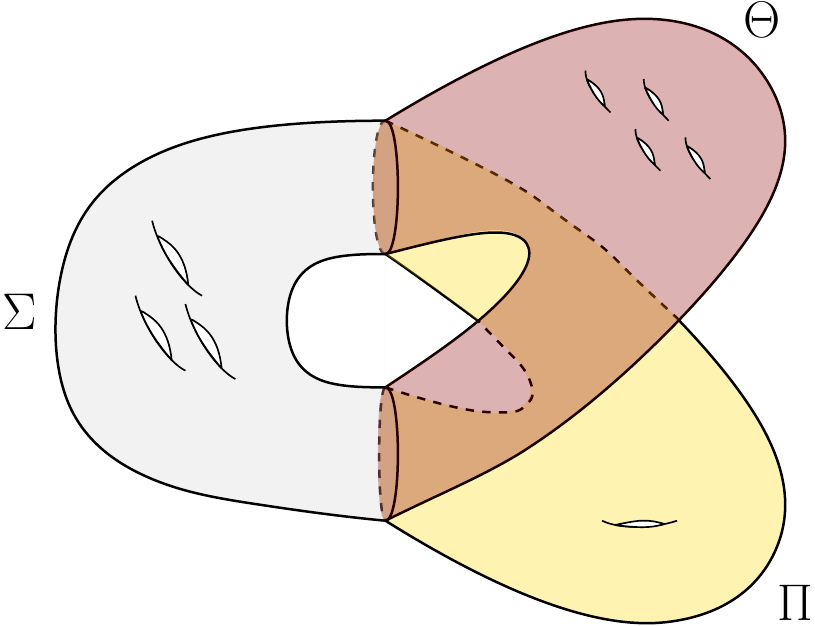}
    \caption{A standard branched surface.}
    \label{fig:standard-branched-surface}
\end{figure}

\begin{lem}[Standard precovers]\label{lem:standard-pre-covers}
    Let $\mathcal{B}$ be a branched surface. Then there exists a precover $\mathcal{B}'$ of $\mathcal{B}$ which is a standard branched surface.
\end{lem}

\begin{proof}
    Up to taking a finite-sheeted cover of $\mathcal{B}$, we can assume by \Cref{prop:normalization} that $\mathcal{B}$ is in normal form. Fix a cyclic vertex space $C$ of $\mathcal{B}$ in the branching locus $\mathcal{C}$, and fix $3$ (pairwise distinct) boundary components $\sigma$, $\theta$ and $\pi$ which are identified with $C$, coming from three (not necessarily distinct) surfaces. \par \smallskip

\begin{enumerate}[label=\textbf{Step \arabic*:}]
    \item \textbf{(The three pieces).} Take two copies of each surface-type vertex of $\mathcal{B}$. For each cyclic vertex $C' \ne C$, we take as many copies of $C'$ as its branching degree in $\mathcal{B}$. For each cyclic vertex in this collection, choose (arbitrarily) two surface-type vertices that can be glued to it. The precover of $\mathcal{B}$ obtained after gluing is a (possibly disconnected) surface with boundary containing $2\cdot \deg (C)$ hanging elevations; these include two hanging elevations corresponding to each of $\sigma$, $\theta$ and $\pi$. \par \smallskip
    Next, take $\deg(C)-1$ copies of $C$, and glue to each of them two hanging elevations, to obtain a surface $S$ with two boundary components $\sigma_1$ and $\sigma_2$ which are elevations of $\sigma$. We may assume that $S$ is connected (but possibly non-orientable), for if not we replace it with a connected component containing $\sigma_1$ (and pass to a $2$-sheeted cover with two boundary components if $\sigma_2$ is not in the same connected component as $\sigma_1$). Similarly, we construct connected (possibly non-orientable) surfaces $T$ and $P$ with boundary components $\theta_1,\theta_2$ and $\pi_1,\pi_2$ respectively. 

    \item \textbf{(Orientability).} The goal of this step is to make sure that the surfaces obtained in the previous step are orientable. Suppose that $S$ is non-orientable, and let $S'$ be its orientable double cover. By \Cref{lem:boundary-non-orient}, $\partial S'$ has four boundary components $\sigma_1',\sigma_1'',\sigma_2'$ and $\sigma_2''$ with 
    \[\partial[S']=(\sigma_1'-\sigma_1'')+(\sigma_2'-\sigma_2'').\] 
    If $T$ is orientable, and $\partial[T]=\theta_1\pm\theta_2$, glue $\theta_1$ and $\theta_2$ to two of the four boundary components of $S'$; by matching orientations correctly (that is, choosing the boundary components of $S'$ to which we glue $\theta_1$ and $\theta_2$), we obtain an orientable surface $\Sigma$ with two boundary components, each of which is a degree $1$ elevation of $\sigma$, as desired. \par \smallskip
    
    Suppose now that $T$ is non-orientable, and let $T'$ be its orientable $2$-sheeted cover with boundary 
    \[\partial[T']=(\theta_1'-\theta_1'')+(\theta_2'-\theta_2'').\]
    Here $\theta_i'$ and $\theta_i''$ both cover $\theta_i$ with degree $1$. Using \Cref{lem:covers-of-surfaces}, take a further finite-sheeted covering $\overline{S}$ of $S'$ with six boundary components:
    \begin{enumerate}
        \item  $s_1$ and $s_1'$, each covering $\sigma_1'$ with degree $1$,
        \item $s_2$ and $s_2'$ each covering $\sigma_2''$ with degree $1$, 
        \item $s_3$ covering $\sigma_1''$ with degree $2$, and
        \item $s_4$ covering $\sigma_2'$ with degree $2$.
    \end{enumerate}
    It follows that
    \[\partial[\overline{S}]=(s_1+s_1')-(s_2+s_2')+(s_3-s_4).\] 
    To finish, glue the boundary components of $T'$ to the four boundary components $s_1,s_1',s_2$ and $s_2'$ of $\overline{S}$, matching orientations so that the resulting precover is an orientable surface $\Sigma$ with two boundary components $s_3,s_4$, each covering $\sigma$ (with degree $2$). \par \smallskip

    In both cases we obtain a connected and orientable surface $\Sigma$ with two boundary components covering $\sigma$ (both with degree either $1$ or $2$). By \Cref{lem:covers-of-surfaces}, up to taking a $2$-sheeted cover of $\Sigma$, we can assume that the two components of $\partial \Sigma$ both cover $\sigma$ with degree $2$. Repeating the same construction for $\theta$ and $\pi$, we obtain connected, orientable surfaces $\Theta$ and $\Pi$, each with two boundary components covering $\theta$ and $\pi$, respectively, with degree $2$.

    \item \textbf{Induced orientations on boundary.} To obtain a standard branched surface, we need to verify that condition $4$ of \Cref{def:branched-surface} is satisfied. Note that if the surface $S$ constructed in step $1$ is non-orientable, then 
    \[\partial[\Sigma]=\sigma_1-\sigma_2\]
    as desired. More generally, if one of the surfaces $S$, $T$ or $P$ produced in step $1$ is non-orientable, one can glue it to a two-sheeted cover of each of the others, and make sure that all three surfaces $S$, $T$ and $P$ are non-orientable. In this case, step $2$ will output three surfaces $\Sigma$, $\Theta$ and $\Pi$ with
    \[\partial[\Sigma]=\sigma_1-\sigma_2\;\;,\;\; \partial[\Theta]=\theta_1-\theta_2\;\;\text{and}\;\; \partial[\Pi]=\pi_1-\pi_2,\]
    which can be assembled into a standard branched surface. \par \smallskip

    It remains to consider the case where $S$, $T$ and $P$ from step $1$ are all orientable. Suppose first that $\partial[\Sigma]=\sigma_1-\sigma_2$ and that $\partial[\Theta]=\theta_1+\theta_2$. Take a $2$-sheeted cover $\Sigma'$ of $\Sigma$ with $\partial[\Sigma']=(\sigma_1'+\sigma_1'')-(\sigma_2'+\sigma_2'')$ and glue the two boundary components of $\Theta$ to $\sigma_1''$ and $\sigma_2''$ (with ``mismatched'' orientations). The resulting surface is non-orientable, which allows us to repeat the previous steps and obtain the conclusion.\par \smallskip

    Finally, suppose that $\partial[\Sigma]=\sigma_1+\sigma_2$ and $\partial[\Theta]=\theta_1+\theta_2$ and $\partial[\Pi]=\pi_1+\pi_2$. In this case, take another copy $\Sigma'$ of $\Sigma$, with $\partial[\Sigma']=\sigma_1'+\sigma_2'$. Glue $\sigma_2$ to $\theta_1$, then $\theta_2$ to $\pi_1$, and finally $\pi_2$ to $\sigma_1'$. The boundary of the fundamental class of the resulting surface is $\sigma_1-\sigma_2'$, which reduces the problem to the previous case.
\end{enumerate}
\end{proof}

\subsection{Prescribed homological torsion in covers of branched surfaces}

\begin{prop}
    \label{prop:branched-torsion}
    Let $\mathcal{B}$ be a standard branched surface. For every finite abelian group $M$, there is a finite cover $\mathcal{B}_M$ of $\mathcal{B}$ such that $H_1(B_M)\cong M \oplus \mathbb{Z}^r$ for some $r\in\mathbb{N}$.
\end{prop}

\begin{rmk}
    \Cref{prop:branched-torsion} can be reformulated in terms of systems of $\partial$-equations: suppose that $\mathcal{B}$ is a standard branched surface, and let $\Xi$ be a finite system of linear equations with integral coefficients. Then there exists a system $\Psi$ of $\partial$-equations which is equivalent to a void extension of $\Xi$, and a finite-sheeted cover $\mathcal{B}_\Psi$ of $\mathcal{B}$ which is a $\Psi$-complex.
\end{rmk}

\begin{cor}
    Combining \Cref{lem:standard-pre-covers} and \Cref{cor:vr-direct-factor} we obtain the following: suppose that $\mathcal{B}$ is a branched surface, and let $M$ be a finite abelian group. Then $\mathcal{B}$ admits a finite-sheeted cover $\mathcal{B}_M$ such that $M$ is a direct factor of $H_1(B_M)$.
\end{cor}

\begin{proof}[Proof of \Cref{prop:branched-torsion}]
    Write $M=\mathbb{Z}/r_1\mathbb{Z} \oplus \cdots \oplus \mathbb{Z}/r_n \mathbb{Z}$ with $r_i\ge 2$ for every $1\le i \le n$. Denote the three surface-type (equation) vertices of $\mathcal{B}$ by $\Sigma$, $\Theta$ and $\Pi$ and by $y$ and $z$ the two cyclic (variable) vertices of $\mathcal{B}$; we orient $y$ and $z$ so that
    \[\partial[\Sigma]=\partial[\Theta]=\partial[\Pi]=y-z.\]
    To simplify notation, we set $x=y-z$. Take the following covers of $\Sigma$, $\Theta$, $\Pi$, $y$ and $z$ (whose existence is guaranteed by \Cref{lem:covers-of-surfaces}):
    \begin{enumerate}
        \item Copies $y_{i,j}$ of $y$, for $i=1,...,n$ and $j=1,...,3r_i-2$.
        \item Copies $z_{i,j}$ of $z$, for $i=1,...,n$ and $j=1,...,3r_i-2$.
        \item Connected covers $\Sigma_{i,0}$ of $\Sigma$, for $i=1,...,n$, so that $\Sigma_{i,0}$ has $r_i$ boundary components covering $y$ with degree $1$, and $r_i$ boundary components covering $z$ with degree $1$.
        \item Connected covers $\Sigma_{i,j}$ of $\Sigma$, for $i=1,...,n$ and $j=1,...,r_i-1$, each with four boundary components: two degree-$1$ elevations of $y$, and two of $z$.
        \item Connected covers $\Theta_{i,j}$ of $\Theta$, for $i=1,...,n$ and $j=1,...,r_i-1$, each with four boundary components: two degree-$1$ elevations of $y$, and two of $z$.
        \item Connected covers $\Pi_{i,j}$ of $\Pi$, for $i=1,...,n$ and $j=1,...,r_i-1$, each with four boundary components: two degree-$1$ elevations of $y$, and two of $z$.
        \item Connected covers $\Theta_0$ and $\Pi_0$ of $\Theta$ and $\Pi$, respectively, which will be determined later on.
    \end{enumerate}

    Again, to simplify notation, we set $x_{i,j}=y_{i,j}-z_{i,j}$ for $i=1,\dots,n$ and $j=1,\dots,3r_i-2$. We glue the surfaces and the circles in the collection above as laid out in \Cref{tab:torsion-summand} (where $1$ means that two boundary components of the surface representing the column, are glued to the two circles corresponding to the line). As the construction is fairly elaborate, the reader may find it helpful to examine the illustrative example provided in \Cref{fig:example_torsion}.

\begin{table}[h!]
    \setlength{\arrayrulewidth}{0.2mm}  
    \setlength{\tabcolsep}{6pt}         
    \renewcommand{\arraystretch}{1.2}  
    \centering

    \begin{tabular}{|c || c | cc | cc | c | cc | cccc|}
        \hline
        & $\Sigma_{i,0}$ & $\Theta_{i,1}$ & $\Pi_{i,1}$ & $\Theta_{i,2}$ & $\Sigma_{i,2}$ & $\cdots$ & $\Theta_{i,r_i-1}$ & $\Sigma_{i,r_i-1}$ & $\Pi_{i,1}$ & $\Pi_{i,2}$ & $\cdots$ & $\Pi_{i,r_i-1}$ \\
        \hline\hline
        $x_{i,1}$      & 1 & 0 & 0 & 0 & 0 &$\cdots$& 0 & 0 & 1 & 0 &$\cdots$& 0 \\
        $x_{i,2}$      & 1 & 0 & 0 & 0 & 0 &$\cdots$& 0 & 0 & 0 & 1 &$\cdots$& 0 \\
        \vdots            &\vdots&\vdots&\vdots&\vdots&\vdots&\vdots&\vdots&\vdots&\vdots&\vdots&\vdots&\vdots\\
        $x_{i,r_i-1}$  & 1 & 0 & 0 & 0 & 0 &$\cdots$& 0 & 0 & 0 & 0 &$\cdots$& 1 \\
        $x_{i,r_i}$    & 1 & 1 & 0 & 0 & 0 &$\cdots$& 0 & 0 & 0 & 0 &$\cdots$& 0 \\
        \hline
        $x_{i,r_i+1}$  & 0 & 1 & 1 & 0 & 0 &$\cdots$& 0 & 0 & 1 & 0 &$\cdots$& 0 \\
        $x_{i,r_i+2}$  & 0 & 0 & 1 & 1 & 0 &$\cdots$& 0 & 0 & 0 & 0 &$\cdots$& 0 \\
        \hline
        $x_{i,r_i+3}$  & 0 & 0 & 0 & 1 & 1 &$\cdots$& 0 & 0 & 0 & 1 &$\cdots$& 0 \\
        $x_{i,r_i+4}$  & 0 & 0 & 0 & 0 & 1 &$\cdots$& 0 & 0 & 0 & 0 &$\cdots$& 0 \\
        \hline
        \vdots            &\vdots&\vdots&\vdots&\vdots&\vdots&\vdots&\vdots&\vdots&\vdots&\vdots&\vdots&\vdots\\
        \hline
        $x_{i,3r_i-3}$ & 0 & 0 & 0 & 0 & 0 &$\cdots$& 1 & 1 & 0 & 0 &$\cdots$& 1 \\
        $x_{i,3r_i-2}$ & 0 & 0 & 0 & 0 & 0 &$\cdots$& 0 & 1 & 0 & 0 &$\cdots$& 0 \\
        \hline
    \end{tabular}
    \caption{\emph{A ``gluing matrix'' for the cover $\mathcal{B}_M$ of $\mathcal{B}$}. The first column gives the $\partial$-equation $x_{i,1}+\dots+x_{i,r_i}$, so that, in order to an element of order $r_i$ in $\pi_1(B_M)^{\ab}$, we only need to identify $x_{i,1},\dots,x_{i,r_i}$ with one other. The subsequent blocks $\Theta_{i,1}$ and $\Sigma_{i,1}$ identify $x_{i,r_i},x_{i,r_i+2},x_{i,r_i+4},\dots$, creating several ``copies'' of $x_{i,r_i}$. The last columns $\Pi_{i,1},\dots,\Pi_{i,r_i-1}$ identify $x_{i,1},\dots,x_{i,r_i-1}$ with these ``copies'', producing the desired result.}
    \label{tab:torsion-summand}
\end{table}
    In more detail, for $i=1,\dots,n$, the surface $\Sigma_{i,0}$ the $r_i$ elevations of $y$ are glued to $y_{i,1},\ldots,y_{i,r_i}$, and the $r_i$ elevations of $z$ are glued to $z_{i,1},\ldots,z_{i,r_i}$, so that the fundamental class of $\Sigma_{i,0}$ has boundary
    \[
        \partial[\Sigma_{i,0}]=x_{i,1}+\dots+x_{i,r_i}.
    \]
    Similarly, for $i=1,\dots,n$ and $j=1,\dots,r_i-1$, we glue the surfaces $\Sigma_{i,j},\Theta_{i,j}$ and $\Pi_{i,j}$ so that they correspond to the following $\partial$-equations:
    \[
        \partial[\Sigma_{i,j}]=x_{i,r_i+2j-1}+x_{i,r_i+2j},
    \]
    \[
        \partial[\Theta_{i,j}]=x_{i,r_i+2j-2}+x_{i,r_i+2j-1}, \text{ and}
    \]
    \[
        \partial[\Pi_{i,j}]=x_{i,j}+x_{i,r_i+2j-1}.
    \]
    Finally, the covers $\Theta_0$ and $\Pi_0$ are constructed to have boundary covers that can be attached to all of the cyclic vertices that are not glued to any $\Theta_{i,j}$ and $\Pi_{i,j}$, respectively. \par \smallskip
    
    One can easily verify that the resulting space $B_M$ is a connected, finite-sheeted cover of $B$. In the level of the abelianization, the elements $y_{i,j},z_{i,j}$ (for $i=1,...,n$ and $j=1,...,3r_i-2$) generate a direct summand of $\pi_1(B_M)^{\ab}$, subject only to the $\partial$-equations coming from the equation vertices (i.e. the surface pieces). The rest $\pi_1(B_M)^{\ab}$ is a free abelian group. \par \smallskip
    
    It remains to verify that the aforementioned direct summand in $\pi_1(B_M)^{\ab}$ is isomorphic to $M$. For each $i\in\{1,...,n\}$, the surfaces $\Sigma_{i,j}$ and $\Theta_{i,j}$ tell us that
    \[
        x_{i,r_i}=x_{i,r_i+2}=\cdots=x_{i,3r_i-2}=-x_{i,r_i+1}=-x_{i,r_i+3}=\cdots=-x_{i,3r_i-3},
    \]
    while the surfaces $\Pi_{i,j}$ imply
    \[
        x_{i,1}=x_{i,2}=\cdots=x_{i,r_i}=x_{i,3r_i-2}.
    \]
    The surface $\Sigma_{i,0}$ yields
    \[
        r_i\cdot x_{i,1}=0.
    \]
    Finally, the surfaces $\Theta_0$ and $\Pi_0$ only add ``redundant'' constraints ($\Theta_0$ gives $\sum_{i=1}^n r_i \cdot x_{i,1} = 0$ and $\Pi_0$ gives $x_{i,r_i}+(r_i-1)\cdot x_{i,1}=0$). Thus $\pi_1(B_M)^{\ab}$ is isomorphic to the direct sum of $\bbZ/r_1\bbZ\oplus...\oplus\bbZ/r_n\bbZ$ and a free abelian group.
\end{proof}

\begin{figure}[h!]
    \centering
    \hspace*{-0.1\linewidth}
    \includegraphics[width=1.1\linewidth]{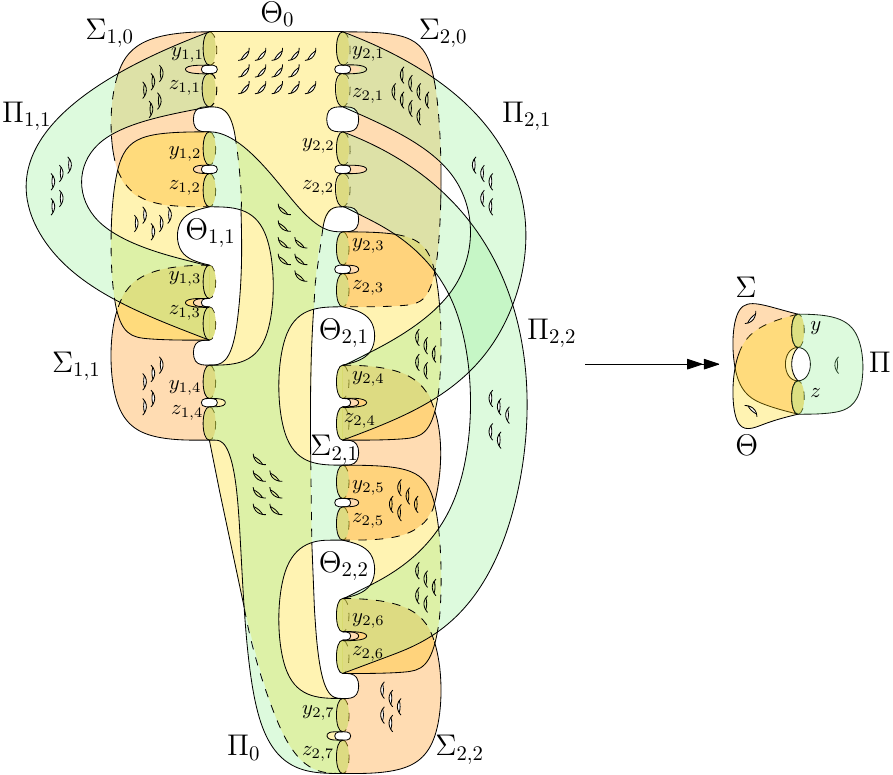}
    \caption{\emph{The cover $\mathcal{B}_{\mathbb{Z}/2\mathbb{Z} \oplus \mathbb{Z}/3\mathbb{Z}}$ of $\mathcal{B}$}. The standard branched surface $\mathcal{B}$ appears on the right, and its $11$-sheeted cover $\mathcal{B}_{\mathbb{Z}/2\mathbb{Z} \oplus \mathbb{Z}/3\mathbb{Z}}$ constructed in \Cref{prop:branched-torsion} appears on the left.}
    \label{fig:example_torsion}
\end{figure}

A corollary of \Cref{prop:branched-torsion} above is that many ``branching'' hyperbolic graphs of free groups with cyclic edge groups—those whose normal forms include a cyclic vertex of degree at least three—admit finite index subgroups with arbitrary torsion in their abelianization.

\begin{cor}
    \label{cor:general-branch}
    Let $G$ be a hyperbolic fundamental group of a graph of free groups with cyclic edge groups. Suppose that the normal form of one of the free factors in the Grushko decomposition of $G$ has a cyclic vertex of valence $\ge3$. Then, for every finite abelian group $M$, there exists a finite index subgroup $G_0\le G$ such that $M$ is a direct summand of $G_0^\ab$.
\end{cor}

\begin{proof}
    We restrict to a free factor $H$ in the Grushko decomposition of $G$ whose normal form contains a cyclic vertex of degree at least $3$; passing to a finite-index subgroup we can assume that $H$ is in normal form. Fix a graph of groups splitting $\mathcal{H}$ of $H$ satisfying the properties laid out in \Cref{prop:normalization}.\par \smallskip

    Let $H_v$ be a rigid vertex of $\mathcal{H}$. By Wilton's \Cref{thm:Wilton-surfaces}, we can substitute the induced pair $(H_v,[\underline{w}_v])$ at $H_v$ with a pair $(H'_v,[\underline{u}_v])$ of weak surface type. Moreover, the admissibility condition implies that the total degree of the elevations of $[w_i]\in [\underline{w}_v]$ in $[\underline{u}_v]$ is the same across all $i$. We do so for every rigid vertex. By \Cref{lem:covers-of-surfaces}, up to taking further finite-sheeted covers of these surfaces (and of the surface-type vertices of $\mathcal{H}$), we can assume that every non-cyclic vertex of $\mathcal{H}$ is of weak surface type, and admits the same number of elevations of the same degree over each conjugacy class in the original induced peripheral structure at every vertex of $\mathcal{H}$. We do so ensuring that the number of elevations, as well as their degree, does not depend on the chosen non-cyclic vertex. \par \smallskip 

    We next attach all of these surfaces (to appropriate covers of the cyclic vertices of $\mathcal{H}$) to obtain a precover $\mathcal{H}'$ of $\mathcal{H}$. The existence of a vertex of valence $\ge3$ in the underlying graph of $\mathcal{H}$ implies that we can glue these to obtain a (possibly disconnected) branched surface; we pass to a connected component $\mathcal{B}$ whose branching locus is non-empty. Using \Cref{prop:artificial-branching} we obtain a precover $\mathcal{B}_M$ of $\mathcal{B}$ such that $\pi_1(\mathcal{B}_M)\le G$ is finitely generated and has $M$ as a direct summand in its abelianization. \Cref{cor:vr-direct-factor} yields a finite-index subgroup $G_0\le G$ such that $M$ is a direct summand of $G_0^{\ab}$.
\end{proof}

\section{Artificial branching and the proof of \texorpdfstring{\Cref{mainthm}}{Theorem A}}
\label{sec:main}

We now turn to the proof of \Cref{mainthm}, which we restate below in a more explicit form:

\begin{thm}\label{thm:main}
    Let $G$ be a hyperbolic group that splits as a graph of free groups with cyclic edge groups, and that is not isomorphic to a free product of free and surface groups. Let $M$ be a finite abelian group. Then there exist
    \begin{enumerate}
        \item a finite-index subgroup $G_0\le G$,
        \item a connected branched surface $\mathcal{B}_M$, with $H_1(B_M)=M\oplus \mathbb{Z}^r$, and
        \item a homomorphism $f:\pi_1(B_M)\rightarrow G_0$,        
    \end{enumerate}
    such that $f_*:\pi_1(B_M)^{\ab}\rightarrow G_0^{\ab}$ maps $M$ isomorphically to a direct summand of $G_0^{\ab}$.
\end{thm}

\begin{rmk}
    Taking a more geometric approach, standard arguments yield that \Cref{thm:main} implies the following: there exists a graph of spaces decomposition $\mathcal{X}$ of $G$, a finite-sheeted covering $\mathcal{X}'\twoheadrightarrow \mathcal{X}$ and a map of $2$-complexes $f:B_M\rightarrow X'$ such that $f_*:H_1(B_M)\rightarrow H_1(X')$ maps $M$ isomorphically to a direct summand of $H_1(X')$.
\end{rmk}

\Cref{virtualcor} follows immediately from \Cref{thm:main} above:

\begin{cor}\label{cor:virtual}
    \Cref{thm:main} also applies to hyperbolic graphs of virtually free groups with $2$-ended edge groups (that are not virtually a free product of free and surface groups).    
\end{cor}

\begin{proof}
Indeed, if $G$ is a hyperbolic group that splits as a graph of virtually free groups with virtually cyclic edge groups, then by \cite[Theorem 5.1]{wise:graph-sep} $G$ is residually finite, and so virtually torsion free. Therefore, there is a finite-index subgroup $G_0 \le G$ that splits as a graph of free groups with cyclic edges. Since $G$ is not virtually a free product of free and surface groups, $G$ is not isomorphic to a free product of free and surface groups, and \Cref{mainthm} applied to $G_0$ gives the desired result.
\end{proof}

Before diving into the proof of \Cref{thm:main}, we make a few preliminary observations. First, by \Cref{cor:vr-direct-factor}, it is enough to find any subgroup of $G$ with the desired torsion in the abelianization, and not necessarily a finite-index one. In addition, since the statement of \Cref{thm:main} is a virtual one, we are always allowed to replace $G$ by a finite-index subgroup.

\begin{stassumption} \label{assumption}
    In what follows, we will assume the following:
    \begin{enumerate}
        \item $G$ \textbf{is one-ended, and not virtually a surface group}. This is done by passing to a factor of the Grushko decomposition of $G$, and noting that by Kerckhoff's Nielsen realization theorem, a torsion-free group which contains a finite-index surface subgroup is itself a surface group \cite{Kerckhoff1983}.
        \item Replacing $G$ by a finite-index subgroup, we assume that $G$ \textbf{is in normal form, satisfying the properties described in \Cref{prop:normalization}}. In particular, the graph of groups $\mathcal{G}$ must contain two non-cyclic vertices.
        \item \textbf{Every cyclic vertex of} $\mathcal{G}$ \textbf{has valence} $2$, and $\mathcal{G}$ \textbf{contains a rigid vertex}. \Cref{cor:general-branch} shows that \Cref{thm:main} holds true whenever $\mathcal{G}$ has a cyclic vertex of valence $\ge 3$, and 1 above implies that if every cyclic vertex of $\mathcal{G}$ has valence $2$ then at least one non-cyclic vertex of $\mathcal{G}$ is rigid (and, since $\mathcal{G}$ is in normal form, this rigid vertex is not a pair of pants).
    \end{enumerate}
\end{stassumption}

\subsection{Strategy outline} We now give an overview of the main ideas underpinning the proof of \Cref{thm:main}. As seen in \Cref{cor:general-branch}, if $\mathcal{G}$ contains a cyclic vertex of valence at least $3$, then one can construct a branched surface inside the group $G$, and use it to produce a subgroup with any prescribed torsion in its abelianization. However, in general, $\mathcal{G}$ may not contain a cyclic vertex of valence $\ge 3$. The bulk of the proof is therefore devoted to replacing this missing feature. Our strategy is to construct a piece, which we call an \emph{artificial branching block}, that replaces such a cyclic vertex, and mimics the branching behaviour of a branched surface. This is carried out in \Cref{prop:artificial-branching} through an extremely delicate construction.

The idea is to work inside a rigid vertex group $G_v$ (where the induced pair $(G_v,[\underline{w}])$ at $G_v$ is strongly one-ended). We choose a conjugacy class $[w_i] \in [\underline{w}]$, remove it from the peripheral structure $[\underline{w}]$, and invoke Wilton’s theorem (\Cref{thm:Wilton-surfaces}) to find a surface subgroup whose boundary avoids elevations of $[w_i]$. Ignoring two different conjugacy classes in $[\underline{w}]$ one at a time, we obtain two surfaces $\Sigma_1$ and $\Sigma_2$ with a common boundary component $\sigma$; the union $\Sigma_1 \cup \Sigma_2$ (accompanied by additional precovers of $\mathcal{G}$), will play the role of a branched surface. \par \smallskip

Nonetheless, naively joining $\Sigma_1$ and $\Sigma_2$ may cause complications in homology, and controlling the image of $\sigma$ in the abelianization is crucial for producing the desired torsion. In particular, attaching surfaces to cap off boundary components of $\Sigma_1$ and $\Sigma_2$ that get in the way of our construction, may inadvertently kill $\sigma$ in the abelianization. To deal with this, we rely on Calegari’s Rationality Theorem (see \Cref{cor:Calegari}), which produces a surface $\Sigma$ realizing any given linear combinations of conjugacy classes which sum to $0$ in the abelianization.  A minimality argument applied to $\pcl{[\underline{w}]}{\partial \Sigma}$ rids us of such unwanted equations (\cref{itm:minimality} of \Cref{prop:artificial-branching} and \Cref{lem:key-property}), and allows us to prove that $\sigma$ survives as an infinite-order element in the abelianization of a suitable artificial branching block. \par \smallskip

Finally, we combine copies the artificial branching block with surface pieces obtained via \Cref{prop:surfaces-prescribed-peripherals}, allowing us to emulate the same torsion-producing behaviour seen in the branched surface case (the system of $\partial$-equations laid out in \Cref{tab:torsion-summand}). This completes the proof of \Cref{thm:main}.

\subsection{Constructing surfaces to cap off hanging elevations}

We begin by constructing surfaces with prescribed boundary, which serve as auxiliary pieces in our construction and allow us to cap off hanging elevations. The underlying arguments are standard, but the proof is technically involved.

\begin{prop}[Surfaces with prescribed boundary]\label{prop:surfaces-prescribed-peripherals}
    Let $G_v$ be a non-cyclic vertex group of $\mathcal{G}$ and let $[t]$ be a conjugacy class in the peripheral structure induced by $\mathcal{G}$ at $G_v$. Then there exist $\epsilon\in\{-1,1\}$ and a pre-cover $f:\mathcal{H}\rightarrow\mathcal{G}$ giving rise to a subgroup $H\le G$ with the following properties:
    \begin{enumerate}
        \item \label{itm:surf-1} There are vertex groups $H_{v_1},H_{v_2}$ of $\mathcal{H}$ that are mapped to $G_v$ under $f$, and elevations $[t_1],[t_2]$ of $[t]$ in the respective peripheral structures induced by $\mathcal{H}$ at $H_{v_1}$ and $H_{v_2}$, such that $[t_1]$ and $[t_2]$ are hanging elevations in $\mathcal{H}$.
        \item $(H,\{[t_1],[t_2]\})$ is a surface-type pair (i.e. a connected orientable surface of genus $\ge1$ with exactly two boundary components $[t_1]$ and $[t_2]$).
        \item The degrees of of the $[t_1]$ and $[t_2]$ to $H$ are both equal to $d>0$. The boundary of the fundamental class of the surface $(H,\{[t_1],[t_2]\})$ is $\partial [H]=[t_1]+\epsilon\cdot[t_2]$.
    \end{enumerate}
\end{prop}

Before proving \Cref{prop:surfaces-prescribed-peripherals}, we make a few remarks.

\begin{rmk}
    In \cref{itm:surf-1}, we allow for $v_1$ and $v_2$ to coincide; however, we always require that $[t_1],[t_2]$ are distinct.
\end{rmk}

\begin{rmk}
    Up to taking further finite-sheeted covers of the surfaces produced in \Cref{prop:surfaces-prescribed-peripherals}, we may assume that the degree $d$ does not depend on the chosen non-cyclic vertex $G_v$ or on the chosen conjugacy class $[t]$.
\end{rmk}

\begin{rmk}\label{rmk:propagating-non-orientability}
    The value of $\epsilon=\pm1$ can depend on the choice of the non-cyclic vertex $G_v$ and of the conjugacy class $[t]$ in the peripheral structure induced by $\mathcal{G}$ at $G_v$. However, the choice $\epsilon=-1$ can be propagated along the graph of groups $\mathcal{G}$ (an explanation appears after the proof of \Cref{prop:surfaces-prescribed-peripherals}):
    \begin{enumerate}
        \item If $[t]$ and $[t']$ are two conjugacy classes that are attached to the same cyclic vertex of $\mathcal{G}$, and the choice $\epsilon=-1$ is possible for $[t]$, then it is also possible for $[t']$.
        \item Suppose that that $G_v$ is of surface-type. If the choice $\epsilon = -1$ is possible for some conjugacy class in the peripheral structure induced by $\mathcal{G}$ at $G_v$, then one can choose $\epsilon = -1$ for every conjugacy class in the peripheral structure induced by $\mathcal{G}$ at $G_v$.
        \item Suppose that $G_v$ is rigid, and denote by $(G_v,[\underline{w}])$ the peripheral structure induced by $\mathcal{G}$ at $G_v$. Suppose further that there exists a surface-type subpair $(\Sigma,\partial \Sigma)\le (G_v,[\underline{w}])$. If one can choose $\epsilon=-1$ for some $[t]\in \pcl{[\underline{w}]}{\partial \Sigma}$, then the same holds for every $[t']\in \pcl{\underline{w}]}{\partial \Sigma}$.
    \end{enumerate}
    Finally, even though we focus on the case where every cyclic vertex of $\mathcal{G}$ has valence $2$, it is worth noting that in the general case, whenever $[t]$ is attached to a cyclic vertex of valence $\ge 3$, the choice $\epsilon=-1$ is possible. 

\end{rmk}

\begin{proof}[Proof of \Cref{prop:surfaces-prescribed-peripherals}]

We first assemble a collection of surfaces that will be used to construct $\mathcal{H}$. \par \smallskip

Let $v$ be a rigid vertex of $\mathcal{G}$ and let $(G_v,[\underline{w}])$ be the pair induced by $\mathcal{G}$ at $G_v$. By \Cref{thm:Wilton-surfaces}, there exists an admissible subpair $(\Sigma_v,\partial \Sigma_v) \le (G_v,[\underline{w}])$ of weak surface-type. Up to taking finite-index covers of the connected components of $\Sigma_v$ using \Cref{lem:covers-of-surfaces}, we can assume that all of the boundary components in $\partial \Sigma_v$ are elevations of the same degree $d$ of conjugacy classes in $[\underline{w}]$. Similarly, since $(G_v,[\underline{w}])$ is strongly one-ended, for every $[w_i]\in[\underline{w}]$, there exists an admissible subpair $(\Sigma_v^i,\partial \Sigma_v^i)\le (G_v,[\underline{w}]\setminus \{w_i\})$ of weak surface-type. As before, by passing to finite-sheeted covers, we can assume that all of the boundary components in $\partial \Sigma_i$ are elevations of the same degree $d_i$ of conjugacy classes in $[\underline{w}]\setminus \{w_i\}$. Taking further finite covers, we can assure that all of the $d_i$ coincide and are equal to $d$. Finally, by passing again to finite-sheeted covers (and using admissibility), we can assume that every conjugacy class $[w_i]\in[\underline{w}]$ has the same number $k$ of elevations of degree $d$ in each of $(\Sigma_v,\partial \Sigma_v)$ and $(\Sigma_v^j,\partial \Sigma_v^j)$ (for $j\ne i$); furthermore, $k$ and $d$ do not depend on $i$ and $j$. \par \smallskip

Similarly, let $v$ be a surface-type vertex of $\mathcal{G}$ and let $(G_v,[\underline{w}])$ be the pair induced by $\mathcal{G}$ at $G_v$. Then for every $k$ and $d$ \Cref{lem:covers-of-surfaces} implies that there exists a finite-index pull-back subpair $(\Sigma_v,\partial \Sigma_v) \le (G_v,[\underline{w}])$ to which every $\{w_i\} \in [\underline{w}]$ admits $k$ degree-$d$ elevations. Up to taking finite covers of all the surfaces obtained over all of the non-cyclic vertices of $\mathcal{G}$, we may assume that the numbers $k$ and $d$ are uniform across all of the vertices of $\mathcal{G}$. Note that the construction relies on \Cref{lem:covers-of-surfaces}, which may force $k$ to be even; in any case, we can always choose $k$ to be even. \par \smallskip

We continue with an observation. Let $v$ be a cyclic vertex of $\mathcal{G}$, adjacent to edges $e_1,\dots,e_n$. For each $1\le i \le n$, suppose that there is a collection of $a_i\in \mathbb{N}$ hanging elevations that may be attached to $v$ via $e_i$. To ensure that the resulting precover is a surface, one has to match these hanging elevations to-by-two; note that an elevation attached to $v$ via $e_i$ can not be paired with another elevation attached to $v$ via $e_i$. One can carry out such a matching procedure if and only if the following two conditions hold: 
\begin{enumerate}
    \item $a_1+\cdots+a_n$ is even, and 
    \item for every $1\le i \le n$ the following triangle-like inequality holds:
\begin{equation} \label{ineq:tri}
a_i \le \sum \limits_{j \ne i} a_j.
\end{equation}
\end{enumerate}

When these conditions are satisfied, the matching can be done using a greedy algorithm, prioritizing the incident edges $e_i$ for which the number of available elevations $a_i$ is maximal. Assuming that the pieces involved are all of surface-type, such a matching yields a disjoint union of (possibly non-orientable) surfaces, each connected component being of positive genus (that is, containing an orientable subsurface of positive genus).
 \par \smallskip

The proof is divided into cases; in each case we assume that $v$ is a non-cyclic vertex of $\mathcal{G}$, and that $[t]$ is attached to a cyclic vertex $c$.

\begin{enumerate}[label=\textbf{Case \arabic*:}]
\item For completeness, and as a warm-up exercise, even though we assume that the cyclic vertices of $\mathcal{G}$ are of valence $2$, we prove the proposition for $c$ of valence $\ge 3$. For every non-cyclic vertex $v$ of $\mathcal{G}$ take an admissible subpair of weak surface-type $(\Sigma_v, \partial \Sigma_v)$ as described above. It is immediate that the hanging elevations can be matched two-by-two at every cyclic vertex other than $c$. Pick two hanging elevations $[t_1]$ and $[t_2]$ of $[t]$; since $\deg c \ge 3$, the collection of all hanging elevations of $[t]$ save $[t_1]$ and $[t_2]$ satisfy the triangle-like inequality \ref{ineq:tri}. Match them to copies of $G_c$ to obtain a (possibly disconnected) surface with two boundary components $[t_1]$ and $[t_2]$. Restricting to the connected component of $[t_1]$ (and taking a $2$-sheeted cover if it does not contain $[t_2]$) gives the desired result.

\item Suppose now that $v$ is rigid (and that $\deg c = 2$). Let $[w]=\{[w_1],\cdots,[w_\ell]\}$ be the peripheral structure induced by $\mathcal{G}$ at $G_v$, with $[w_1]=[t]$. Let $\Sigma'$ be a $k$-sheeted cover of the (possibly disconnected) surface $\Sigma_v$ whose boundary contains
\begin{enumerate}
    \item $k^2$ degree-$d$ elevations of each $[w_i]$ for $i\ne 1$,
    \item $k\cdot(k-1)$ degree-$d$ elevations of $[w_1]=[t]$, and
    \item $2$ degree-$k\cdot d$ elevations of $[w_1]=[t]$.
\end{enumerate}
Let $\Theta = \Sigma' \sqcup \bigsqcup_{i=2}^{\ell} \Sigma_v^i$; note that $\Theta$ is a disconnected surface containing $k\cdot (\ell-2+k)$ elevations of degree $d$ of each conjugacy class in $[\underline{w}]$, and two additional elevations of $[w_1]=[t]$ of degree $k\cdot d$. For every non-cyclic vertex $v'\ne v$, take $\ell -2 +k$ copies of $\Sigma_{v'}$. We leave it to the reader to match all of the degree-$d$ elevations in this collection of surfaces, two-by-two, and obtain a (connected) surface with two hanging elevations of $[t]$ of degree $k\cdot d$ (after potentially passing to a connected component, and taking a double-sheeted cover).

\item Finally, suppose that $v$ is of surface-type (and that $\deg c = 2$). In this case, since $\mathcal{G}$ is in normal form, the other non-cyclic vertex $v'$ adjacent to $c$ is rigid. Let $[u]=\{[u_1],\ldots,[u_m]\}$ be the peripheral structure induced by $\mathcal{G}$ at $G_{v'}$, such that $[u_1]$ is attached to $c$. We amass a collection of surfaces:
\begin{enumerate}
    \item For every non-cyclic vertex $x\ne v'$ of $\mathcal{G}$, a degree-$d$ cover $\Sigma'_x$ of $\Sigma_x$ with $k\cdot d$ degree-$d$ elevations of each conjugacy class in the peripheral structure induced by $\mathcal{G}$ at $G_x$.
    \item A degree-$d$ cover of $\Sigma_{v'}^1$ to which for every $1<i\le m$, $[u_i]$ admits $k\cdot d$ elevations of degree $d$.
    \item $\frac{k\cdot d - 2}{2}$ copies of the surface obtained by applying the previous case to $v'$ and $[u_1]$. Each of these surfaces has two boundary components, and each of them covers $[u_1]$ with degree $k\cdot d$.
\end{enumerate}
Once again, we leave it to the reader to match all of the hanging elevations, leaving only two elevations of $[t]$ hanging.
\end{enumerate}

\begin{figure}[h]
    \centering
    \includegraphics[width=\linewidth]{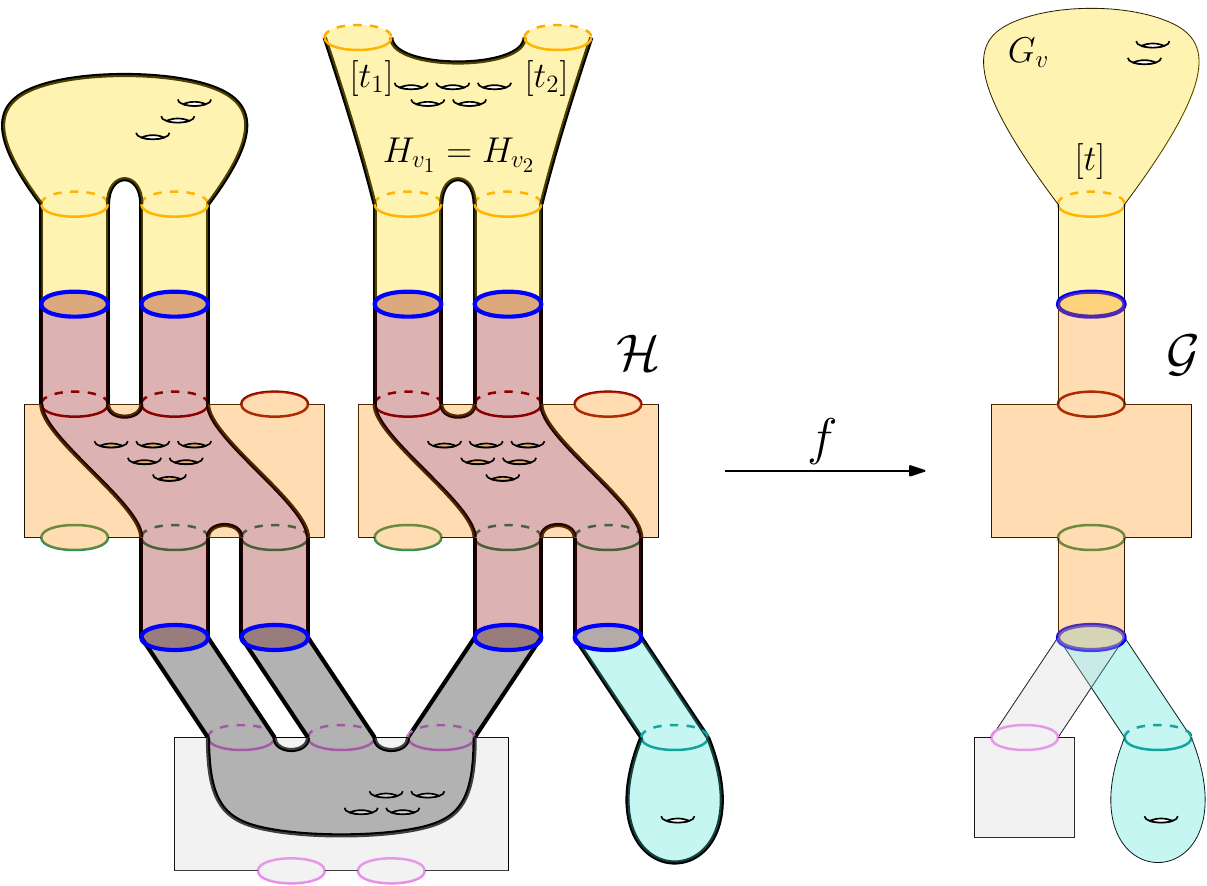}
    \caption{An example of the procedure described in \Cref{prop:surfaces-prescribed-peripherals}. Rigid pieces are represented as rectangles, and the surfaces inside the rectangles are the ones obtained using Wilton's \Cref{thm:Wilton-surfaces}. Cyclic vertices appear in bold blue. The surface $\mathcal{H}$, obtained by gluing several of these pieces, appears in bold, and has two boundary components (hanging elevations) $[t_1]$ and $[t_2]$.}
    \label{fig:surfaces-prescribed-peripherals}
\end{figure}

In each of the cases, if the resulting surface is orientable, the proof is complete. In this case, we have no control over the sign $\epsilon=\pm 1$ and both values can appear. If the resulting surface is non-orientable, we take an orientable double-cover $\Pi$ and obtain a surface with four boundary components $[s_1],[s_2],[s_3]$ and $[s_4]$, each covering $[t]$ with the same degree. By \Cref{lem:boundary-non-orient}, the boundary of the fundamental class of $\Pi$ is
\[
\partial[\Pi] = [s_1]-[s_2]+[s_3]-[s_4].
\]
We take a further $2$-sheeted cover $\Pi'$ of $\Pi$ that admits two boundary components $[t_1]$ and $[t_2]$ covering $[s_1]$ and $[s_2]$, respectively, with degree $2$, and four boundary components $[s_3'],[s_3''],[s_4']$ and $[s_4'']$ covering $[s_3]$ and $[s_4]$ with degree $1$. \par \smallskip

To finish, let $[u]$ be a conjugacy class that can be matched with $[t]$ at a cyclic vertex of $\mathcal{G}$. Repeat the construction for $[u]$ to obtain a surface with two boundary components covering $[u]$. Next, pass to a $2$-sheeted (orientable) cover whose boundary contains four components, and such that the boundary of its fundamental class is given by $[u_1]-[u_2]+[u_3]-[u_4]$. Glue this surface to $\Pi'$, preserving orientability, to obtain $\mathcal{H}$. Note that $H$ is a surface with two boundary components $[t_1]$ and $[t_2]$, and that $\partial [H] = [t_1]-[t_2]$, as desired.
\end{proof}

Finally, \Cref{rmk:propagating-non-orientability} can be obtained in a fashion similar to \Cref{lem:standard-pre-covers}. Suppose we have obtained a surface with $\epsilon=-1$; we take a finite-sheeted cover, and glue some of its boundary components to another surface, in such a way that the resulting surface is non-orientable. Then, we can ``propagate'' this non-orientability along the graph of groups, as gluing a non-orientable surface to any surface yields a non-orientable surface. Finally, whenever we have a non-orientable surface to which $[t]$ admits two hanging elevations, it can be used to construct an orientable one with $\epsilon=-1$, as described in the last paragraph of the proof of \Cref{prop:surfaces-prescribed-peripherals}.

\subsection{Artificial branching within a rigid vertex}

We next produce a pair of overlapping surfaces within a rigid vertex, from which we assemble the artificial branching block that imitates the ``genuine'' branching behaviour of a cyclic vertex of valence $\ge 3$, as in \Cref{prop:branched-torsion}. The construction exploits the strong one-endedness of rigid vertex groups, makes use of Wilton’s surfaces, and relies on surfaces guaranteed by Calegari’s Rationality Theorem to carry out a minimality argument.

\begin{prop} [Artificial branching] \label{prop:artificial-branching}
    Let $(F,[\underline{w}])$ be a malnormal rigid pair, different from a pair of pants. Then, there exists $[w_1]\in [\underline{w}]$ and a finite-index pull-back pair $(H,[\underline{u}])\lepair (F,[\underline{w}])$ satisfying the following properties:
    \begin{enumerate}
        \item There are two embedded subpairs $(H_1,[\underline{u}_1]),(H_2,[\underline{u}_2])\embedpair (H,[\underline{u}])$ of surface-type.
        \item For some $d\ge 1$, $w_1^d\in H_1\cap H_2$; moreover, $[w_1^d]\in[\underline{u}_1]$ and $[w_1^d]\in[\underline{u}_2]$.
        \item $\pcl{[\underline{w}]}{[\underline{u}_1]}\setminus\pcl{[\underline{w}]}{[\underline{u}_2]} \ne \emptyset$ and $\pcl{[\underline{w}]}{[\underline{u}_2]}\setminus\pcl{[\underline{w}]}{[\underline{u}_1]} \ne \emptyset$.
        \item\label{itm:minimality} There is no subpair $(\Sigma, \partial \Sigma)\lepair (H,[\underline{u}])$ of surface-type such that 
        \[
        \pcl{[\underline{w}]}{\partial \Sigma}\subsetneq \pcl{[\underline{w}]}{[\underline{u}_1]}\;\;\text{ or }\;\;\; [w_1]\in \pcl{[\underline{w}]}{\partial \Sigma}\subsetneq \pcl{[\underline{w}]}{[\underline{u}_2]}.
        \]
    \end{enumerate}
\end{prop}

\begin{rmk}
    Each of the two surface-type pairs $(H_1,[\underline{u}_1])$ and $(H_2,[\underline{u}_2])$ can be assumed to have an even number of boundary components, each of which being a degree-$d$ elevation of a conjugacy class in $[\underline{w}]$ for the same $d$ obtained in \Cref{prop:surfaces-prescribed-peripherals}.
\end{rmk}

\begin{proof}
    By Wilton's \Cref{thm:rigid-strongly-one-ended}, up to passing to a finite-index pull-back pair, we may assume that $(F,[\underline{w}])$ is strongly one-ended; by the observation appearing after \Cref{thm:rigid-strongly-one-ended}, we may also assume that each $[w_i]\in [\underline{w}]$ is the conjugacy class of a primitive element. Consider the family
    \[
    \mathcal{S} = \{(\Sigma,\partial \Sigma) \lepair (F,[\underline{w}])\;\;\vert\;\; (\Sigma,\partial \Sigma) \text{ is of surface-type}\}.
    \]
    Wilton's \Cref{thm:Wilton-surfaces} implies that $\mathcal{S}$ non-empty. Take $(\Sigma_1,\partial \Sigma_1) \in \mathcal{S}$ such that $\pcl{[\underline{w}]}{\partial \Sigma_1}$ is minimal by inclusion; choose $[w_1] \in \pcl{[\underline{w}]}{\partial \Sigma_1}$. Since $w_1$ is primitive, the collection of elevations of $w_1$ to $\Sigma_1$ form a part of a common basis (see \Cref{lem:part-of-a-basis}). It follows that these elevations cannot bound a surface, and thus $\pcl{[\underline{w}]}{\partial \Sigma _1}\ne \{[w_1]\}$. Without loss of generality, we assume that $[w_1],[w_2]\in\pcl{[\underline{w}]}{\partial \Sigma_1}$. \par \smallskip

    Consider the following subfamily of $\mathcal{S}$:
    \[
    \mathcal{S}'=\{(\Sigma,\partial \Sigma) \in \mathcal{S} \;\;\vert\;\; [w_1] \in \pcl{[\underline{w}]}{\partial \Sigma} \;\; \text{ and } \;\; [w_2]\notin\pcl{[\underline{w}]}{\partial \Sigma}\}.
    \]
    Since $(F,[\underline{w}])$ is strongly one-ended, once again Wilton's \Cref{thm:Wilton-surfaces} implies that $\mathcal{S}'\ne \emptyset$. Take $(\Sigma_2,\partial \Sigma_2) \in \mathcal{S}'$ such that $\pcl{[\underline{w}]}{\partial \Sigma_2}$ is minimal by inclusion (in $\mathcal{S}'$). Since $[w_2]$ lies outside the peripheral closure of $\partial \Sigma_2$, and since $\pcl{[\underline{w}]}{\partial \Sigma_1}$ is minimal in $\mathcal{S}$, we have that both
    \[
    \pcl{[\underline{w}]}{\partial \Sigma_1} \setminus \pcl{[\underline{w}]}{\partial \Sigma_2} \ne \emptyset \;\; \text{ and }\;\;
    \pcl{[\underline{w}]}{\partial \Sigma_2} \setminus \pcl{[\underline{w}]}{\partial \Sigma_1} \ne \emptyset.
    \]

    Suppose for a contradiction that there exists $(\Sigma, \partial \Sigma) \in \mathcal{S}$ whose peripheral closure is strictly contained in either that of $(\Sigma_1, \partial \Sigma_1)$ or of $(\Sigma_2, \partial \Sigma_2)$. We clearly cannot have that $\pcl{[\underline{w}]}{\partial \Sigma} \subsetneq \pcl{[\underline{w}]}{\partial \Sigma_1}$ by minimality. Similarly, if $\pcl{[\underline{w}]}{\partial \Sigma} \subsetneq \pcl{[\underline{w}]}{\partial \Sigma_2}$ then in particular $[w_2]\notin \pcl{[\underline{w}]}{\partial \Sigma}$ and $(\Sigma, \partial \Sigma)\in \mathcal{S}'$, contradicting minimality (in $\mathcal{S}'$). We deduce that condition \emph{(3)} is satisfied. \par \smallskip
    
    Finally, by taking conjugates and finite index subgroups of $\pi_1(\Sigma_1)$ and $\pi_1(\Sigma_2)$ we can ensure that $w_1^d\in \pi_1(\Sigma_1) \cap \pi_1(\Sigma_2)$ and that $[w_1^d]$ belongs to both $\partial \Sigma_1$ and $\partial \Sigma_2$ for some $d\ge1$. The conclusion now follows from \Cref{prop:two-surfaces-in-one-subgroup}.
\end{proof}

As discussed in our outline of the strategy behind the proof of \Cref{mainthm}, controlling the image in the abelianization of the boundary element $[w_1^d]$ shared by the two surfaces obtained via \Cref{prop:artificial-branching} is crucial for producing homological torsion. The following lemma, which we view as the key property of the artificial branching construction, makes this assertion precise:

\begin{lem}\label{lem:key-property}
    In the setting of \Cref{prop:artificial-branching}, write $[\underline{u}_1]\cap [\underline{u}_2]=\{[s_1],[s_2],\ldots,[s_r]\}$. Then $s_1,\dots,s_r$ generate a rank-$r$ direct summand in the free abelian group $H^\ab$.
\end{lem}
\begin{proof}
    Striving for a contradiction, suppose that a non-trivial equation $\lambda_1 \cdot s_1+\lambda_2\cdot s_2+\cdots+\lambda_r\cdot s_r=0$ holds true in $H^\ab$. Then, by Calegari's Rationality Theorem (see \Cref{cor:Calegari}), there exists a surface-type subpair $(\Sigma,\partial \Sigma)\le (H,[\underline{u}])$ with 
    \[
        \pcl{[\underline{w}]}{\partial \Sigma}\subseteq\pcl{[\underline{w}]}{[\underline{u}_1]}\cap\pcl{[\underline{w}]}{[\underline{u}_2]} \subsetneq \pcl{[\underline{w}]}{[\underline{u}_1]}.
    \]
    This contradicts \cref{itm:minimality} of \Cref{prop:artificial-branching}.
\end{proof}

We are now ready to prove the main \Cref{thm:main}. 

\begin{proof}[Proof of \Cref{thm:main}]
    Recall \Cref{assumption}. In particular, we assume that every cyclic vertex in $\mathcal{G}$ is of valence $2$, and that there is at least one rigid vertex $v$ in $\mathcal{G}$. Denote the pair induced by $\mathcal{G}$ at $G_v$ by $(F=G_v,[\underline{w}])$, and since $\mathcal{G}$ is in normal form, $(F,[\underline{w}])$ is not a pair of pants. \par \smallskip

    \Cref{prop:artificial-branching} gives rise to a finite-index pull-back pair $(H,[\underline{u}])\le (F,[\underline{w}])$, accompanied by two embedded surface-type subpairs
    \[
    (H_1,[\underline{u}_1]),(H_2,[\underline{u}_2])\embedpair(H,[\underline{u}]).
    \]
    $[w_1^d]$ is a boundary component of both $H_1$ and $H_2$, and it satisfies the key property of \Cref{lem:key-property}. In preparation for the construction that mimics the proof of \Cref{prop:branched-torsion}, we amass a collection of surfaces with boundary: each $[w_i]\in [\underline{w}]$ is attached to a cyclic vertex of $\mathcal{G}$, and there exists a unique conjugacy class $[s]$ (different from $[w_i]$) that is also attached to this cyclic vertex. \Cref{prop:surfaces-prescribed-peripherals} equips us with a surface $\Sigma_{[w_i]}$ with two boundary components $[s_1]$ and $[s_2]$ which are degree-$d$ elevations of $[s]$, and such that the boundary of the fundamental class of $\Sigma_{[w_i]}$ is $\partial[\Sigma_{[w_i]}]=[s_1]+\epsilon \cdot [s_2]$ for some $\epsilon \in \{\pm1\}$. By \Cref{rmk:propagating-non-orientability}, $\epsilon$ can be chosen uniform across all such surfaces. In what follows, if $x$ denotes an elevation of $[w_i]$, we also use $\Sigma_x$ to denote $\Sigma_{[w_i]}$.
    \par \smallskip

    We also fix some notation for conjugacy classes in the abelian group $H^{\ab}$, which will serve as variables in a system of $\partial$-equations:
    \begin{enumerate}
        \item $x$ is the element corresponding to $[w_1^d]$.
        \item $s_1,\ldots,s_a$ correspond to the other boundary components shared between $H_1$ and $H_2$.
        \item $t_1,\ldots,t_b$ correspond to the boundary components of $H_1$ that are not boundary components of $H_2$.
        \item $p_1,\ldots, p_c$ correspond to the boundary components of $H_2$ that are not boundary components of $H_1$.
    \end{enumerate}
    The surfaces $H_1$ and $H_2$ give rise to a pair of $\partial$-equations in $H^{\ab}$,
    \[
    x\pm s_1\pm\dots\pm s_a+t_1+\dots+t_b=0 \;\; \text{ and }\;\; x\pm s_1\pm\dots\pm s_a+p_1+\dots+p_c=0,
    \]
    respectively. Note that the signs of each $s_i$ may not be the same in both equations. Finally, by \Cref{lem:key-property} $x$ generates an infinite cyclic direct summand in the quotient $H^{\ab}/\gen{s_1,\ldots,s_a}$. \par \smallskip

    We next assemble the building blocks which will be used in the construction; these building blocks appear in \Cref{fig:pieces}, and are counterparts of the three surfaces and two cyclic vertices that appear in a standard branched surface (see \Cref{def:standard}). 
    \begin{enumerate}
        \item \textbf{The artificial branching block }$X$. This block replaces the cyclic vertices $y$ and $z$ in the proof of \Cref{prop:branched-torsion}. Take two copies $H'$ and $H''$ of $H$. $X$ is obtained by attaching, for each $1\le i \le a$, the surface $\Sigma_{s_i}$ (constructed in the beginning of the proof) to $s_i'$ and $s_i''$. \par \smallskip
        Write 
        \begin{equation}\label{boundary_notation}            
        \tilde{x}=x'+\epsilon x'',\;\; \tilde{s}_i=s'_i+s''_i, \;\; \tilde{t}_i=t'_i+t''_i \;\;\text{ and } \;\; \tilde{p}_i=p'_i+p''_i.
        \end{equation}
        Each $\Sigma_{s_i}$ gives rise to a $\partial$-equation $s_i'+\epsilon s_i''=\tilde{s}_i=0$ in $\pi_1(X)^{\ab}$, so that
        \[
        X^\ab\cong (H')^\ab\oplus (H'')^\ab/\gen{s_1,\ldots,s_a}.
        \]
        Moreover, $\tilde{x}$ generates an infinite direct summand of $\pi_1(X)^{\ab}$, and appears in the following two $\partial$-equations which hold true in $\pi_1(X)^{\ab}$:
        \[
        \tilde{x}+\tilde{t}_1+\cdots+\tilde{t}_b=0\;\;\text{ and }\;\; \tilde{x}+\tilde{p}_1+\cdots+\tilde{p}_c=0.
        \]
        \item \textbf{The surface }$\Sigma$, is simply the surface $\Sigma_x$. Note that gluing $\Sigma$ to $X$ identifies $\partial[\Sigma]$ with $\tilde{x}$.
        \item \textbf{The surface }$\Theta$. Take two copies $H_1^\dagger$ and $H_1^\ddagger$ of $H_1$, and copies of the surfaces $\Sigma_x,\Sigma_{s_1},\ldots, \Sigma_{s_a}$. Glue $\Sigma_x$ and each $\Sigma_{s_i}$ to $H_1^\dagger$ and $H_1^\ddagger$, so that the following $\partial$-equations are satisfied in the abelianization,
        \[
        x^\dagger+\epsilon x ^\ddagger=0,\;\;s_1^\dagger + \epsilon s_1^\ddagger=0,\ldots,s_a^\dagger + \epsilon s_a^\ddagger = 0.
        \]
        Finally, for each $1\le i \le b$, let $T_i$ be a $2$-sheeted cover of $\Sigma_{t_i}$ with four boundary components. $\Theta$ is obtained by gluing two of the boundary components of each $T_i$ to $t_i^\dagger$ and $t_i^\ddagger$ (choosing orientations so that $\Theta$ is orientable). Note that $\Theta$ has $2\cdot b$ boundary components, and that gluing them to $X$ identifies $\partial[\Theta]$ with $\tilde{t}_1+\cdots+\tilde{t}_b$.
        \item \textbf{The surface }$\Pi$ is constructed like $\Theta$, replacing $H_1$ with $H_2$. $\Pi$ has $2\cdot c$ boundary components, and gluing them to $X$ identifies $\partial [\Pi]$ with $\tilde{p}_1+\cdots+\tilde{p}_c$.
    \end{enumerate}

\begin{figure}[h!]
    \centering
    \includegraphics[width=\linewidth]{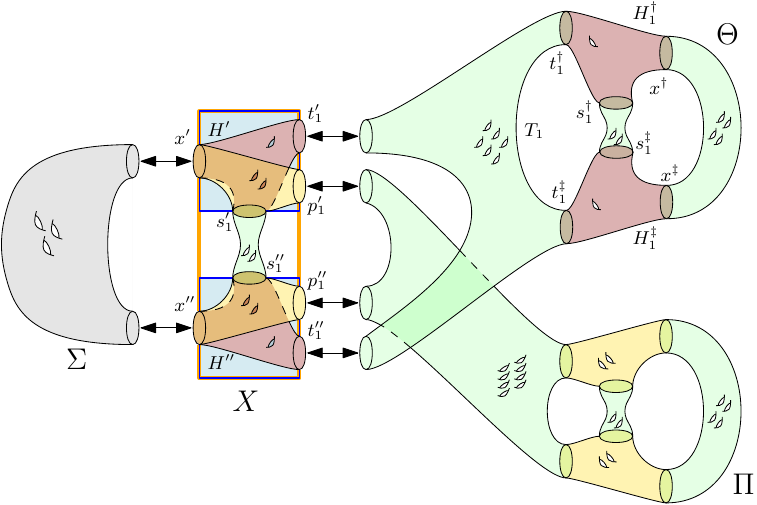}
    \caption{The artificial branching block $X$ (highlighted in bold orange) and the surfaces $\Sigma$, $\Theta$ and $\Pi$, arranged in the configuration of a standard branched surface (compare with \Cref{fig:standard-branched-surface}). Copies of $H$ are blue squares, and copies of the surfaces $H_1$ and $H_2$ appear, respectively, in red and yellow. Surfaces obtained by means of \Cref{prop:surfaces-prescribed-peripherals} are green (except for $\Sigma$). }
     \label{fig:pieces}
\end{figure}

    With the pieces $X,\Sigma, \Theta$ and $\Pi$ in place, we mimic the proof of \Cref{prop:branched-torsion}. To further simplify notation, following \ref{boundary_notation} we write
    \[
    \tilde{t}=\tilde{t_1}+\cdots + \tilde{t}_b\;\;\text{ and }\;\; \tilde{p}=\tilde{p}_1+\cdots+\tilde{p}_c.
    \]
    Recall that $\tilde{x}$ generates an infinite cyclic direct summand of $\pi_1(X)^{\ab}$, and that $\tilde{x}+\tilde{t}=\tilde{x}+\tilde{p}=0$. Also, the boundaries of $\Sigma$, $\Theta$ and $\Pi$ may be glued, respectively, to $\tilde{x}$, $\tilde{t}$ and $\tilde{p}$ in $X$. \par \smallskip

    Denote $M=\bbZ/r_1\bbZ\oplus\dots\oplus\bbZ/r_n\bbZ$ for some integers $r_1,\dots,r_n\ge2$. As in the proof of \Cref{prop:branched-torsion}, we assemble the following pieces:
    \begin{enumerate}
        \item Copies $X_{i,j}$ of $X$ for $i=1,\dots,n$ and $j=1,\dots,3r_i-2$. For each copy $X_{i,j}$ of $X$, we denote by $\tilde{x}_{i,j}$, $\tilde{t}_{i,j}$ and $\tilde{p}_{i,j}$ the corresponding elements in $\pi_1(X_{i,j})^\ab$.
        \item Connected covers $\Sigma_{i,0}$ of $\Sigma$ for $i=1,\ldots,n$, so that $\Sigma_{i,0}$ has $2\cdot r_i$ boundary components covering $x$ with degree $1$.
        \item Connected $2$-sheeted covers $\Sigma_{i,j}$ of $\Sigma$, for $i=1,\ldots,n$ and $j=1,\ldots,r_i-1$, each with four boundary components, each covering $x$ with degree $1$.
        \item Connected $2$-sheeted covers $\Theta_{i,j}$ of $\Theta$, for $i=1,\ldots, n$ and $j=1,\ldots,r_i-1$, each with $4\cdot b$ boundary components.
        \item Connected $2$-sheeted covers $\Pi_{i,j}$ of $\Pi$, for $i=1,\ldots,n$ and $j=1,\ldots,r_i-1$, each with $4\cdot c$ boundary components.
        \item Connected covers $\Theta_0$ and $\Pi_0$ of $\Theta$ and $\Pi$, respectively, which will be explicitly determined later.
    \end{enumerate}
    For a given $1\le i \le n$, following \Cref{tab:torsion-summand}, the surfaces $\Sigma_{i,0},\Sigma_{i,j},\Theta_{i,j}$ and $\Pi_{i,j}$ are glued to the artificial branching blocks $X_{i,j}$ producing the following $\partial$-equations in the abelianization:
    \[
        \partial[\Sigma_{i,0}]=\tilde{x}_{i,1}+\dots+\tilde{x}_{i,r_i},
    \]
    \[
        \partial[\Sigma_{i,j}]=\tilde{x}_{i,r_i+2j-1}+\tilde{x}_{i,r_i+2j},
    \]
    \[
        \partial[\Theta_{i,j}]=\tilde{t}_{i,r_i+2j-2}+\tilde{t}_{i,r_i+2j-1},\text{ and}
    \]
    \[
        \partial[\Pi_{i,j}]=\tilde{p}_{i,j}+\tilde{p}_{i,r_i+2j-1}.
    \]
    Finally, $\Theta_0$ and $\Pi_0$ are chosen so that their boundary covers all of the hanging elevations of $t_1,\ldots,t_b,p_1,\ldots,p_c$ across all artificial branching blocks $X_{i,j}$. \par \smallskip

    One easily verifies that the the resulting precover $\mathcal{K}$ of $\mathcal{G}$ is connected. The abelianization of its fundamental group $K$ is given by (where $i$ and $j$ range over $i=1,\dots,n$ and $j=1,\dots,r_i-1$):
    \[
        K^{\ab}=\frac{\bigoplus_{i,j}X_{i,j}^\ab}{\gen{\partial[\Sigma_{i,0}],\partial[\Sigma_{i,j}],\partial[\Theta_{i,j}],\partial[\Pi_{i,j}],\partial[\Theta_0],\partial[\Pi_0]}} \oplus \mathbb{Z}^\ell,
    \]
    for some $\ell \in \mathbb{N}$. Using the $\partial$-equations $\tilde{x}_{i,j}+\tilde{t}_{i,j}=\tilde{x}_{i,j}+\tilde{p}_{i,j}=0$, as well as the fact that the $\partial$-equations given by $\partial[\Theta_0]$ and $\partial[\Pi_0]$ are redundant, reduces the above to
    \[
        K^{\ab}=\frac{\bigoplus_{i,j}X_{i,j}^\ab}{\gen{\tilde{x}_{i,r_i+2j-1}+\tilde{x}_{i,r_i+2j},\tilde{x}_{i,r_i+2j-2}+\tilde{x}_{i,r_i+2j-1},\tilde{x}_{i,j}+\tilde{x}_{i,r_i+2j-1},\tilde{x}_{i,1}+\dots+\tilde{x}_{i,r_i}}} \oplus \mathbb{Z}^\ell.
    \]
    Lastly, the subgroup of $K^{\ab}$ generated by all the $\tilde{x}_{i,j}$ is a direct summand isomorphic to
    \[
        \frac{\bigoplus_{i,j}\mathbb{Z}\tilde{x}_{i,j}}{\gen{\tilde{x}_{i,r_i+2j-1}+\tilde{x}_{i,r_i+2j},\tilde{x}_{i,r_i+2j-2}+\tilde{x}_{i,r_i+2j-1},\tilde{x}_{i,j}+\tilde{x}_{i,r_i+2j-1},\tilde{x}_{i,1}+\dots+\tilde{x}_{i,r_i}}},
    \]
    which, following the computation in the proof of \Cref{prop:branched-torsion}, contains a direct summand $\bbZ/r_1\bbZ\oplus \cdots \oplus\bbZ/r_n\bbZ\cong M$. We remark that the generator of the summand $\bbZ/r_i\bbZ$ is $\tilde{x}_{i,1}$. \par \smallskip

    Finally, we observe that there is a natural map from a branched surface $\mathcal{B}_M$ (with $H_1(B_M)=M\oplus \mathbb{Z}^r$) to $\mathcal{K}$. Here, $\mathcal{B}_M$ is the branched surface obtained from $\mathcal{K}$ by replacing the copies of the rigid piece $H$ in the artificial branching blocks $X_{i,j}$ with copies of the surfaces $H_1$ and $H_2$ (intersecting only at $\tilde{x}$ and $\tilde{s})$. \Cref{cor:vr-direct-factor} yields a finite-index subgroup $G_0\le G$ that retracts onto $K$, finishing the proof.
\end{proof}

\section{Applications to profinite rigidity} \label{sec:profinite}

In this section, we explain how \Cref{mainthm} can be applied to the study of profinite rigidity. As mentioned in the \nameref{intro}, the collection of abelianizations of all finite-index subgroups of a finitely generated, residually finite group is a profinite invariant. It therefore follows almost immediately from \Cref{mainthm} that free products of free groups and (closed) surface groups are profinitely rigid among hyperbolic graphs of free groups with cyclic edges. The lion's share of the proof of \Cref{maincorprof} below is identical to the proof of Wilton and Zalesskii in \cite[Theorem A]{wz:decompositions}, which shows that the profinite completion of the fundamental group of a closed, orientable 3-manifold determines its Kneser-Milnor decomposition (Wilkes extended this result to non-closed 3-manifolds with incompressible boundary in \cite[Theorem 6.22]{wilkes:decompositions}). Their proof uses a profinite version of Poincar\'e duality; for a thorough treatment of profinite Poincaré duality groups, we refer the reader to \cite[Section 1.2]{wz:decompositions}.

\begin{thm}[\Cref{maincorprof}]
    \label{thm:pf-rigid}
    Let $G$ be a free product of finitely many free groups and fundamental groups of compact (possibly non-orientable) surfaces, and let $H$ be a hyperbolic group that splits as a graph of free groups amalgamated along cyclic subgroups. If $\widehat{G} \cong \widehat{H}$ then $G\cong H$.
\end{thm}

Before turning to the proof of \Cref{thm:pf-rigid}, we note that this result forms part of a broader effort to approach Remeslennikov’s question—and related problems—by progressively distinguishing free, surface, and closely related groups from other groups that share many of their structural features. A natural next step would be to prove that free products of free and surface groups are profinitely rigid within the natural class of finitely generated, residually free groups. While we expect this to be true, our methods do not currently yield such a result. In earlier work, the second author and Morales showed that direct products of free and surface groups are profinitely rigid among finitely presented, residually free groups \cite[Theorem F]{fru24}. The case of free products, however, seems to present significantly greater challenges.

\begin{proof}[Proof of \Cref{thm:pf-rigid}]
    Suppose for a contradiction that $H$ is not a free product of free and surface groups; by \Cref{mainthm}, $H$ admits a finite-index subgroup $H'\le H$ such that $\mathbb{Z}/3\mathbb{Z}$ is a direct summand of $(H')^{\ab}$. Since $\widehat{G} \cong \widehat{H}$, the same holds for $G$, that is, there is a finite-index subgroup $G'\le G$ such that $\mathbb{Z}/3\mathbb{Z}$ is a direct summand of $(G')^{\ab}$. But this is absurd: since $G'$ itself is a free product of free and surface groups, $\mathrm{Tor}(G')^{\ab}\cong \bigoplus_{i=1}^\ell \mathbb{Z}/2\mathbb{Z}$. 

    Write $G=\Asterisk_{i=1}^n S_i \ast F$ where each $S_i$ is a (non-trivial, closed) surface group and $F$ is a finitely generated free group. Assume first that each $S_i$ is orientable. By the previous paragraph, we have that $H=\Asterisk_{i=1}^m S'_i \ast F'$ where each $S'_i$ is a (non-trivial, closed) surface group and $F'$ is a finitely generated free group. Since $H$ is hyperbolic, subgroup separable and torsion free, none of the $S'_i$ is the fundamental group of a torus. In addition, $\widehat{G} \cong \widehat{H}$ implies that $H^{\ab}$ is torsion-free, and therefore each $S'_i$ is orientable. Therefore, each one of $S_i$ and $S'_i$ is a Poincar\'e duality group of dimension $2$. We obtain that whenever $S_i$ or $S'_i$ acts on a profinite tree with trivial edge stabilizers, it must fix a vertex: this follows from the fact that surface groups of negative Euler characteristic are cohomologically good \cite[Proposition 3.7]{gjz:good} combined with \cite[Corollary 1.11]{wz:decompositions}. \par \smallskip

    Wilton's and Zalesskii's argument from \cite[Theorem A]{wz:decompositions} applies to our setting. We briefly sketch their proof for completeness. Let $T$ be the Bass-Serre tree corresponding to the free splitting of $H$, and let $\widehat{T}$ be the corresponding profinite tree on which $\widehat{H}$ acts. $\widehat{G}$ also acts on this tree, and each $\widehat{S}_i$ must fix a vertex of $\widehat{T}$ and is therefore conjugate into some $\widehat{S}'_j$. Reversing the roles of $H$ and $G$ we get that every $\widehat{S}'_i$ is conjugate into some $\widehat{S}_j$. This implies that $n=m$, and up to reordering indices, $\widehat{S}_i \cong \widehat{S}'_i$. Since surface groups are distinguished from one another by their abelianizations, we have that $S_i \cong S'_i$. Finally, quotienting out each $\langle \langle \widehat{S}_i \rangle \rangle$ in $\widehat{G}$ and each $\langle \langle \widehat{S}'_i \rangle \rangle$ in $\widehat{H}$, we obtain that $\widehat{F} \cong \widehat{F'}$ which implies $F\cong F'$. \par \smallskip

    Suppose now that some of the $S_i$ are non-orientable, say, $S_1,\ldots,S_k$ are non-orientable and $S_{k+1},\ldots,S_n$ are orientable. Comparing the torsion part of $G^{\ab}$ and $H^{\ab}$ tells us that $m\ge k$ and that, up to reordering, $S'_1,\ldots,S'_k$ are non-orientable (while $S'_{k+1},\ldots,S'_m$ are orientable). Let $G'$ be the unique index-$2$ subgroup of $G$ such that $G'^{\ab}$ is torsion-free, and such that $G'$ has the maximal number of surface subgroups as free factors in its Grushko decomposition; denoting the orientable double cover $S_i$ by $T_i$ for $1\le i \le k$, we have that 
    \begin{equation} \label{surf_decomp}
    G'\cong T_1\ast \cdots \ast T_k \ast S_{k+1} \ast S_{k+1} \ast \cdots \ast S_n \ast S_n \ast F \ast F \ast F_{k-1}.
    \end{equation}
    Let $H'$ be the corresponding index-$2$ subgroup of $H$. It follows from the orientable case that $G'\cong H'$. Note in addition that $H'$ must be unique index-$2$ subgroup of $H$ that has a torsion-free abelianization, and such that it has the maximal number of surface subgroups as free factors in its Grushko decomposition: otherwise, there would exist an index-$2$ subgroup $G''$ of $G$ such that the number of surface subgroups in the Grushko decomposition of $G''$ is greater than that of $G'$. Since the numbers of non-orientable surfaces in $G$ and $H$ coincide, this tells us that $n=m$. \par \smallskip
    To finish, consider the $n-k$ index-$2$ subgroups of $G$, each obtained from the expression in \ref{surf_decomp} by replacing some $S_i \ast S_i$ with a double cover $T_i$ of $S_i$ for some $k<i\le n$ (and $F_{k-1}$ with $F_k$). The abelianization of each of these subgroups is torsion-free. Comparing these to the corresponding index-$2$ subgroups of $H$ as in the orientable case, shows that $G\cong H$.
    \end{proof}

\begin{rmk}
    We should point out that a similar argument implies that virtual free products of free and surface groups are profinitely recognized among hyperbolic graphs of virtually free groups with virtually cyclic edges. Indeed, if $G$ is virtually a free product of free and surface groups, then $G$ has a finite-index subgroup $G'\le G$ such that for every finite-index $H\le G$, the torsion part of $H^{\ab}$ is a direct sum of copies of $\mathbb{Z}/2\mathbb{Z}$. By \Cref{mainthm}, $H$ must also take a similar form. However, this result can not be strengthened to obtain a result similar to \Cref{thm:pf-rigid}: Bessa, Grunewald and Zalesskii gave examples of non-isomorphic virtually surface groups with the same profinite completion \cite[Section 2]{bessa:virsurface}.
\end{rmk}

Recall that as mentioned in the \nameref{intro}, a word $w$ in a free group $F_k$ is \emph{profinitely rigid} (in $F_k$) if its $\mathrm{Aut}(\widehat{F}_k)$-orbit in $\hat{F}_k$ meets $F_k$ precisely in its $\mathrm{Aut}(F_k)$ orbit, that is,
\[\mathrm{Aut}(\widehat{F}_k).w \cap F = \mathrm{Aut}(F_k).w.\]
\Cref{thm:pf-rigid} above implies that \emph{partial surface words} in $F_k$ are profinitely rigid, giving $\lceil 3k/2 \rceil -4$ new examples of non-power words which are profinitely rigid in $F_k$. 
\begin{cor}[\Cref{maincorwords}]
\label{cor:partial_words}
Let $w\in F_k=\langle x_1,\ldots,x_k \rangle$ be a partial surface word, that is $w$ has one of the following two forms:
\begin{enumerate}
    \item $w=[x_1,x_2]\cdots[x_{2n-1},x_{2n}]$ for $2n < k$, or
    \item $w=x_1^2\cdots x_n^2$ for $n < k$.
\end{enumerate}
Then $w$ is profinitely rigid in $F_k$. 
\end{cor}

\begin{rmk}
    \cite[Theorem 1.7]{puder:commutator} implies that powers of partial surface words in $F_k$ are profinitely rigid. Moreover, \cite[Claim 2.5]{puder:commutator} shows that the automorphic orbits of partial surface words in $F_k$ are closed in the profinite topology on $F_k$.
\end{rmk}

The proof of \Cref{cor:partial_words} is essentially identical to \cite[Corollary 4]{wilton:words} (and \cite[Corollary E]{Wil18}), and relies on the fact that if $w'\in F_k$ lies in the $\mathrm{Aut}(\widehat{F}_k)$-orbit of $w$, then the profinite double $\widehat{F_k \ast_{w'} F_k}$ is isomorphic to $\widehat{F_k \ast_{w} F_k}$. Since $F_k \ast_{w} F_k$ is a free product of a free and a surface group, \Cref{thm:pf-rigid} yields the desired result. 

\bibliographystyle{plain}

\begin{thebibliography}{10}

\bibitem{agol}
I.~Agol.
\newblock The virtual haken conjecture (with an appendix by {I}an {A}gol, {D}aniel {G}roves and {J}ason {M}anning).
\newblock {\em Documenta Mathematica}, 18:1045–1087, 2013.

\bibitem{amit:measure_conj}
A.~Amit and U.~Vishne.
\newblock Characters and solutions to equations in finite groups.
\newblock {\em Journal of Algebra and Its Applications}, 10(04):675–686, August 2011.

\bibitem{Bergeron2012}
N.~Bergeron and A.~Venkatesh.
\newblock The asymptotic growth of torsion homology for arithmetic groups.
\newblock {\em Journal of the Institute of Mathematics of Jussieu}, 12(02):391–447, June 2012.

\bibitem{bessa:virsurface}
V.~Bessa, F.~Grunewald, and P.~A. Zalesskii.
\newblock Genus for virtually surface groups and pullbacks.
\newblock {\em Manuscripta Mathematica}, 145(1–2):221–233, May 2014.

\bibitem{wise:local}
H.~Bigdely and D.~T Wise.
\newblock Quasiconvexity and relatively hyperbolic groups that split.
\newblock {\em Michigan Mathematical Journal}, 62(2):387--406, 2013.

\bibitem{cal:subs}
D.~Calegari.
\newblock {Surface subgroups from homology}.
\newblock {\em Geometry \& Topology}, 12(4):1995 -- 2007, 2008.

\bibitem{Cal09}
D.~Calegari.
\newblock {\em Stable Commutator Length}, volume~20 of {\em Mathematical Society of Japan Memoirs}.
\newblock Mathematical Society of Japan, Tokyo, 2009.

\bibitem{CM11}
C.~G. Cashen and N.~Macura.
\newblock Line patterns in free groups.
\newblock {\em Geometry \& Topology}, 15(3):1419--1475, 2011.

\bibitem{Cas16}
C.~H. Cashen.
\newblock Splitting line patterns in free groups.
\newblock {\em Algebraic \& Geometric Topology}, 16(2):621--673, 2016.

\bibitem{Chu2022}
M.~Chu and D.~Groves.
\newblock Prescribed virtual homological torsion of 3-manifolds.
\newblock {\em Journal of the Institute of Mathematics of Jussieu}, 22(6):2931–2941, June 2022.

\bibitem{DF05}
G.~Diao and M.~Feighn.
\newblock The grushko decomposition of a finite graph of finite rank free groups: An algorithm.
\newblock {\em Geometry \& Topology}, 9(4):1835--1880, 2005.

\bibitem{fru24}
J.~Fruchter and I.~Morales.
\newblock Virtual homology of limit groups and profinite rigidity of direct products.
\newblock {\em Israel Journal of Mathematics}, 2024.
\newblock Accepted for publication.

\bibitem{garrido:ff}
A.~Garrido and A.~Jaikin-Zapirain.
\newblock Free factors and profinite completions.
\newblock {\em International Mathematics Research Notices}, 2023(24):21320–21345, November 2022.

\bibitem{gjz:good}
F.~Grunewald, A.~Jaikin-Zapirain, and P.~A. Zalesskii.
\newblock Cohomological goodness and the profinite completion of bianchi groups.
\newblock {\em Duke Mathematical Journal}, 144(1), July 2008.

\bibitem{GL11}
V.~Guirardel and G.~Levitt.
\newblock Trees of cylinders and canonical splittings.
\newblock {\em Geometry \& Topology}, 15(2):977--1012, 2011.

\bibitem{GL17}
V.~Guirardel and G.~Levitt.
\newblock {\em {JSJ decompositions of groups}}, volume 395 of {\em Ast{\'e}risque}.
\newblock {Soci{\'e}t{\'e} Math{\'e}matique de France}, 2017.

\bibitem{wise:special}
F.~Haglund and D.~T. Wise.
\newblock Special cube complexes.
\newblock {\em Geometric and Functional Analysis}, 17(5):1551–1620, October 2007.

\bibitem{puder:commutator}
L.~Hanany, C.~Meiri, and D.~Puder.
\newblock Some orbits of free words that are determined by measures on finite groups.
\newblock {\em Journal of Algebra}, 555:305–324, August 2020.

\bibitem{hatcher}
A.~Hatcher.
\newblock {\em Algebraic topology}.
\newblock Cambridge University Press, 2002.

\bibitem{wise:res}
T.~Hsu and D.~T. Wise.
\newblock On linear and residual properties of graph products.
\newblock {\em Michigan Mathematical Journal}, 46(2), September 1999.

\bibitem{wise:cube}
T.~Hsu and D.~T Wise.
\newblock Cubulating graphs of free groups with cyclic edge groups.
\newblock {\em American journal of mathematics}, 132(5):1153--1188, 2010.

\bibitem{Kahn2012}
J.~Kahn and V.~Markovic.
\newblock Immersing almost geodesic surfaces in a closed hyperbolic three manifold.
\newblock {\em Annals of Mathematics}, 175(3):1127–1190, May 2012.

\bibitem{Kerckhoff1983}
S.~P. Kerckhoff.
\newblock The nielsen realization problem.
\newblock {\em The Annals of Mathematics}, 117(2):235, March 1983.

\bibitem{Lck1994}
W.~L\"{u}ck.
\newblock Approximating {$L^2$}-invariants by their finite-dimensional analogues.
\newblock {\em Geometric and Functional Analysis}, 4(4):455–481, July 1994.

\bibitem{Lck2002}
W.~L\"{u}ck.
\newblock {\em $L^2$-Invariants: Theory and Applications to Geometry and K-Theory}.
\newblock Springer Berlin Heidelberg, 2002.

\bibitem{Lck2013}
W.~L\"{u}ck.
\newblock Approximating {$L^2$}-invariants and homology growth.
\newblock {\em Geometric and Functional Analysis}, 23(2):622–663, March 2013.

\bibitem{Magee2021}
M.~Magee and D.~Puder.
\newblock Surface words are determined by word measures on groups.
\newblock {\em Israel Journal of Mathematics}, 241(2):749–774, March 2021.

\bibitem{min:vr}
A.~Minasyan.
\newblock Virtual retraction properties in groups.
\newblock {\em International Mathematics Research Notices}, 2021(17):13434–13477, November 2019.

\bibitem{neu:surfaces}
W.~D. Neumann.
\newblock Immersed and virtually embedded $\pi_1$–injective surfaces in graph manifolds.
\newblock {\em Algebraic \& Geometric Topology}, 1(1):411–426, July 2001.

\bibitem{Noskov1982}
G.~A. Noskov, V.~N. Remeslennikov, and V.~A. Roman’kov.
\newblock Infinite groups.
\newblock {\em Journal of Soviet Mathematics}, 18(5):669–735, 1982.

\bibitem{puder:primitive}
D.~Puder and O.~Parzanchevski.
\newblock Measure preserving words are primitive.
\newblock {\em Journal of the American Mathematical Society}, 28(1):63--97, 2015.

\bibitem{Scott1973}
G.~P. Scott.
\newblock Compact submanifolds of 3-manifolds.
\newblock {\em Journal of the London Mathematical Society}, s2-7(2):246–250, November 1973.

\bibitem{shalev:measure_conj}
A.~Shalev.
\newblock {\em Some Results and Problems in the Theory of Word Maps}, page 611–649.
\newblock Springer Berlin Heidelberg, 2013.

\bibitem{Shenitzer1955}
A.~Shenitzer.
\newblock Decomposition of a group with a single defining relation into a free product.
\newblock {\em Proceedings of the American Mathematical Society}, 6(2):273–279, April 1955.

\bibitem{Sun2015}
H.~Sun.
\newblock Virtual homological torsion of closed hyperbolic 3-manifolds.
\newblock {\em Journal of Differential Geometry}, 100(3), July 2015.

\bibitem{Swarup1986}
G.A. Swarup.
\newblock Decompositions of free groups.
\newblock {\em Journal of Pure and Applied Algebra}, 40:99–102, 1986.

\bibitem{wilkes:decompositions}
G.~Wilkes.
\newblock Relative cohomology theory for profinite groups.
\newblock {\em Journal of Pure and Applied Algebra}, 223(4):1617--1688, 2019.

\bibitem{wil:sep}
H.~Wilton.
\newblock Hall’s theorem for limit groups.
\newblock {\em Geometric and Functional Analysis}, 18(1):271–303, March 2008.

\bibitem{wil:one-ended}
H.~Wilton.
\newblock One-ended subgroups of graphs of free groups with cyclic edge groups.
\newblock {\em Geometry \& Topology}, 16(2):665–683, April 2012.

\bibitem{Wil18}
H.~Wilton.
\newblock Essential surfaces in graph pairs.
\newblock {\em J. Amer. Math. Soc.}, 31:893--919, 2018.

\bibitem{wilton:words}
H.~Wilton.
\newblock On the profinite rigidity of surface groups and surface words.
\newblock {\em Comptes Rendus. Mathématique}, 359(2):119–122, March 2021.

\bibitem{wz:decompositions}
H.~Wilton and P.~A. Zalesskii.
\newblock Profinite detection of 3-manifold decompositions.
\newblock {\em Compositio Mathematica}, 155(2):246–259, January 2019.

\bibitem{wise:graph-sep}
D.~T. Wise.
\newblock Subgroup separability of graphs of free groups with cyclic edge groups.
\newblock {\em The Quarterly Journal of Mathematics}, 50(1):107–129, March 2000.

\bibitem{Wise2021}
D.~T. Wise.
\newblock {\em The Structure of Groups with a Quasiconvex Hierarchy: (AMS-209)}.
\newblock Princeton University Press, May 2021.

\end{thebibliography}

\vspace{1cm}

(D. Ascari) \textsc{Department of Mathematics, University of the Basque Country, Barrio Sarriena, Leioa, 48940, Spain}

\emph{Email address:} \texttt{ascari.maths@gmail.com} \par \smallskip
\vspace{0.5cm}
(J. Fruchter) \textsc{Mathematisches Institut, Rheinische Friedrich-Wilhelms-Universit\"at Bonn, Endenicher Allee 60, 53115 Bonn, Germany}

\emph{Email address:} \texttt{fruchter@math.uni-bonn.de}
\end{document}